\documentclass{imsart}

\RequirePackage{amsthm,amsmath}
\RequirePackage[numbers]{natbib}
\RequirePackage[colorlinks,citecolor=blue,urlcolor=blue]{hyperref}
\RequirePackage{graphicx}
\usepackage{enumitem}
\usepackage{amssymb}
\usepackage{bm}
\usepackage{booktabs}
\usepackage{float}
\usepackage{xcolor}
\arxiv{2201.12839}

\startlocaldefs
\numberwithin{equation}{section}
\theoremstyle{plain}

\newtheorem{example}{Example}
\newtheorem{theorem}{Theorem}
\newtheorem{lemma}{Lemma}
\newtheorem{corollary}{Corollary}
\newtheorem{proposition}{Proposition}
\newtheorem{remark}{Remark}

\endlocaldefs

\definecolor{blue-violet}{rgb}{0.56, 0.0, 1.0}
\definecolor{beaublue}{rgb}{0.74, 0.83, 0.9}

\def\pico{\color{black}}
\def\picotwo{\color{black}}

\newcommand{\numbereqn}{\addtocounter{equation}{1}\tag{\theequation}} 
\newcommand\bigsubset[1][1.44]{%
	\mathrel{\vcenter{\hbox{\scalebox{#1}{$\subset$}}}}}

\begin{document}

\begin{frontmatter}
\title{Two-Step Mixed-Type Multivariate Bayesian Sparse Variable Selection with Shrinkage Priors}
\runtitle{Mixed-type Multivariate Bayesian Model with Shrinkage Priors}

\begin{aug}
\author{\fnms{Shao-Hsuan} \snm{Wang}\thanksref{t1}\ead[label=e1]{picowang@gmail.com}}

\address{Graduate Institute of Statistics, National Central University, \\
	 Taoyuan City 32001, Taiwan \\
\printead{e1}}

\author{\fnms{Ray} \snm{Bai}\thanksref{t2}\ead[label=e2]{rbai@mailbox.sc.edu}}

\address{Department of Statistics, University of South Carolina, \\
	Columbia, South Carolina 29225, USA \\
\printead{e2}}

\author{\fnms{Hsin-Hsiung} \snm{Huang}\thanksref{t3}\ead[label=e3]{Hsin-Hsiung.Huang@ucf.edu}}

\address{Department of Statistics and Data Science, University of Central Florida,  \\
	Orlando, Florida 32826, USA \\
\printead{e3}}

\thankstext{t1}{Supported by Ministry of Science and Technology grant 109-2118-M-008-005.}
\thankstext{t2}{Supported by National Science Foundation grant DMS-2015528.}
\thankstext{t3}{Supported by National Science Foundation grants DMS-1924792, DMS-2318925.}
\runauthor{Wang et al.}

\end{aug}

\begin{abstract}
\noindent  We introduce a Bayesian framework for mixed-type multivariate regression using continuous shrinkage priors. Our framework enables joint analysis of mixed continuous and discrete outcomes and facilitates variable selection from the $p$ covariates. Theoretical studies of Bayesian mixed-type multivariate response models have not been conducted previously and require more intricate arguments than the corresponding theory for univariate response models due to the correlations between the responses. In this paper, we investigate necessary and sufficient conditions for posterior contraction of our method when $p$ grows faster than sample size $n$. The existing literature on Bayesian high-dimensional asymptotics has focused only on cases where $p$ grows subexponentially with $n$. In contrast, we study the asymptotic regime where $p$ is allowed to grow \textit{exponentially} {\picotwo in} terms of $n$. We develop a novel two-step approach for variable selection which possesses the sure screening property and provably achieves posterior contraction even under exponential growth of $p$. We demonstrate the utility of our method through simulation studies and applications to real data, {\pico including a cancer genomics dataset where $n=174$ and $p=9183$}. The R code to implement our method is available at \url{https://github.com/raybai07/MtMBSP}. 
\end{abstract}

\begin{keyword}[class=MSC]
\kwd[Primary ]{62F15}
\kwd{62F12}
\kwd[; secondary ]{62J12}
\end{keyword}

\begin{keyword}
\kwd{global-local shrinkage prior}
\kwd{mixed responses}
\kwd{multivariate regression}
\kwd{posterior contraction}
\kwd{ultra-high dimensionality}
\kwd{variable screening}
\end{keyword}
\tableofcontents
\end{frontmatter}

\newpage

\section{Introduction} \label{intro}

\subsection{Motivation}

Mixed-type responses (i.e. continuous, categorical, and count data) are commonly encountered. For example, in developmental toxicology, a group of pregnant dams may be exposed to a test substance, and the outcomes of interest include the presence or absence of birth defects or fetal death (binary responses), fetal weight and dam's uterine weight (continuous responses), and the number of pups in each litter (a count response) \citep{DunsonJRSSB2000, HwangPennell2014,ReganCatalano1999}. In clinical trials, mixed-type outcomes are used to assess the efficacy and safety of treatments or interventions. Efficacy is often quantified by a continuous variable, while safety is a binary variable \citep{Stamey2013, ZhouPharmaceutical2006}. In manufacturing systems, performance metrics are frequently a mix of continuous and discrete responses \citep{chen2022bayesian, DengJin2015, KangJQT2018}. Finally, multivariate counts and categorical data often arise in image classification analysis \citep{zhou2014regularized} and in gene expressions of patients with different cancer types \citep{salman2018impact}. 

Analysts often wish to jointly model $q$ mixed-type responses given a set of $p$ covariates, since the responses are correlated and the $p$ predictors jointly affect all $q$ responses. However, specifying a joint likelihood function is difficult. Thus, correlations between the responses are usually accommodated by employing shared latent variables or shared random effects \citep{chen2022bayesian, DunsonJRSSB2000, McCulloch2008, WagnerTuchler2010} or by factorizing the joint likelihood as a product of conditional densities and a marginal density \citep{CoxWermuth1992,DengJin2015, FitzmauriceJASA1995}. In this paper, we adopt the first approach and use a latent vector to connect the $q$ mixed responses and the $p$ predictor variables. 

\subsection{Related works}

Bayesian approaches are appealing because they provide natural uncertainty quantification through the posterior distribution. However, it has historically been challenging to implement Bayesian models for mixed-type multivariate regression, especially if $p$ is moderate or large. Several Bayesian computational tools, such as integrated nested Laplace approximation (INLA) \citep{RueINLA2009} and Hamiltonian Monte Carlo (HMC) \citep{neal2011mcmc} are only appropriate for small $p$ \cite{BradleyBA2022}. Other Bayesian approaches require inefficient rejection samplers or Metropolis-Hastings (MH) algorithms that are difficult to tune \citep{BradleyBA2022}. To address these challenges, \cite{WagnerTuchler2010} proposed to use mixtures of normal distributions to approximate the residual errors for logistic and log-linear models with shared random effects. These approximations permit closed-form Gibbs sampling updates, but the subsequent algorithm is no longer exact. Recently, \cite{BradleyBA2022} introduced a hierarchical generalized transformation model where the mixed-typed responses are first transformed to continuous variables and then an exact Gibbs sampler is applied to the transformed responses. While this alleviates some computational difficulties, the regression coefficients may be less interpretable since the transformation changes the scale of the original responses. 

In this paper, we develop a new approach for mixed-type multivariate Bayesian regression based on P\'{o}lya–gamma (PG) augmented latent variables \citep{polson2013bayesian}. All of the conditional distributions in our model are tractable. Thus, our method can be implemented with an exact Gibbs sampling algorithm, and no transformations of the responses or approximations for residual error terms are required to fit it. While PG data augmentation (PGDA) \citep{he2024framework, polson2013bayesian} has been widely used for \textit{univariate} Bayesian regression models with discrete responses, PGDA has surprisingly been overlooked for Bayesian analysis of \textit{mixed-type multivariate} data. A strength of our proposed method is its simplicity and the fact that it allows for exact Markov chain Monte Carlo (MCMC) sampling from the posterior without transforming the data \cite{BradleyBA2022}.

In addition to its computational simplicity, our method also performs variable selection. Earlier works on mixed-type multivariate regression models (e.g. \cite{BradleyBA2022,DengJin2015,DunsonJRSSB2000,ZhouPharmaceutical2006}) did not consider variable selection from the $p$ covariates. However, if $p$ is moderate or large, then variable selection is often desirable. Recently, several variable selection methods for mixed-type multivariate regression models have been proposed. From the frequentist perspective, \cite{kang2020multivariate} introduced the Multivariate Regression of Mixed Responses (MRMR) method which applies an $\ell_1$ penalty on the regression coefficients. However, the MRMR method of \cite{kang2020multivariate} only provides point estimates and lacks inferential capabilities. \cite{chen2022bayesian} recently proposed the Bayesian sparse multivariate regression for mixed responses (BS-MRMR) model which uses spike-and-slab priors to identify significant variables and provides uncertainty quantification through the posterior distributions.

We pause here to highlight the main differences between our method and the BS-MRMR method of \cite{chen2022bayesian}. First, we use \textit{continuous} global-local shrinkage (GL) priors, whereas \cite{chen2022bayesian} use \textit{point-mass} spike-and-slab priors. In particular, \cite{chen2022bayesian} only consider cases where $p<n$, and in their numerical and real data examples, $p$ is no larger than 80. In contrast, we consider cases where $p \gg n$, {\pico including a real application where $p=9183$ and $n=174$}. In the past, global-local priors enjoyed significant computational advantages over point-mass spike-and-slab priors \citep{BhadraISR2020,BhadraDattaPolsonWillard2017, BhattacharyaPatiPillaiDunson2015}. For example, \cite{KonerWilliams2023} and \cite{KunduMitraGaskins2021} give numerical examples where it is very computationally challenging or even impractical to fit multivariate Bayesian regression models with point-mass spike-and-slab priors when $p$ is large. Recently, however, there has been some work that greatly improves the computational efficiency of point-mass spike-and-slab priors \citep{GriffinBiometrika2020, ZhouJRSSB2022}. Unfortunately, these promising new spike-and-slab approaches are restricted to univariate Gaussian linear regression and have yet to be explored for mixed-type multivariate regression models. In this work, we focus on GL priors primarily due to their analytical convenience, especially in simplifying posterior sampling.  

Motivated by applications from genomics where $p$ is often larger than $n$, we develop a new Bayesian variable selection method for mixed-type multivariate regression using GL shrinkage priors. We call our approach the Mixed-type Multivariate Bayesian Model with Shrinkage Priors (Mt-MBSP). GL priors are especially convenient because of their representation as Gaussian scale mixtures, which allows for conjugate updates. This leads to a straightforward \textit{one-step} estimation procedure.  Here, ``one-step'' refers to fitting the Mt-MBSP model with all $p$ predictors only once and then using the subsequent posterior distribution for inference.

Apart from their different prior distributions, a second major difference between our Mt-MBSP method and the BS-MRMR method of \cite{chen2022bayesian} is that we propose a \textit{two-stage} estimation method that first screens out many spurious variables in the first stage. In contrast, \cite{chen2022bayesian} do not perform a variable prescreening procedure before conducting final variable selection. {\pico It is well-established that when $p \gg n$, variable prescreening often improves the estimation and prediction accuracy of variable selection methods \citep{fan2008sure, narisetty2014bayesian}. For example, in the genomic data application in Section 9 of \cite{narisetty2014bayesian}, the authors employ prescreening to reduce the number of predictor genes from $p=22{,}575$ to $p \in \{ 200, 400 \}$ before fitting their Bayesian variable selection method.} 

Finally, \cite{chen2022bayesian} do not provide any theoretical justifications for BS-MRMR, whereas we theoretically characterize the asymptotic behavior of Mt-MBSP (both its one-step and two-step variants). Theory for Bayesian mixed-type multivariate regression models is very scarce. Due to the unknown correlations between multiple responses, theory for univariate response regression models does not trivially extend to the multivariate, mixed-type setting. Moreover, theoretical results established under the assumption of subexponential growth of $p$ do \textit{not} automatically hold in the setting of exponential growth, as we later show.

\subsection{Our contributions: methodological and theoretical novelties}

One of the contributions of this paper is to enhance theoretical understanding of Bayesian mixed response models. To the best of our knowledge, we present the first asymptotic results for Bayesian mixed-type multivariate regression.  First, we derive the posterior contraction rate for the one-step Mt-MBSP model when $p$ grows subexponentially with respect to $n$. We further establish that a subexponential growth rate of $p$ in $n$ is both necessary \textit{and} sufficient for posterior consistency. Although the Bayesian high-dimensional literature commonly makes this assumption \citep{BaiGhosh2018, chakraborty2020bayesian, narisetty2014bayesian}, it was not previously known whether subexponential growth of $p$ in $n$ is actually necessary or whether this can be relaxed (in either multivariate \textit{or} univariate Bayesian models). 

The assumption of subexponential growth may also be invalid in the era of ``big data,'' especially in applications involving high-throughput biological and genomic data \citep{fan2008sure}. This has prompted some researchers to study variable selection methods when $p$ can grow \textit{exponentially} fast with $n$ \citep{LahiriAOS2021}. However, the asymptotic regime where $p$ is allowed to diverge exponentially fast has not been studied before in the \textit{Bayesian} framework. 
In this paper, we derive a negative result showing that under standard assumptions on the design matrix and sparsity of the regression coefficients, exponential growth of $p$ with respect to $n$ can lead to posterior \textit{inconsistency} under the one-step Mt-MBSP method and other one-stage estimators.

To overcome the limitation of the one-step estimator and improve estimation when $p \gg n$, we introduce a novel \textit{two-step} estimation procedure. The first step of our two-stage approach screens out a large number of predictors. The second step then fits the Mt-MBSP model to \textit{only} the predictors that remain in the model after the first step. We prove that our proposed two-step algorithm has the sure screening property. Consequently, we show that our two-step estimator can consistently estimate the true regression coefficients \textit{even when} $p$ grows exponentially with $n$. Furthermore, even if $p$ does not grow exponentially fast, the two-step estimator still improves upon the one-step estimator with a sharper posterior contraction rate under the $p \gg n$ regime.

{\pico Two-stage estimators are frequently used in statistics. In this general class of estimators, a ``pilot'' estimate is first obtained, and then based on this pilot estimate, a second estimation procedure is done to yield an improved estimator. For example, a two-stage estimator may achieve an optimal or accelerated convergence rate or attain an oracle property that the one-stage estimator is incapable of achieving \citep{BelitserGhosalvanZanten2012, ChakrabortyGhosal2021, FanZhang1999,  TangBanerjeeMichailidis2011, YooGhosal2019,ZouJASA2006}. In the present context, our two-step Mt-MBSP corrects the potential inconsistency of one-step Mt-MBSP and achieves an accelerated convergence rate when $p > n$.}

Our main contributions can be summarized as follows:

\begin{enumerate}
	
\item Methodologically, we introduce the Mt-MBSP method for Bayesian estimation and variable selection in \textit{mixed-type} multivariate regression models even when $p \gg n$. Our method accommodates correlations between the mixed-type (discrete and continuous) response variables trough latent variables and facilitates variable selection from the $p$ covariates using continuous shrinkage priors. By employing PGDA, Mt-MBSP can also be fit with an exact Gibbs sampler.  	
\item Theoretically, we derive both necessary \textit{and} sufficient conditions for posterior contraction under the one-step Mt-MBSP model when $p \gg n$.  We shed new light on asymptotics for Bayesian high-dimensional regression models by formally establishing that subexponential growth of $p$ with respect to $n$ is a \textit{necessary} condition for the posterior consistency of one-stage estimators. The overwhelming majority of asymptotic theory for Bayesian high-dimensional regression models focuses only on sufficient conditions for consistency. Our necessary conditions are useful in that they imply conditions which, if violated, lead to posterior \textit{inconsistency}. 
	\item We introduce a two-step estimation procedure that screens out spurious variables \textit{and} that gives consistent estimation when $p$ grows \textit{exponentially} with respect to $n$. The proposed Bayesian prescreening procedure using GL priors is a new idea and can be applied to either multivariate or univariate regression models. We provide theoretical support for the two-step procedure by establishing its sure screening property and its posterior contraction rate when $p$ is allowed to grow exponentially fast with $n$. This ultrahigh-dimensional regime is different from the existing literature on Bayesian high-dimensional asymptotics, which has assumed \textit{subexponential} growth of $p$ with $n$.
\end{enumerate}
The rest of this paper is structured as follows. In Section \ref{method}, we introduce the Mt-MBSP method and discuss posterior sampling. Section \ref{theory} presents asymptotic results for the one-step Mt-MBSP estimator. In Section \ref{computation}, we introduce a novel two-step algorithm which we study theoretically. Section \ref{experiments} presents simulations and examples on real datasets. Section \ref{Discussion} concludes the paper. All of the proofs are deferred to Appendix \ref{S4:Proofs}.

\subsection{Notation}

For two sequences of positive real numbers $a_n$ and $b_n$, $a_n = o(b_n)$ means $\lim_{n \rightarrow \infty} a_n / b_n = 0$, while $a_n = O(b_n)$ means $|a_n / b_ n| \leq M$ for some positive real number $M$ independent of $n$. The $\ell_2$ norm of a vector $\mathbf{v}$ is denoted by $\lVert \mathbf{v} \rVert$, while the Frobenius norm of a matrix $\mathbf{A}$ is denoted by $\lVert \mathbf{A} \rVert_F$. For a {\picotwo square} matrix $\mathbf{C}$, we denote its {\picotwo trace by $\text{tr}(\mathbf{C})$} and {\picotwo its} minimum and maximum eigenvalues by $\lambda_{\min}(\mathbf{C})$ and $\lambda_{\max}(\mathbf{C})$ respectively. For a set $\mathcal{S}$, we denote its cardinality by $|\mathcal{S}|$ and for a subset $\mathcal{T} \subset \mathcal{S}$, $\mathcal{T}^{c}$ means $\mathcal{T}^{c} = \mathcal{S} \setminus \mathcal{T}$. Finally, a positive measurable function $L$ defined over $(c, \infty)$, for some $c \geq 0$, is said to be slowly varying if for every fixed $\alpha > 0$, $L(\alpha x)/L(x) \rightarrow 1$ as $x \rightarrow \infty$. 

\section{The Mt-MBSP Method}\label{method}
\setcounter{equation}{0}

\subsection{Joint modeling of mixed-type responses}
\label{MBSP-extension}

Let $\mathbf{Y}_n=(\mathbf{y}_1,\dots,\mathbf{y}_n)^\top$ be an
$n\times q$ response matrix of $n$ samples, where $\mathbf{y}_i =(y_{i1},\dots,y_{iq})^\top, i = 1,\dots,n$, are $q \times 1$ response vectors. Suppose that we have a mixture of continuous and discrete responses in $\mathbf{Y}_n$. 
Let $\bm{\theta}_i =(\theta_{i1},\dots,\theta_{iq})^\top$ be a $q \times 1$ natural parameter vector whose elements are related to a linear combination of covariates through a specified link function.

In our joint modeling approach, we model the \textit{continuous} responses in $\mathbf{y}_i$ using Gaussian densities, i.e. if the $k$th component {\picotwo $y_{ik}$} of each response vector $\mathbf{y}_i, i = 1, \ldots, n$, is continuous, then 
\begin{equation} \label{gaussianresponse}
	p_k(y_{ik} \mid \theta_{ik} ) = \frac{1}{(2\pi)^{1/2}} \exp\left( - \frac{(y_{ik} - \theta_{ik})^2}{2} \right).
\end{equation} 
Meanwhile, we model the \textit{discrete} responses in $\mathbf{y}_i$ using the following density function, i.e.  if the $\ell$th component of $\mathbf{y}_i$, $\ell \neq k$, is discrete, then 
\begin{eqnarray} 
	p_{\ell}(y_{i\ell}\mid \theta_{i\ell})= C \frac{\{\exp(\theta_{i\ell})\}^{f^{(1)}_{i\ell}}}{
		\{1+\exp(\theta_{i\ell})\}^{f^{(2)}_{i\ell}}}, ~~i=1,\dots,n,~~\ell=1,\dots,q,\label{polson_model}
\end{eqnarray}
where  $\{f^{(1)}_{i\ell}\}_{\ell=1}^q$ 
and $\{f^{(2)}_{i\ell}\}^q_{\ell=1}$ 
are values which may depend on $\mathbf{Y}_n$, and  $C$ is the normalizing constant. 
The family of densities
\eqref{polson_model} is quite broad and can model a variety of discrete data types. Some examples of response variables belonging to this family are listed below.

\begin{example}[Bernoulli] 
	Let $f^{(1)}=y\in\{0,1\}$ and $f^{(2)}=1$. Then
	\[
	P(Y=y\mid \theta) =
	\left(\frac{\exp(\theta)}{1+\exp(\theta)}\right)^y
	\left(\frac{1}{1+\exp(\theta)}\right)^{1-y}= 
	\frac{\{\exp(\theta)\}^y}{1+\exp(\theta)}.
	\]
\end{example}

\begin{example}[binomial]
	Let $f^{(1)}=y\in\{0,1,\dots,M\}$ and $f^{(2)}=M$. Then 
	\[
	P(Y=y\mid \theta,~~M) =
	\begin{pmatrix}
		M\\y
	\end{pmatrix} 
	p^y(1-p)^{M-y}  \propto 
	\frac{\{\exp(\theta)\}^{y}}{\{1+\exp(\theta)\}^{M}},
	\]
	where $p=\exp(\theta)/\{1+\exp(\theta)\}$. 
\end{example}

\begin{example}[multinomial]
	For multinomial response variables with $L > 2$ classes and $M$ trials, we can reexpress the multinomial density as
	\[ P(Y_1 = y_1, \ldots, Y_{L-1} = y_{L-1} \mid \theta_1, \ldots, \theta_{L-1}, M ) \propto \prod_{\ell=1}^{L-1} \frac{ \exp(\theta_\ell)^{y_\ell}}{ [1+\exp(\theta_\ell)]^{M_\ell}},
	\]
	where $M_1 = M$ and $M_\ell = M - \sum_{j < \ell} y_j$ for $2 \leq \ell \leq L-1$ {\rm \citep{LindermanNIPS2015}}. That is, for class $\ell \in \{1, \ldots, L-1 \}$, 
	\[ P(Y_{\ell} = y_\ell \mid \theta_\ell ) \propto \frac{\exp(\theta_\ell)^{y_\ell}}{\{ 1+\exp(\theta_\ell)\}^{M_\ell}}, \]
	i.e. $f_{\ell}^{(1)} = y_\ell \in \{ 0, 1, \ldots, M \}$ and $f_{\ell}^{(2)} = M_\ell$ in \eqref{polson_model}.
\end{example}

\begin{example}[negative binomial]
	Let $f^{(1)}={\picotwo r}$ and $f^{(2)}=y+r$, {\picotwo where $0 < r < \infty$}. Then
	\[
	P(Y=y\mid \theta, r) =
	{\picotwo \frac{\Gamma(y+r)}{y!~\Gamma(r)}
	p^r(1-p)^{y}}   \propto 
	\frac{\{\exp(\theta)\}^{\picotwo r}}{\{1+\exp(\theta)\}^{y+r}},
	\]
	{\picotwo where $\Gamma(\cdot)$ denotes the gamma function} and $p=\exp(\theta)/\{1+\exp(\theta)\}$.
\end{example}

\begin{remark}
	The Poisson distribution with rate parameter $\theta$ is a limiting case of the negative binomial (NB) distribution {\picotwo as $r \rightarrow \infty$}. To model equidispersed count data with \eqref{polson_model}, we can set the dispersion parameter $r$ to be large. For large $r$, the NB is practically indistinguishable from a {\picotwo Poisson {\rm \citep{he2024framework,jackman2009bayesian}}.} The NB is also more flexible in its ability to model overdispersed counts.
\end{remark}

\begin{remark}
{\picotwo For the unknown dispersion parameter $r$ in an NB random variable $m\sim{\rm NB}(r,p)$, we consider a gamma prior $\mathcal{G}(c_1,c_2)$. This allows for closed-form updates of $r$ in the Gibbs sampler. Based on Theorem 1 of {\rm\cite{zhou2013negative}}, 
\[{\rm CRT}(\ell\mid m, r)
{\rm NB}(m; r, p)
=
{\rm SumLog}(m\mid \ell, p)
{\rm Poisson}(\ell\mid -r\log(1-p)), 
\]
where the definitions of the Chinese restaurant table (CRT) and the SumLog distributions can be found in Theorem 1 of {\rm\cite{zhou2013negative}}. This implies that an NB random variable can be 
expressed as the following mixture model,
\[
{\rm NB}(m; r, p)
=
\sum_{\ell=0}^\infty 
{\rm SumLog}(m\mid \ell, p)
{\rm Poisson}(\ell\mid -r\log(1-p)), 
\]
which implies that 
\[
m \sim {\rm SumLog}(m \mid \ell, p),~~ \ell\sim{\rm Poisson} (\ell \mid -r\log(1-p)). 
\]
The conditional distribution of $r$ given the other parameters is then $\mathcal{G}(c_1+\ell, c_2-\log(1-p))$, since the gamma distribution is conjugate with the Poisson distribution.}
\end{remark}

The parametric family \eqref{polson_model} is especially convenient, as it facilitates closed-form updates for Gibbs sampling. In particular, \eqref{polson_model} enables a PGDA approach 
to easily sample from the posterior distributions of the parameters. A random variable $\omega$ is said to follow a P\'{o}lya–gamma distribution $\mathcal{PG}(a,b)$ with parameters $a>0$ and $b \in \mathbb{R}$ if 
\begin{align*}
	\omega = \frac{1}{2 \pi^2} \sum_{k=1}^{\infty} \frac{g_k}{\left( k- \frac{1}{2} \right)^2 + \left( \frac{b}{2 \pi} \right)^2},\mbox{~~and~~} g_k \sim \mathcal{G}(a,1).
\end{align*}
Let  ${\cal PG}(\omega\mid f^{(2)}, 0)$ denote the density function for ${\cal PG}(f^{(2)}, 0)$, and then we have the following identity from \cite{polson2013bayesian}:
{\pico \begin{equation*}
	p(y\mid \theta ) \propto \frac{\{\exp(\theta)\}^{f^{(1)}}}{\{1+\exp(\theta)\}^{f^{(2)}} }=
	{\picotwo 2^{-f^{(2)}}}e^{\kappa \theta} 
	\int^\infty_0 e^{-\omega\theta^2/2}\ \mathcal{PG}(\omega\mid f^{(2)},0) d\omega,
\end{equation*}}
which implies that {\picotwo
\begin{eqnarray}
p(y, \omega\mid \theta ) 
=
2^{-f^{(2)}}e^{\kappa \theta} 
	e^{-\omega\theta^2/2}\ \mathcal{PG}(\omega\mid f^{(2)},0), \label{wyz}
\end{eqnarray}
}where {\picotwo $\kappa= f^{(1)}- 0.5 f^{(2)}$}. 
Thus, conditional on $\omega$, the likelihood contribution in $\theta$ is
\begin{eqnarray} 
	p(\theta\mid \omega,~y) \propto p(\theta) e^{\kappa\theta}e^{-\omega \theta^2/2}
	~\propto~ p(\theta)\exp\left\{
	-\frac{\omega}{2}\left(\theta-\frac{\kappa}{\omega}\right)^2
	\right\}. \label{p1}
\end{eqnarray} 
On the other hand, if $y$ is \textit{continuous} and has the density \eqref{gaussianresponse}, then
\begin{eqnarray} 
	p(\theta\mid y) \propto  p(\theta)p(y\mid \theta)\propto p(\theta) \exp\left\{
	-\frac{1}{2}\left(\theta-y\right)^2
	\right\}. \label{p2}
\end{eqnarray} 
Let {\picotwo
\begin{equation} \label{p4}
	\left\{ \begin{array}{ll}
		z=y\mbox{~~and~~}\omega=1&\mbox{~~~if $y$ is continuous as in \eqref{gaussianresponse};}\\[.1in]
		 z=\frac{\kappa}{\omega}\mbox{~~and~~}\omega\sim {\cal PG}( f^{(2)},0)&
		\mbox{~~~if $y$ is discrete as in \eqref{polson_model}.}
	\end{array} \right.
\end{equation}}

\noindent From \eqref{p1}-\eqref{p2}, we can \textit{unify} continuous and discrete outcomes with the following conditional density of $\theta$, {\picotwo
\begin{eqnarray} 
p(\theta\mid \omega, y)=p(\theta\mid \omega, z)\propto p(\theta) \exp\left\{
	-\frac{\omega}{2}\left(z-\theta\right)^2
	\right\}. \label{p3}
\end{eqnarray} 
It is noted that $y$ is determined by $\omega$ and $z$. }
Our construction \eqref{p4}-\eqref{p3} allows for simplicity of Bayesian computation. In particular, if $p(\theta)$ in \eqref{p3} is a Gaussian density, then the conditional density $p(\theta\mid\omega,z)$ is \textit{also} Gaussian. This implies that modeling the $q$ responses through a latent Gaussian distribution leads to tractable full conditionals.

\subsection{Prior specification} \label{PGDA}
Having specified the form for the $q$ responses in $\mathbf{y}_i$ in \eqref{gaussianresponse}-\eqref{polson_model}, we now jointly model them given a set of $p$ covariates $\mathbf{x}_i$, where $\mathbf{x}_i = (x_{i1}, \ldots x_{ip})^\top$. Let $\mathbf{B}$ denote a $p \times q$ unknown matrix of regression coefficients. We model the natural parameter $\bm{\theta}_{i} {\picotwo = (\theta_{i1}, \ldots, \theta_{iq})^\top}$ in \eqref{gaussianresponse}-\eqref{polson_model} with covariates $\mathbf{x}_i$ as follows: 
\begin{eqnarray}
	\bm{\theta}_{i} = \mathbf{B}^\top \mathbf{x}_i+ \mathbf{u}_i,\quad \mathbf{u}_i \sim \mathcal{N}_q (\bm{0}, \bm{\Sigma}),~~ i = 1, \ldots, n, \label{blatent}
\end{eqnarray}
where $\mathbf{u}_i$ is a $q$-dimensional random effect vector with covariance $\bm{\Sigma}$ that models the correlations among the $q$ responses. 

To apply the Bayesian approach, we endow the rows $\mathbf{b}_j$ of $\mathbf{B}$ with the following independent GL priors:
\begin{equation} \label{Mt-MBSP}
	\mathbf{b}_j \mid \xi_j \sim \mathcal{N}_q(\bm{0}, \tau \xi_j \mathbf{I}_q), \quad \xi_j \sim p(\xi_j), \quad j = 1, \ldots, p,
\end{equation} 
where $\tau>0$ is a \textit{global} shrinkage parameter which shrinks all elements in $\mathbf{b}_j$ to zero. Meanwhile, the $\xi_j$ is a \textit{local} shrinkage parameter that controls the  individual shrinkage in $\mathbf{b}_j$ through a combination of heavy mass around zero and heavy tails. We assume that $p(\xi_j)$ is a polynomial-tailed density,
\begin{equation}
	p(\xi_j) = K\xi^{-v-1}_j L(\xi_j), 
	\label{global_local_prior}
\end{equation}
where $K > 0$ is the constant of proportionality, $v$ is a positive real number, and $L$ is a slowly varying function valued over $(0, \infty)$. If $\lVert \mathbf{b}_j \rVert$ is zero or close to zero, then the combination of global shrinkage and heavy mass near zero in $p(\xi_j)$ ensures that the entries are in $\mathbf{b}_j$ are heavily shrunk to zero. On the other hand, if $\lVert \mathbf{b}_j \rVert$ is large, then the heavy tails of $p(\xi_j)$ counteract the global shrinkage from $\tau$ and prevent overshrinkage of $\mathbf{b}_j$. This \textit{adaptive} shrinkage allows Mt-MBSP to identify the nonzero signals.

{\picotwo The prior \eqref{Mt-MBSP} shrinks many rows in $\mathbf{B}$ towards zero, and thus, Mt-MBSP produces a (nearly) row-sparse estimate of $\mathbf{B}$. Row sparsity does not preclude the nonzero rows of $\mathbf{B}$ from also containing null entries, as long as these rows contain at least one nonzero entry. Row sparsity is a very common assumption in the literature on variable selection for multivariate regression \cite{ BaiGhosh2018,  ChenHuang2012, ChunKeles2010, GohDeyChen2017,KonerWilliams2023,LiNanZhu2015, LiquetMengersonPettittSutton2017, ZHANG2022104835}. Researchers possibly make this assumption for interpretability reasons, because if the $j$th row of the estimate for $\mathbf{B}$ contains all zeros or all negligible entries, then we can conclude that the $j$th covariate is not significantly associated with \emph{any} of the $q$ response variables. In Section \ref{Discussion}, we discuss how to extend Mt-MBSP to more general sparsity structures.}

\begin{table}[t!] 
	\caption{Polynomial-tailed priors, their respective hyperpriors $p(\xi_j)$ up to normalizing constant $C$, and the slowly-varying component $L(\xi_j)$.}\label{table:priors}
	\centering
	\resizebox{\textwidth}{!}{
		\begin{tabular}{lcc}
			Prior & $p(\xi_j)/C$ & $L(\xi_j)$  \\
			\toprule
			Student's $t$ & $\xi_j^{-a-1}\exp(-{a}/{\xi_j})$ & $\exp \left( -a/\xi_j \right)$ \\
			TPBN & $\xi_j^{u-1}(1+\xi_j)^{-a-u}$ & $\left\{ \xi_j/(1+\xi_j) \right\}^{a+u}$ \\
			Horseshoe & $\xi_j^{-1/2}(1+\xi_i)^{-1}$ & $\xi_j^{a+1/2} / (1+\xi_j)$  \\
			NEG & $ \left(1+\xi_j\right)^{-1-a} $ & $\left\{ \xi_j/(1+\xi_j) \right\}^{a+1}$ \\
			GDP & $\int_0^\infty ({\lambda^2}/{2}) \exp (- {\lambda^2\xi_j}/{2} )\lambda^{2a-1}\exp(-\eta\lambda)d\lambda$ & $\int_{0}^{\infty} t^{a} \exp(-t - \eta (2t/\xi_j)^{1/2}) dt$ \\
			Horseshoe+ & $ \xi_j^{-1/2}(\xi_j-1)^{-1}\log(\xi_j)$ & $\xi_j^{a+1/2}(\xi_j - 1)^{-1} \log (\xi_j)$  \\
			\bottomrule
		\end{tabular} }
\end{table}

The polynomial-tailed prior density function \eqref{global_local_prior} encompasses many well-known shrinkage prior distributions, including the Student's $t$, three parameter beta normal (TPBN) \citep{ArmaganClydeDunson2011}, horseshoe \citep{CarvahoPolsonScott2010}, normal-exponential-gamma (NEG) \citep{GriffinBrown2013}, generalized double Pareto (GDP) \citep{ArmaganDunsonLee2013}, and horseshoe+ \citep{BhadraDattaPolsonWillard2017}. Table \ref{table:priors} lists the $p(\xi_j)$'s and $L(\xi_j)$'s in \eqref{global_local_prior} for these priors. The TPBN family includes the horseshoe ($u=0.5, a = 0.5$) and NEG ($u=1, a >0$) as special cases.

A similar prior \eqref{Mt-MBSP}-\eqref{global_local_prior} was used by \cite{BaiGhosh2018} for the regression coefficients in the Gaussian multivariate linear regression model $\mathbf{y} = \mathbf{B}^\top \mathbf{x} + \bm{\varepsilon}, \bm{\varepsilon} \sim \mathcal{N}_q(\bm{0}, \bm{\Sigma})$. Different from \cite{BaiGhosh2018}, the mixed-type response model introduced in this paper allows for binary, categorical, or discrete outcomes by employing latent vectors \eqref{jointzomega} to link the responses and the covariates. In contrast, the method of \cite{BaiGhosh2018} can only
be used if \textit{all} of the response variables are continuous.
Moreover, \cite{BaiGhosh2018} make a very restrictive assumption that the covariance matrix $\bm{\Sigma}$ is known. In this article, we treat $\bm{\Sigma}$ as unknown with a prior.

To complete our prior specification, we endow the covariance matrix $\bm{\Sigma}$ for the random effect $\mathbf{u}_i \sim \mathcal{N}_q(\bm{0}, \bm{\Sigma})$ in \eqref{blatent} with an inverse Wishart prior,
\begin{equation} \label{IWpriorSigma}
	\bm{\Sigma} \sim {\cal IW}(d_1,d_2 \mathbf{I}_q), 
\end{equation}
where $d_1 > q-1$ is the degrees of freedom and $d_2 > 0$ is a scale parameter.


\subsection{Posterior sampling and variable selection} \label{BVS}

Thanks to the data augmentation scheme with latent vectors $(\mathbf{z}_i, \bm{\omega}_i)$, it is very convenient to derive a simple Gibbs sampler to implement our method. In particular, if the TPBN prior \citep{ArmaganClydeDunson2011} is used for \eqref{Mt-MBSP}-\eqref{global_local_prior}, then all conditional distributions under \eqref{Mt-MBSP} have a closed form. The TPBN prior can be reparameterized as
\begin{equation} \label{TPBN}
	\mathbf{b}_j \mid \nu_j \sim {\cal{N}}_q (0, \nu_j I_q),~~ \nu_j \mid \eta_j \sim {\cal G}(u, \eta_j),~~ \eta_j\sim {\cal G}(a, \tau),~~ j = 1, \ldots, p,
\end{equation}
and includes the popular horseshoe prior \citep{CarvahoPolsonScott2010} when $u=a=0.5$. For concreteness, we discuss the Gibbs sampler with TPBN priors \eqref{TPBN} for the local scale parameters. However, it is straightforward to derive the Gibbs sampler for the other polynomial-tailed priors listed in Table \ref{table:priors} by similarly reparameterizing them as in \eqref{TPBN}. Our theoretical results in Sections \ref{theory} and \ref{computation} are also derived under the general class of polynomial-tailed densities \eqref{global_local_prior}.

We first consider the Gibbs sampling algorithm for our model with all $p$ predictors. We refer to this as the \textit{one-step} Gibbs sampler. In Section \ref{computation}, we will introduce and theoretically analyze a \textit{two-step} algorithm suitable for large $p$. The Gibbs sampler for the one-step Mt-MBSP model with the TPBN hyperprior \eqref{TPBN} is given in Appendix \ref{S1:Gibbs}. All of the conditional distributions in our Gibbs sampler are available in closed form, obviating the need to use approximations for the residual errors \citep{WagnerTuchler2010} or transformations of the responses \citep{BradleyBA2022}. 

{\picotwo We provide some intuition about the Gibbs sampler in Appendix \ref{S1:Gibbs}. Based on \eqref{p4}-\eqref{p3}, there exist corresponding augmented data vectors $\mathbf{z}_i = (z_{i1}, \ldots, z_{iq})^\top$ and $\bm{\omega}_i = (\omega_{i1}, \ldots, \omega_{iq})^\top$ for each $\mathbf{y}_i=(y_{i1}, \ldots, y_{iq})^\top$. From \eqref{p1}, we have $p(y_{ik}, \omega_{ik}, \theta_{ik})=p(y_{ik},\omega_{ik}, z_{ik}, \theta_{ik})$ since $z_{ik}$ is determined by $y_{ik}$ and $\omega_{ik}$.  Let $\mathbf{x}_i = (x_{i1}, \ldots, x_{ip})^\top$ be the vector of $p$ covariates for the $i$th observation. Meanwhile, let $\mathbf{u}_i \sim \mathcal{N}_q(\bm{0},~\bm{\Sigma})$ and $\bm{\Omega}_i=\textrm{diag}(\omega_{i1},\ldots,\omega_{iq})$. Since $\mathbf{u}_i$ in \eqref{blatent} captures the correlations among the $q$ responses in $\mathbf{y}_i$, we obtain 
 from \eqref{p3}-\eqref{blatent} that
 the conditional density 
$(\mathbf{z}_i,~\bm{\omega}_i)\mid \mathbf{x}_i,~\mathbf{u}_i$ is
\begin{equation}  \label{jointzomega}
p( \mathbf{z}_i,~\bm{\omega}_i \mid \mathbf{x}_i, \mathbf{u}_i) = 
p_0(\mathbf z_i,~\bm \omega_i)
  \exp\left\{-\frac12
  (\mathbf z_i-\mathbf B^\top \mathbf x_i-\mathbf u_i)^\top\mathbf \Omega_i
  (\mathbf z_i-\mathbf B^\top \mathbf x_i-\mathbf u_i)
  \right\}, 
\end{equation}

\noindent where  $p_0(\mathbf z_i,~\bm \omega_i)$ is 
a baseline function not depending on $\mathbf{B}$ or $\mathbf{u}_i$.
While the conditional density $p(\mathbf{z}_i,~\bm{\omega}_i \mid \mathbf{x}_i, \mathbf{u}_i)$ may not be Gaussian, we have a Gaussian kernel for the term that depends on $\mathbf{B}$ and $\mathbf{u}_i$. Therefore, the conditional distributions for $\mathbf{u}_i, i = 1, \ldots, n$, are Gaussian. Since we use a conditionally conjugate Gaussian scale mixture prior \eqref{Mt-MBSP}-\eqref{global_local_prior} for $\mathbf{B}$ and an inverse-Wishart prior \eqref{IWpriorSigma} for $\bm{\Sigma}$, the conditional distributions for $\mathbf{B}$ and $\bm{\Sigma}$ are both available in closed form. By \eqref{p4}, the conditional distributions of the latent variables $\omega_{ik}, k = 1, \ldots, q$, also have closed forms. Note that in the Gibbs sampler in Appendix \ref{S1:Gibbs}, the augmented vectors $\mathbf{z}_i, i=1, \ldots, n$, are not sampled but are updated as in \eqref{p4} whenever the vectors $\bm{\omega}_i, i=1, \ldots, n$, are updated.}

{\pico A desirable property of MCMC algorithms is geometric ergodicity. If a Markov chain is geometrically ergodic, then a Markov chain central limit theorem exists, and this allows practitioners to compute asymptotically valid standard errors for Markov chain-based estimates of posterior quantities \cite{RobertsRosenthal1998}. In the literature, there has been significant progress made on establishing geometric ergodicity for both P\'{o}lya-gamma-based Gibbs samplers \cite{ChoiHobert2013, WangRoy2018} and Gibbs sampling algorithms for regression models under Gaussian scale-mixture shrinkage priors \cite{BhattacharyaKharePal2022,PalKhare2014, PalKhareHobert2017}. Since our MCMC algorithm in Appendix \ref{S1:Gibbs} employs P\'{o}lya-gamma data augmentation and Gaussian scale-mixtures, we expect that our algorithm is also geometrically ergodic under suitable conditions. However, a rigorous theoretical analysis on the convergence of the MCMC chain is needed to confirm this and should be explored in future work. In Appendix \ref{S5:Diagnostics}, we report MCMC diagnostics for several of our simulated examples, demonstrating adequate convergence our MCMC algorithm and quality of the MCMC samples.}

After estimating the $(j,k)$th entries $b_{jk}$'s in $\mathbf{B}$ using the Gibbs sampler, we can use the 0.025 and 0.975 quantiles, $q_{0.025}(b_{jk})$ and
$q_{0.975}(b_{jk})$, 
for each $j = 1, \ldots, p, k = 1, \ldots, q$, to select the set of variables $\mathcal{A}_0 \subset \{1, \ldots, p \}$ that are significantly associated with the $q$ responses as
\begin{equation}
	{\cal A}_0= 
	\left\{
	j \mid 
	q_{0.025}(b_{jk}) > 0 \textrm{ or }  q_{0.975}(b_{jk}) < 0 \textrm{ for some } k =  1,\dots,q \right\}. \label{a0}
\end{equation}
That is, if at least one of the entries in the $j$th row of $\mathbf{B}$ has a 95\% credible interval that does \textit{not} contain zero, then we select the $j$th covariate {\picotwo and conclude that it is significantly associated with at least one of the $q$ responses. In particular, if the credible interval for $b_{jk}$ does not contain zero, then we conclude that the $j$th covariate is significantly associated with the $k$th response.}


\section{Asymptotic theory for one-step Mt-MBSP} \label{theory}

\setcounter{equation}{0}

\subsection{Posterior contraction rate of one-step Mt-MBSP} \label{onestepcontraction}

Theoretical analysis of Bayesian \textit{mixed-type} multivariate regression models has been missing in the literature. We address this gap by considering the high-dimensional regime where the number of covariates $p$ diverges as $n \rightarrow \infty$ at rates which depend on $n$. Accordingly, we use $p_n$ to denote the dimensions of our parameter of interest $\mathbf{B}_n = \mathbf{B}_{p_n \times q}$. We first characterize the asymptotic behavior of the posterior for $\mathbf{B}_n$ under the \textit{one-step} Mt-MBSP method (i.e. Mt-MBSP fitted to all $p_n$ variables without variable screening). In Section \ref{computation}, we introduce a novel two-step algorithm that can overcome some limitations of the one-step estimator.

It bears mentioning that establishing new theoretical results for mixed-type multivariate regression does \textit{not} follow straightforwardly from existing results in univariate regression or Gaussian multivariate regression (e.g. \citep{ BaiGhosh2018,song2023nearly}). {\picotwo Our mixed-type regression framework involves latent vectors in \eqref{blatent} }that are correlated through a shared random effect with an unknown covariance matrix $\bm{\Sigma}$, which we endow with a prior \eqref{IWpriorSigma}. Our proofs involve carefully bounding the eigenvalues of the induced marginal covariance matrices of these latent vectors. In contrast,  latent vectors are not required for univariate regression models or for multivariate Gaussian regression models. Furthermore, most of the existing literature only considers a subexponential growth of the predictors with sample size, i.e. $\log p_n = o(n)$, whereas we also allow for exponential growth $\log p_n = O(n^{\alpha}), \alpha \geq 1$.

Before studying the general regime of $\log p_n = O(n^{\alpha}), \alpha \geq 1$, we first consider the special case of $\log p_n = o(n)$. Throughout this section, we assume we have a mixed-type response model where \eqref{gaussianresponse}-\eqref{polson_model} and \eqref{blatent} hold with true parameters $\mathbf{B}_0$ and $\bm{\Sigma}_0$. That is, for $k = 1, \ldots, q$, $p_k(y_{ik} \mid \theta_{i0}) \propto \exp \left\{ -(y_{ik}- \theta_{ik0})^2/2 \right\}$ if $y_{ik}$ is continuous and $p_k(y_{ik}\mid \theta_{i0}) \propto \{\exp(\theta_{ik0})\}^{f^{(1)}_{ik}}/
\{1+\exp(\theta_{ik0})\}^{f^{(2)}_{ik}}$ if $y_{ik}$ is discrete. Meanwhile,
$\bm{\theta}_{i0} = \mathbf{B}_0^\top \mathbf{x}_i+ \mathbf{u}_{i}$, $\mathbf{u}_i \sim \mathcal{N}_q(\bm{0},\bm{\Sigma}_0)$, for each $i=1,\dots,n$. 
We make the following assumptions.
\begin{enumerate}[label=(A\arabic*)]
	\item $n \ll p_n$ and $\log p_n=o(n)$. 
	\item Let $S_0 \subset \{1, \dots, p_n \}$ denote the set of indices of the rows in $\mathbf{B}_0$ with at least one nonzero entry. Then $|S_0|=s_0$ satisfies $1\leq s_0$ and ${ s_0=o(n/ \log p_n)}$.
 \item {\pico For $\mathbf{X}_n = (\mathbf{x}_1, \ldots, \mathbf{x}_n)^\top$, and for an arbitrary set $S \subset \{1, \ldots, p_n \}$, let $\mathbf{X}^S$ denote the submatrix of $\mathbf{X}_n$ containing the columns with indices in $S$. For any $S$ where $|S| = o(n / \log p_n)$, there exists a sequence $\varepsilon_n$ satisfying $\varepsilon_n= n^{-k_0}$ such that 
 $n^{k_0+d}=o(n/\log p_n)$ for some $k_0>0,~d>0$  and   
\[
\varepsilon_n< 
\lambda_{\min}\left\{\frac1n(\mathbf{X}^{S})^\top\mathbf{X}^{S}\right\} <\lambda_{\max}\left\{\frac1n(\mathbf{X}^{S})^\top\mathbf{X}^{S}\right\}<\varepsilon^{-1}_n
\mbox{~~as~~}n\rightarrow\infty.
\]
}
	\item There exists a constant $0 < k_1 < \infty$ such that 
	\[ k_1^{-1} \leq \lambda_{\min}(\bm{\Sigma}_0)<\lambda_{\max}(\bm{\Sigma}_0)< k_1. 
	\]
\end{enumerate}
Assumption (A1) allows the number of covariates $p_n$ to grow at a subexponential rate with $n$. Assumption (A2) restricts the number of true nonzero rows in $\mathbf{B}_0$.
Assumption (A3) bounds the eigenvalues of the Gram matrices $n^{-1}\sum^n_{i=1}\mathbf{x}^S_i (\mathbf{x}_i^{S})^\top$ for all sets $S$ of size $o(n / \log p_n)$. This assumption is needed to ensure identifiability of the model parameters when $p \gg n$ \cite{BaiGhosh2018}. {\pico Assumption (A3) is analogous to Assumption 4 of \cite{ShinBhattacharyaJohnson2018}, which is weaker than 
Assumptions (B3)-(B4) of \cite{BaiGhosh2018}, Assumption $\text{A}_1(3)$ of \cite{song2023nearly}, and 
Assumption 1 of \cite{CaoKhareGhosh2020}. In particular, the minimum eigenvalue in Assumption (A3) is lower bounded by a sequence $\varepsilon_n$ that \emph{decreases} to zero as $n$ increases, while the upper bound on the maximum eigenvalue $\varepsilon_n^{-1}$ also \emph{diverges} as $n$ increases. Similar assumptions on the eigenstructure of the design matrix are made in \cite{LiDuttaRoy2023, narisetty2014bayesian, YangWainwrightJordan2016}.} Finally, Assumption (A4) bounds the eigenvalues of the unknown covariance matrix $\boldsymbol{\Sigma}_0$. Since the response dimension $q$ is fixed, this assumption is reasonable.

 {\picotwo \begin{remark} Assumption (A2) assumes that $\mathbf{B}_0$ is row-sparse, or that most of the rows in $\mathbf{B}_0$ are zero vectors. This assumption is commonly made in the literature {\rm\cite{BaiGhosh2018,  ChenHuang2012, ChunKeles2010, GohDeyChen2017,KonerWilliams2023,LiNanZhu2015,LiquetMengersonPettittSutton2017,ZHANG2022104835}}. In what follows, all of our theoretical results rely on this assumption. In Section \ref{Discussion}, we discuss how to extend our theoretical results to more general sparsity structures in $\mathbf{B}_0$. 
 \end{remark}}  

To derive the main theoretical results for the one-step Mt-MBSP model, we also need the following assumptions.
\begin{enumerate}[label=(B\arabic*)]
	\item
 {\picotwo Let $k_2 =k_0+d$ and $k_2\in(0,1)$, where 
 $k_0$ and $d$ are defined in Assumption (A3)}.  Assume 
$\tau =  C n^{k_2 - 1 - \rho}\exp(n^{k_2-1})$ is a decreasing function of $n$ for some constants $C > 0$ and $\rho>k_2$.
	\item For the slowly varying 
	function $L(\cdot)$ in the hyperprior $p(\xi_j)$ in \eqref{global_local_prior}, $\lim_{t\rightarrow \infty}L(t)\in(0,\infty)$. That is, there exists $c_0>0$ such that $L(t)\geq c_0$ for all $t\geq t_0$ for some $t_0$ which depends on both $L$ and $c_0$.
	\item The maximum entry in $\mathbf{B}_0$ satisfies $\sup_n\sup_{j,~k}|(\mathbf{B}_{0})_{jk}|^2 =O(\log p_n)$.
 \item {\picotwo Let $\kappa_{i\ell}=f_{i\ell}^{(1)} -0.5 f^{(2)}_{i\ell}$, $i=1,\dots,n$, where $\kappa_{i\ell}$ as defined as in \eqref{wyz}, and $\ell \in \{1, \ldots, q\}$ is the index of a discrete random variable of the form \eqref{polson_model}. For any $t>0$,  $\max_{1\leq i \leq n}\kappa_{i\ell}<n^{t}$~ almost surely as $n \rightarrow \infty$.} 
\end{enumerate}
Assumption (B1) specifies an appropriate rate of decay for the global shrinkage parameter $\tau$ in \eqref{Mt-MBSP} to ensure that most coefficients are shrunk to zero. Assumption (B2) is a very mild condition on $L(\cdot)$ in the hyperprior \eqref{global_local_prior} to ensure that the tails do not decay too quickly. All of the shrinkage priors listed in Table \ref{table:priors} satisfy condition (B2). Assumption (B3) bounds the magnitude of the entries in $\mathbf{B}_0$ at a rate that diverges polynomially in $n$.
{\picotwo  Finally, Assumption (B4) is a mild condition which gives a tail bound for any discrete response variable $y_{\ell}$ in our model. 
Note that $\kappa_{\ell}^2=1/4$ (a constant) when $y_{\ell}$ is Bernoulli, $\kappa_{\ell}^2=(y_{\ell}-M/2)^2$ when $y_{\ell}$ is binomial with $M$ trials, and $\kappa_{\ell}^2=(y_{\ell}-r)^2/4$ when $y_{\ell}$ is negative binomial with dispersion parameter $r$. Obviously, (B4) holds for Bernoulli, binomial, and multinomial distributions. In Proposition \ref{prop:NB} of the Appendix, we prove that Assumption (B4) also holds for the NB distribution.
} 
\begin{theorem}[posterior contraction rate for one-step Mt-MBSP estimator] \label{one-step_contraction_rate}
	Assume that \eqref{gaussianresponse}-\eqref{polson_model} hold where the natural parameters $\bm{\theta}_{i0}$'s satisfy \eqref{blatent} with true parameters $(\mathbf{B}_0, \bm{\Sigma}_0)$. Suppose that Conditions (A1)-(A4) and (B1)-{\picotwo (B4)} hold, and we endow $\mathbf{B}_n$ with the prior \eqref{Mt-MBSP}-\eqref{global_local_prior} and $\bm{\Sigma}$ with the inverse-Wishart prior \eqref{IWpriorSigma}. Then, for any arbitrary $\varepsilon>0$, 
	\begin{equation*}
		\sup_{\mathbf{B}_0}  E_{\mathbf{B}_0} P \left( 
		\|\mathbf{B}_n-\mathbf{B}_0\|_F> \varepsilon \left( \frac{s_0 \log p_n}{n} \right)^{1/2} ~\bigg|~{\picotwo \mathbf{X}_n, \mathbf{Y}_n}
		\right)\longrightarrow  0\mbox{~as~} n\rightarrow \infty,
	\end{equation*}
\end{theorem}
where $E_{\mathbf{B}_0}$ denotes the expectation operator under the probability law induced by our mixed-type regression model \eqref{gaussianresponse}-\eqref{polson_model} with true parameter $\mathbf{B}_0$ in \eqref{blatent}.

{\pico \begin{remark} In the literature on high-dimensional Bayesian regression, several authors have relaxed the assumption of exact sparsity. For example, in {\rm\cite{LiDuttaRoy2023, narisetty2014bayesian, YangWainwrightJordan2016}}, the regression coefficients in $S_0^c = \{ 1, \ldots, p_n \} \setminus S_0$ do \emph{not} need to be exactly zero but can instead have bounded norm. An inspection of the proof of Theorem \ref{one-step_contraction_rate} reveals that we can also obtain the same posterior contraction rate under ``approximate'' sparsity of $\mathbf{B}_0$. The proof of Theorem \ref{one-step_contraction_rate} entails decomposing $\lVert \mathbf{B}_n - \mathbf{B}_0 \rVert_F^2 = \sum_{j \in S_0} \lVert \mathbf{b}_j^{n} - \mathbf{b}_j^{0} \rVert^2 + \sum_{j \in S_0^c} \lVert \mathbf{b}_j^{n} - \mathbf{b}_j^{0} \rVert^2$ and separately bounding the error for each term (here, $\mathbf{b}_j^{n}$ and $\mathbf{b}_j^{0}$ denote the $j$th row of $\mathbf{B}_n$ and $\mathbf{B}_0$ respectively). Therefore, provided that $\sum_{j \in S_0^c} \lVert \mathbf{b}^{0}_{j} \rVert^2$ is bounded above appropriately, the total approximation error can be well-controlled and we can obtain the same result in Theorem \ref{one-step_contraction_rate}. 

However, the aim of our work is to highlight the theoretical and practical advantage of a two-stage method (Section \ref{computation}) when $p_n = O(\exp(n \alpha)), \alpha \geq 1$. Thus, to simplify the presentation, we keep the assumption of exact sparsity for $\mathbf{b}_j^{0}, j \in S_0^c$, i.e. $\sum_{j \in S_0^c} \lVert \mathbf{b}^{0}_{j} \rVert^2 = 0$. It should also be noted that {\rm\cite{LiDuttaRoy2023, narisetty2014bayesian, YangWainwrightJordan2016}} establish \emph{variable selection} consistency for the Gaussian linear regression model, whereas Theorem \ref{one-step_contraction_rate} concerns \emph{estimation} consistency for a mixed-type multivariate regression model. In general, variable selection consistency does not imply estimation consistency or vice versa. For example, {\rm\cite{ZouJASA2006}}showed that when the tuning parameter $\lambda_n$ in the LASSO estimator {\rm\cite{Tibshirani1996}} satisfies $\lambda_n = O(n^{1/2})$, the LASSO achieves estimation consistency but has inconsistent variable selection. 
\end{remark}}

\subsection{Posterior inconsistency of one-step Mt-MBSP under exponential growth}

Theorem \ref{one-step_contraction_rate} implies that $\log p_n = o(n)$ is sufficient for posterior contraction, i.e. as $n \rightarrow \infty$, the one-step Mt-MBSP posterior concentrates all of its mass in a ball of shrinking radius $(s_0 \log p_n / n)^{1/2} \rightarrow 0$. Subexponential growth of the predictors relative to sample size is a common assumption for posterior consistency in the Bayesian variable selection literature \citep{BaiGhosh2018,chakraborty2020bayesian,narisetty2014bayesian,ning2020bayesian, ZHANG2022104835}. However, it is unclear if the condition that $\log p_n = o(n)$ is necessary or if it can be relaxed. The next theorem establishes that this condition is in fact necessary. In other words, if we relax $\log p_n = o(n)$ to $\log p_n = O(n)$, then the posterior may be \textit{inconsistent}. 

\begin{theorem} [posterior inconsistency under exponential growth] \label{thm_inconsistent} Assume the same set-up as Theorem \ref{one-step_contraction_rate}, where we endow $\mathbf{B}_n$ with the prior \eqref{Mt-MBSP}-\eqref{global_local_prior} and $\bm{\Sigma}$ with the prior \eqref{IWpriorSigma}. Assume that $\log p_n=  C n$ where $C > 0$, $s_0<\infty$, and conditions (A2)-(A4) and (B1)-{\picotwo(B4)} hold. Then there exists a constant $\delta\in(0,1)$ such that for $\varepsilon > 0$,
	\begin{equation*}
		{\picotwo \inf_{\mathbf{B}_0}}~ E_{\mathbf{B}_0} P \left( 
		\| \mathbf{B}_n-\mathbf{B}_{0} \|_F> \varepsilon ~\bigg|~{\picotwo \mathbf{X}_n, \mathbf{Y}_n}
		\right)> \delta\ \mbox{~~as~~} n\rightarrow \infty.
	\end{equation*}
\end{theorem}
Theorem \ref{thm_inconsistent} shows that if we relax $\log p_n = o(n)$ to $\log p_n = O(n)$, then the one-step Mt-MBSP posterior may \textit{not} concentrate all of its mass inside a ball of fixed radius $\varepsilon > 0$ almost surely as $n \rightarrow \infty$. This leads to the following \textit{necessary} condition for posterior consistency.
\begin{corollary} \label{corollary_inconsistent}
	Assume the same set-up as Theorem \ref{one-step_contraction_rate}. Suppose that conditions (A2)-(A4) and (B1)-(B3) hold. Then the condition that $\log p_n = o(n)$ is necessary for posterior consistency under the one-step Mt-MBSP model.
\end{corollary}

Almost all of the literature on asymptotics for Bayesian high-dimensional regression focuses on sufficient conditions for posterior consistency. Meanwhile, \textit{necessary} conditions are rarely investigated. We have only been able to find two other exceptions to this. Namely, \cite{GhoshKhareMichailidis2019} and  \cite{SparksKhareGhosh} investigated necessary conditions for posterior consistency in Gaussian linear regression models with Gaussian priors when $p \leq n$. Apart from the fact that they do not consider the $p \gg n$ regime,  \cite{GhoshKhareMichailidis2019} and 
\cite{SparksKhareGhosh} are able to handle the posterior distribution directly since they have Gaussian models with conjugate Gaussian priors. Their proof techniques thus seem difficult to adapt to cases where explicit, conjugate computations are unavailable. Taking a different approach, our proof of Theorem \ref{thm_inconsistent} provides a general recipe for proving posterior inconsistency which does \textit{not} require conjugacy.

Our results confirm that if all $p_n$ covariates are used to estimate $\mathbf{B}_0$, then posterior consistency is guaranteed if and only if $\log p_n = o(n)$. This motivates us to develop a two-step estimator in Section \ref{computation} which first \textit{reduces} the dimensionality of the covariate space and thus can consistently estimate the true $\mathbf{B}_0$ \textit{even if} $\log p_n = O(n^{\alpha}), \alpha \geq 1$.

Our proof technique in Theorem \ref{thm_inconsistent} may be useful for establishing necessary conditions for posterior consistency in other high-dimensional Bayesian models. For example, the proof of Theorem \ref{thm_inconsistent} can be suitably adapted to show that subexponential growth of $p$ with respect to $n$ is both necessary and sufficient for posterior consistency of one-stage estimators in univariate regression models. This is formalized in the following proposition.

\begin{proposition}
	Suppose that we have a univariate regression model where $p( y_i \mid \theta_{i0}) \propto \exp\{ -(y_i-\theta_{i0})^2 / 2 \}$ if the $y_i$'s are continuous, and $p(y_i \mid \theta_{0i}) \propto \{ \exp(\theta_{i0}) \}^{f^{(1)}} / \{ 1+\exp(\theta_{i0}) \}^{f^{(2)}}$ if the $y_i$'s are discrete. Assume that for $i=1, \ldots, n$, $\theta_{i0}= \mathbf{x}_i^\top \bm{\beta}_0$, where $\bm{\beta}_0 = (\beta_{01}, \ldots, \beta_{0p_n})^\top \in \mathbb{R}^{p}$ and $\sup_{1 \leq j \leq p_n} |\beta_{0j} | = O(\log p_n)$. Suppose that we estimate $\bm{\beta}_0$ with a GL prior \eqref{global_local_prior} (with $q=1$) and Assumptions (A2)-(A3) and (B1), {\picotwo (B2), and (B4)} hold. Then the condition that $\log p_n = o(n)$ is both necessary and sufficient for posterior consistency.
\end{proposition}

\section{Two-step Mt-MBSP}
\label{computation}

\setcounter{equation}{0}

\subsection{Two-step algorithm} \label{twostep}
In this section, we devise a new two-step approach suitable for large $p$ and study its asymptotic properties when $\log p_n = O(n^{\alpha}), \alpha \geq 1$. The regime of $\log p_n = O(n^{\alpha}), \alpha \geq 1$, subsumes slower growth rates (e.g. polynomial growth $p_n \propto n^{k}, k > 1$) as special cases. We stress that the algorithm in this section is meant to be applied in scenarios when $p \gg n$. If $p$ is small, then it may be sufficient to run the one-step Gibbs sampler to obtain good inference. 

Our proposed procedure is as follows:
\begin{itemize}
	\item[]\textbf{Step 1.} Fit the Mt-MBSP model with all $p$ predictors and obtain the 0.025 and 0.975 quantiles of the entries $b_{jk}$'s in $\mathbf{B}$. For some small threshold $\gamma > 0$, first find an initial set ${\cal A}_n$ such that
	\begin{equation} \label{setAn}
		{\cal A}_n= \left\{
		j \mid\mbox{
			$\max\limits_{k=1,\dots,q}q_{0.025}(b_{jk}) > -\gamma$~~or~~ 
			{\picotwo $ \min\limits_{k=1,\dots,q}q_{0.975}(b_{jk}) <\,\gamma$} 
            }
		\right\}.
	\end{equation}
	{\picotwo In \eqref{setAn},  the element $j$ in ${\cal A}_n$ must satisfy either the condition 
``\textit{at least one} of the $q$ entries in $\{b_{jk}\}_{k=1}^q$ meets $q_{0.025}(b_{jk}) > -\gamma$'' or  the condition  ``\textit{at least one} of the $q$ entries in $\{b_{jk}\}_{k=1}^q$ meets $q_{0.975}(b_{jk}) <  \gamma$.''} 
 Letting 
	$q_j= \lvert \max_{k=1,\dots,q}q_{0.5}( b_{jk}) \rvert$, where $q_{0.5}(\cdot)$ denotes the median, and $K_n = \min\{n-1, |\mathcal{A}_n| \}$, determine the final candidate set $\mathcal{J}_n \subset {\cal A}_n$ as
	\begin{equation} \label{setJn}
		{\cal J}_n= \left\{
		j \in {\cal A}_n \mid j \mbox{ is one of the $K_n$ largest  $q_j$'s}
		\right\}. 
	\end{equation}
	\item[]\textbf{Step 2.} Fit the one-step Mt-MBSP model to \textit{only} the final $K_n$ variables in the candidate set $\mathcal{J}_n$ from Step 1. Then select variables using selection rule \eqref{a0}.
\end{itemize}
Our approach differs from other feature screening methods, most of which are based on ranking the covariates by some measure of marginal utility such as marginal correlation \citep{fan2008sure, HeXuKang2019}. In contrast, Step 1 above more closely resembles the role of a slack variable in support vector machines (SVMs) which allows for violations of the margin \citep{camastra2005novel}. For example, if $\gamma = 0.05$ and $q_j$ has a corresponding 95\% posterior credible interval of $(-0.03, 0.12)$, then we would still regard the $j$th covariate as significant in Step 1 and include it in 
 the set $\mathcal{J}_n$. 

The goal of Step 1 in our two-step algorithm is to reduce the number of parameters to a much smaller candidate set, rather than to estimate the coefficients well. We \textit{want} the initial set \eqref{setAn} to be larger than the set \eqref{a0} so that we can identify more significant variables, particularly if estimation is suboptimal at this stage. Since $|{\cal A}_n|$ could still be large, we collect the variables in ${\cal A}_n$ with the $K_n$ largest values of $\{ \lvert \max_{k=1,\dots,q}q_{0.5}(b_{jk}) \rvert \}_{j=1}^{p}$ into a final candidate set $\mathcal{J}_n$, where $|\mathcal{J}_n | = K_n < n$. Step 2 of our algorithm then further refines the search only on the set $\mathcal{J}_n$. 

In practice, both Step 1 and Step 2 are fit using the Gibbs sampler in Appendix \ref{S1:Gibbs}. However, since Step 1 typically screens out a very large number of covariates, Step 2 is computationally inexpensive to run, regardless of the size of $p$. We also employ the fast sampling technique of \cite{bhattacharya2016fast} in Step 1, so that we can avoid having to invert a large $p \times p$ matrix when sampling from the columns of $\mathbf{B}$. The fast sampling algorithm of \cite{bhattacharya2016fast} ensures that the computational complexity for each Gibbs sampling iteration of Step 1 is only linear in $p$, thereby ensuring the computational feasibility of our approach.

{\pico The role of $\gamma$ in \eqref{setAn} is distinct from the role of the global shrinkage parameter $\tau$ in the GL shrinkage prior \eqref{global_local_prior}. Under sparse situations, the global shrinkage parameter $\tau$ shrinks all of the posterior estimates in $\mathbf{B}$ to zero, which decreases the chance of making Type I errors (or false discoveries). The local shrinkage parameters $\xi_j$'s in \eqref{global_local_prior} counteract the chances of making Type II errors (or false negatives). Even so, the selection rule \eqref{a0} may not work well in high dimensions, especially in light of Theorem \ref{thm_inconsistent}. On the other hand, $\gamma$ in \eqref{setAn} allows more variables to remain in the initial set $\mathcal{J}_n$ even if they fail to meet the selection criterion \eqref{a0}. This not only helps to control the Type II error when $p \gg n$ but it also improves the estimation performance in Step 2.}




\subsection{Sure screening property and sharper contraction rate} \label{twosteptheory}

We now rigorously justify our two-step procedure. The next theorem states that under a mild ``beta-min'' condition, our two-step algorithm has the sure screening property when $\log (p_n) = O(n^{\alpha})$, $\alpha \geq 1$. Intuitively, we require the nonzero entries in $\mathbf{B}_0$ to be sufficiently large; otherwise, it would be difficult to distinguish them from zero.

\begin{theorem}[screening property] \label{thm_C}
	Suppose that $\log (p_n) = O(n^{\alpha})$ where $\alpha \geq 1$, $s_0 = o(n)$ and $s_0 \geq 1$, and Assumptions (A3)-(A4) and (B1)-{\picotwo(B4)} hold. Further, assume that 
	\[
	\min_{j\in S_0}\max_{k=1,\dots,q} |(\mathbf{B}_{0})_{ij}| \geq \frac{c_3}{n^{\zeta}} \mbox{~~for some  $2\zeta < \alpha$},
	\]
	where $S_0$ is the set of indices of the true nonzero rows in $\mathbf{B}_0$. Let $\gamma$ in \eqref{setAn} be set to $\gamma = (r_n\log p_n/n)^{1/2}$, where $r_n \in (s_0,~n)$, and let $\mathcal{J}_n$ be the final candidate set \eqref{setJn} in Step 1. Then under the two-step algorithm,
	\begin{equation*}
		\sup_{\mathbf{B}_0}  E_{\mathbf{B}_0}  P( S_0\subset {\cal J}_n\mid {\picotwo \mathbf{X}_n, \mathbf{Y}_n} )\longrightarrow 1\mbox{~~as $n \rightarrow \infty$}.
	\end{equation*}
\end{theorem}
The screening property in Theorem \ref{thm_C} guarantees that the final candidate set ${\cal J}_n$ in Step 1 of the proposed method can capture the true set of nonzero regression coefficients $S_0$, even if $\log p_n/n\rightarrow\infty$. Based on Theorem \ref{thm_C}, we construct a two-step estimator $\widetilde{\mathbf{B}}_n \in \mathbb{R}^{p_n \times q}$ as follows:
\begin{itemize}
	\item If $j \in \mathcal{J}_n$, then set the $(j,k)$th entry $\widetilde{b}_{jk}$ to be the estimator obtained from Step 2 of the two-step algorithm. 
	\item Otherwise, if $j \notin \mathcal{J}_n$, then set $\widetilde{b}_{jk} = 0$.  
\end{itemize}
The next theorem establishes the convergence rate of this two-step estimator $\widetilde{\mathbf{B}}_n$.

\begin{theorem}[posterior contraction rate for two-step Mt-MBSP estimator] \label{thm_D2}
	Assume the same setup as Theorem \ref{thm_C} where $\log p_n = O(n^{\alpha}), \alpha \geq 1$. For the set $\mathcal{J}_n$ in \eqref{setJn}, assume that $|\mathcal{J}_n | = K_n$ satisfies $s_0\log K_n=o(n)$. Then for any arbitrary $\varepsilon>0$, the two-step estimator $\widetilde{\mathbf{B}}_n$ satisfies
	\begin{equation*}
		\sup_{\mathbf{B}_0}  E_{\mathbf{B}_0} P \left( \lVert \widetilde{\mathbf{B}}_n - \mathbf{B}_0 \rVert_F > \varepsilon  \left( \frac{s_0\log K_n}{n}\right)^{1/2} ~\bigg|~{\picotwo \mathbf{X}_n, \mathbf{Y}_n}
		\right)\longrightarrow  0\mbox{~as~} n\rightarrow \infty.
	\end{equation*}
\end{theorem}

Theorem \ref{thm_D2} shows that, unlike the one-step estimator analyzed in Section \ref{theory}, the two-step estimator $\widetilde{\mathbf{B}}_n$ \textit{can} consistently estimate $\mathbf{B}_0$ with posterior contraction rate $(s_0\log K_n/n)^{1/2} \rightarrow 0$ when $\log p_n/n\rightarrow\infty$. Since $K_n = \min\{ n-1, | \mathcal{A}_n| \}$, Theorem \ref{thm_D2} essentially requires that $s_0 = o(n / \log n)$ in order for the two-step estimator $\widetilde{\mathbf{B}}_n$ to achieve posterior contraction. This is less stringent than Assumption (A2) that $s_0 = o(n / \log p_n)$, which is needed for posterior contraction of the one-step estimator in Theorem \ref{one-step_contraction_rate}.  Thus, one practical benefit of the two-step estimator is that it can achieve consistent estimation under less strict requirements on the true sparsity level. 

We stress that the growth rate $\log p_n = O(n^{\alpha}), \alpha \geq 1$, encompasses scenarios where $p_n$ can grow subexponentially. For example, Theorems \ref{thm_C} and \ref{thm_D2} can allow $\log p_n \propto n^{\delta}$, where $0 < \delta < 1$, and even slower growth rates as well. As long as $K_n \prec n \prec p_n$, the two-step estimator's contraction rate of $(s_0 \log K_n / n)^{1/2}$ in Theorem \ref{thm_D2} is provably sharper than the one-step estimator's contraction rate of $(s_0 \log p_n / n)^{1/2}$ in Theorem \ref{one-step_contraction_rate}. Thus, the two-step estimator has an advantage over the one-step estimator whenever $p_n$ diverges faster than $n$ (not necessarily exponentially fast). As shown in Theorem \ref{thm_inconsistent}, the one-step estimator is also incapable of consistent estimation when $p_n$ grows exponentially fast, unlike the two-step estimator.

\subsection{Selecting the threshold in two-step Mt-MBSP}

{\pico 
In Theorem \ref{thm_C}, the threshold $\gamma$ in Step 1 of our two-step MT-MBSP algorithm depends on $r_n \in (s_0, n)$. However, in practice, the true sparsity level $s_0$ is unknown. One approach could be to first obtain an estimate for the sparsity level $\widehat{s}_0$ and then set (for example) $\widehat{\gamma} = [(\widehat{s}_0+1) \log p_n / n]^{1/2}$. This type of empirical Bayes approach has been considered by several authors \citep{Rockova2018,vanderPasKleijnvanderVaart2014}. 

Another method is to tune $\gamma$ from a grid of candidate values, as in \citep{KonerWilliams2023, narisetty2014bayesian}. We prefer this approach since it obviates the challenge of obtaining an estimate for $s_0$. We propose tuning $\gamma$ from a set of candidate values using the Watanabe-Akaike information criterion (WAIC) \citep{Watanabe2010}. Let
\begin{align*} \label{logLL}
     & {\picotwo \ell( \mathbf{B}_{\mathcal{J}_n(\gamma)}, \boldsymbol{\Sigma} \mid \mathbf{z}_i, \boldsymbol{\omega}_i)} \\
     & {\picotwo = - \frac{1}{2} \log | \bm{\Omega}_i + \bm{\Sigma} | -\frac{1}{2} (\mathbf{z}_i - \mathbf{B}_{\mathcal{J}_n(\gamma)}^\top \mathbf{x}_{i, \mathcal{J}_n(\gamma)})^\top ( \boldsymbol{\Omega}_i + \boldsymbol{\Sigma} )^{-1} ( \mathbf{z}_i - \mathbf{B}_{\mathcal{J}_n(\gamma)}^\top \mathbf{x}_{i, \mathcal{J}_n(\gamma)})}, \numbereqn\\
\end{align*}
where $\mathcal{J}_n(\gamma)$ is the set \eqref{setJn} determined by using $\gamma$ as the threshold in \eqref{setAn}, $\mathbf{B}_{\mathcal{J}_n(\gamma)}$ is the submatrix of $\mathbf{B}$ with row indices in $\mathcal{J}_n(\gamma)$, and $\mathbf{x}_{i, \mathcal{J}_n(\gamma)}$ is the subvector of $\mathbf{x}_i$ with indices in $\mathcal{J}_n(\gamma)$. {\picotwo  Technically, the log-likelihood for the $i$th observation contains a term $\log[ p_0(\mathbf{z}_i, \bm{\omega}_i)]$ where $p_0(\mathbf{z}_i, \bm{\omega}_i)$ is the baseline function in \eqref{jointzomega}. However, this contribution to the log-likelihood does not depend on the parameters $(\mathbf{B}_{\mathcal{J}_n(\gamma)}, \bm{\Sigma})$ and is typically small in practice. Therefore, we use \eqref{logLL} in the WAIC, which corresponds to the log of the second term depending on $(\mathbf{B}, \bm{\Sigma})$ in \eqref{jointzomega} after integrating out the random effect $\mathbf{u}_i$.}

Let $m_i$ and $v_i$ denote the estimated posterior mean and variance of $\ell( \mathbf{B}_{\mathcal{J}_n(\gamma)}, \boldsymbol{\Sigma} \mid \mathbf{z}_i, \boldsymbol{\omega}_i)$ in \eqref{logLL}. The WAIC for a given $\gamma$ is defined as
\begin{equation} \label{WAIC}
 \text{WAIC}(\gamma) = -2 \sum_{i=1}^{n} m_i + 2 \sum_{i=1}^{n} v_i.
\end{equation}
We select the $\gamma$ which minimizes the WAIC \eqref{WAIC}. In practice, we run Step 2 of the two-step Mt-MBSP algorithm in parallel for each candidate $\gamma$ and estimate \eqref{WAIC} from the MCMC samples in Step 2.

\cite{KonerWilliams2023, narisetty2014bayesian} used the Bayesian information criterion (BIC) \citep{Schwarz1978} to select a suitable threshold in their Bayesian variable selection methods.  We prefer to use WAIC \eqref{WAIC}, since WAIC is asymptotically equivalent to Bayesian leave-one-out cross-validation (LOOV) while also being much less computationally intensive than LOOCV \citep{Watanabe2010}. Unlike other criteria such as BIC or deviance information criterion (DIC) \citep{Spiegelhalter2002}, WAIC also does \emph{not} require the posterior to be approximately normal. Finally, WAIC is more ``fully Bayesian'' since it averages each term \eqref{logLL} over the entire posterior distribution, whereas BIC and DIC are conditional on a single point estimate of \eqref{logLL} \cite{GelmanHwangVehtari2014}. 

\begin{remark}
    In practice, it cannot be verified whether the $\gamma$ selected by the WAIC criterion satisfies $\gamma = (r_n \log p_n / n)^{1/2}$ for some $r_n \in (s_0, n)$ as in Theorem \ref{thm_D2}. This is because $s_0$ is typically unknown. However, models with lower WAIC indicate that these particular subsets of predictors \eqref{setJn} provide better fits to the data. Thus, choosing $\gamma$ from WAIC should at least produce models which exclude spurious predictors that are not supported by the data. More rigorous theoretical guarantees of WAIC in the present context can be explored in future work.
\end{remark}
}

\section{Simulations and real data analyses}\label{experiments}

{\pico In this section, we illustrate Mt-MBSP on a variety of simulated and real datasets with a wide range of sizes for $p$. In Section \ref{simulation}, we provide numerical evidence to support our theoretical results from Sections \ref{theory} and \ref{computation} in finite samples. In particular, we show that when $p > n$, two-step MBSP outperforms one-step Mt-MBSP in terms of estimation and variable selection performance. In Section \ref{applications1}, we demonstrate Mt-MBSP on a chronic kidney disease dataset with a small number of covariates ($p=23$), and in Section \ref{applications2}, we study a very high-dimensional cancer dataset with $p=9183$ genes.}

\subsection{Simulation studies}
\label{simulation}

To corroborate our theory, we conducted simulation studies with different combinations of continuous, binary, and count responses in $\mathbf{Y}_n$. We have implemented our method in an \textsf{R} package, available at \url{https://github.com/raybai07/MtMBSP}.


We considered the following {\picotwo triplets for $(p ,n, s_0)$: $(500, 100, 6)$, $(1000, 150, 8)$, and $(2000, 200, 10)$}.  Therefore, we were in the situation where $p > n$, $p$ grows faster than $n$, {\picotwo and $s_0$ also grows but at a rate of $o(n / \log p)$}. The rows of the design matrix $\mathbf{X}_n$ were independently generated from $\mathcal{N}_p(\bm{0}, {\picotwo \bm{\Gamma}_X})$, where ${\picotwo \bm{\Gamma}_X} = (\Gamma_{ij})$ with $\Gamma_{ij} = 0.5^{|i-j|}$. For the true matrix $\mathbf{B}_0$ in \eqref{blatent}, we randomly chose {\picotwo $s_0$ of the rows} to be nonzero and {\picotwo set} the remaining rows to be zero. {\picotwo For each nonzero row, we first randomly sampled $t \in \{ 1, \ldots, q\}$ with equal probability and then randomly selected $t$ of the entries in that row to be nonzero. Therefore, the nonzero rows could also contain null entries, and some of these rows might contain only a single nonzero entry. If $t > 1$, then the nonzero entries corresponding to continuous or binary responses were sampled independently from $\mathcal{U} ([-2, -0.5] \cup [0.5, 2])$, while the nonzero entries corresponding to count responses were sampled independently from $\mathcal{U}([-0.8, -0.4] \cup [-0.4, 0.8])$. If $t=1$, then the single nonzero entry in that row was chosen to be either 1.5 or -1.5 with equal probability if the corresponding response variable was a continuous or binary response; otherwise, it was chosen to be either 0.75 or -0.75 with equal probability if the corresponding response variable was a count response.} The reason that we simulated smaller nonzero coefficients for the count responses was because in order to obtain realistic count values, we could not allow the corresponding entries in $\mathbf{B}_0$ to be too large. 

{\picotwo Following \cite{ChenHuang2012}, we generated the covariance matrix $\bm{\Sigma}$ as $\bm{\Sigma} = \sigma^2 \bm{\Xi}$ where the diagonal elements of $\bm{\Xi} $ were set to one and the off-diagonal entries were set to $0.5$. We chose $\sigma^2$ so that the signal-to-noise ratio $\text{tr}(\mathbf{B}_0^\top \bm{\Gamma}_{X} \mathbf{B}_0) / \text{tr}(\mathbf{U}^\top\mathbf{U})$ was equal to one, where $\mathbf{U} = (\mathbf{u}_1, \ldots, \mathbf{u}_n)^\top$ is the $n \times q$ noise matrix with rows $\mathbf{u}_i \sim \mathcal{N}_q(\bm{0}, \bm{\Sigma})$.}
Finally, based on \eqref{gaussianresponse}-\eqref{polson_model} and \eqref{blatent}, we generated the continuous outcomes from a Gaussian distribution, the binary outcomes from a Bernoulli distribution, and the count outcomes from an NB distribution with $r=50$.

We  considered the following six simulation scenarios:
\begin{itemize}
	\item[] Scenario 1: two continuous and two binary responses ($q=4$)
	\item[] Scenario 2: two continuous, two binary, two count responses ($q=6$)
	\item[] Scenario 3: three count and two continuous responses ($q=5$) 
	\item[] Scenario 4: four binary responses ($q=4$)
	\item[] Scenario 5: three binary and two count responses ($q=5$)
	\item[] Scenario 6: three count responses ($q=3$)
\end{itemize}
We generated 100 synthetic datasets for each scenario and fit the Mt-MBSP model with horseshoe priors (i.e. TPBN priors \eqref{TPBN} with $u=a=0.5$) as the local scale parameters \eqref{global_local_prior}. For the other hyperparameters in the Mt-MBSP model, {\picotwo we followed \cite{BaiGhosh2018, ZHANG2022104835}} and set $\tau = \max \{ p^{-1} (n \log n)^{-1/2}, 10^{-5} \}$ in \eqref{Mt-MBSP} and $d_1 = q$ and $d_2 = 10$ in \eqref{IWpriorSigma}. Although Theorems \ref{one-step_contraction_rate}-\ref{thm_D2} imply that the theoretical value of $\tau$ should be even smaller (Assumption (B2)), the choice for $\tau$ in (B2) is extremely close to zero for large $p$ and causes numerical problems in practice \cite{ZHANG2022104835, vanderPasKleijnvanderVaart2014}. For count responses, we endowed the dispersion parameter $r$ with a $\mathcal{G}(c_1,c_2)$ prior with $c_1=10$ and $c_2=1$. 

We ran the one-step algorithm (Section \ref{BVS}) for 1100 MCMC iterations, with a burn-in of 100 iterations. For the two-step algorithm (Section \ref{twostep}), Step 1 and Step 2 were each run for 1100 MCMC iterations, with a burn-in of 100 iterations. {\pico Our MCMC convergence analysis in Appendix \ref{S5:Diagnostics} confirms that this number of Gibbs sampling iterations was enough to achieve convergence and obtain quality MCMC posterior estimates. For two-step Mt-MBSP, we tuned the threshold $\gamma$ in Step 1 using WAIC \eqref{WAIC} from a grid of candidate values $\{ 0.02, 0.04, \ldots, 0.40 \}$. This choice of grid worked well in our simulations; however, on real data, practitioners may want to experiment with different choices of candidate values for $\gamma$.}

 {\pico Let $\widehat{\mathbf{B}}$ denote a posterior median estimator for $\mathbf{B}_0$, and let $\widehat{\mathbf{B}}^{S_0}$ and $\mathbf{B}^{S_0}$ denote the $s_0 \times q$ submatrices of $\widehat{\mathbf{B}}$ and $\mathbf{B}_0$ which contain the row indices in $S_0$. That is, $\mathbf{B}^{S_0}$ contains the true \emph{nonzero} rows of $\mathbf{B}_0$.} Meanwhile, TP, FP, TN, and FN denote the number of true positives, false positives, true negatives, and false negatives, respectively. {\picotwo Note that TP, FP, TN, and FN were calculated using all $pq$ regression coefficients. This is because some of the $s_0$ nonzero rows in $\mathbf{B}_0$ contained individual entries equal to zero.} 
 
 We recorded the following metrics for both one-step Mt-MBSP and two-step Mt-MBSP:  
\begin{itemize}
    \item root mean squared error (rMSE) where rMSE $ = (pq)^{-1/2} \lVert \widehat{\mathbf{B}} - \mathbf{B}_0 \rVert_F$;
    \item  {\pico rMSE for $\mathbf{B}^{S_0}$, i.e. $\text{rMSE}(S_0)$ where $\text{rMSE}(S_0) = (s_0 q)^{-1/2} \lVert \widehat{\mathbf{B}}^{S_0} - \mathbf{B}_0^{S_0} \rVert_F$;}
    \item coverage probability (CP) for the 95\% posterior credible intervals, i.e. the percentage of credible intervals for the entries $b_{jk}$'s that contained the true $(\mathbf{B}_0)_{jk}$'s;
    \item sensitivity (Sens) where Sens = TP/(TP+FN);
    \item specificity (Spec) where Spec = TN/(TN+FP);
    \item and Matthews correlation coefficient (MCC) \cite{Matthews1975} where  
    \begin{align*}
        \text{MCC} = \frac{\text{TP} \times \text{TN}-\text{FP} \times \text{FN}}{ \left\{ (\text{TP}+\text{FP}) \times (\text{TP}+\text{FN}) \times (\text{TN}+\text{FP}) \times (\text{TN}+\text{FN}) \right\}^{1/2}}. 
    \end{align*}
    MCC lies between -1 and 1, with a higher MCC indicating greater selection accuracy.
\end{itemize}
{\pico In two-step Mt-MBSP, recall that $\mathcal{J}_n$ \eqref{setJn} is the set of variables that are \textit{not} screened out after Step 1. Our two-step Mt-MBSP estimator $\widetilde{\mathbf{B}}$ is by construction a $p \times q$ matrix where the rows $\mathbf{b}_j, j \notin \mathcal{J}_n$, are fixed to $\boldsymbol{0}_q$. Therefore, to corroborate the result in Theorem \ref{thm_D2}, we evaluated the estimation performance (or rMSE) using \textit{all} $pq$ entries in $\widetilde{\mathbf{B}}$. However, unlike two-step Mt-MBSP, the one-step Mt-MBSP estimator is \emph{not} exactly sparse. Thus, to ensure a fair comparison between one-step Mt-MBSP and two-step Mt-MBSP, we also evaluated rMSE($S_0$), i.e. the rMSE for \emph{only} the true nonzero rows in $\mathbf{B}_0$.

In our two-stage estimator, only the rows $\mathbf{b}_j, j \in \mathcal{J}_n$, are re-estimated in Step 2. Thus, for two-step Mt-MBSP, we used the 95\% credible intervals from Step 1 to quantify uncertainty for the entries in $\mathbf{b}_j, j \in \mathcal{J}_n^c$ (i.e. the entries estimated as $\widetilde{b}_{jk} = 0$ in Step 2) and the 95\% credible intervals from Step 2 to quantify uncertainty for the entries in $\mathbf{b}_j, j \in \mathcal{J}_n$. This procedure allowed us to calculate the CP for two-step Mt-MBSP.

It should be noted that most post-selection inference methods do \emph{not} provide inference for the variables that are not selected \citep{KuchibhotlaKolassaKuffner2022}. However, in our case, we can use the 95\% credible intervals from Step 1 to provide some measure of uncertainty for the variables that have been screened out in our two-step algorithm. By construction of the set \eqref{setJn}, all of the credible intervals for $b_{jk}, j \in \mathcal{J}_n^c$, contain zero. Although the estimator for $\mathbf{B}$ in Step 1 is not guaranteed to be consistent (Theorem \ref{thm_inconsistent}), Theorem \ref{thm_C} shows that with a well-chosen threshold $\gamma$ in \eqref{setAn}, we have $\mathcal{J}_n \subset \mathcal{S}_0$, or equivalently, $\mathcal{S}_0^c \subset \mathcal{J}_n^c$. As a result, the 95\% credible intervals for $\mathbf{b}_j, j \in \mathcal{J}_n^c$, obtained from Step 1 will typically still correctly contain zero and give correct inference in this respect.}

\begin{table}[!]
	\caption{{\pico Comparisons between the one-step method and two-step method averaged across 100 Monte-Carlo replicates. The empirical standard errors are provided in parentheses.}}
	\resizebox{.95\textwidth}{!}{\begin{tabular}{rcccccc} 
			\multicolumn{7}{c}{Scenario 1: continuous and binary ($q=4$)} \\[.05in]
			  & \multicolumn{3}{c}{One-step method}  & \multicolumn{3}{c}{Two-step method}  \\ 
     \cmidrule(lr){2-4} \cmidrule(lr){5-7}
			{\picotwo $(p,n,s_0)$}  & {$\picotwo (500,100,6)$} & {\picotwo $(1000,150,8)$} & {\picotwo $(2000,200,10)$} & {\picotwo $(500,100,6)$} & {\picotwo $(1000,150,8)$} & {\picotwo $(2000,200,10)$} \\ 
          \cmidrule(lr){2-4} \cmidrule(lr){5-7}
			{rMSE} & {\footnotesize {\picotwo 0.068~(0.013)}} & {\footnotesize {\picotwo 0.057~(0.009)}} & {\footnotesize {\picotwo 0.046~(0.005)}} & {\footnotesize {\picotwo 0.034~(0.011)}} & {\footnotesize {\picotwo 0.025~(0.007)}} & {\footnotesize {\picotwo 0.017~(0.004)}} \\
                {$\text{rMSE}(S_0)$} & {\footnotesize {\picotwo 0.649~(0.097)}} & {\footnotesize {\picotwo 0.630~(0.072)}} & {\footnotesize {\picotwo 0.601~(0.116)}} & {\footnotesize {\picotwo 0.293~(0.089)}} & {\footnotesize {\picotwo 0.263~(0.075)}} & {\footnotesize {\picotwo 0.232~(0.062)}} \\
			{CP}  & {\footnotesize {\picotwo 0.999~(0.001)}} & {\footnotesize {\picotwo 0.999~(0.000)}} & {\footnotesize {\picotwo 0.999~(0.000)}} & {\footnotesize {\picotwo 0.999~(0.000)}} & {\footnotesize {\picotwo 0.999~(0.000)}} & {\footnotesize {\picotwo 0.999~(0.000)}} \\ 
			{Sens}  &{\footnotesize {\picotwo 0.829~(0.128)}} &  {\footnotesize {\picotwo 0.850~(0.094)}} & {\footnotesize {\picotwo 0.727~(0.104)}} & {\footnotesize {\picotwo 0.919~(0.076)}} & {\footnotesize {\picotwo  0.943~(0.057)}} & {\footnotesize {\picotwo 0.948~(0.054)}} \\  
			{Spec} & {\footnotesize {\picotwo 1.000~(0.000)}} & {\footnotesize {\picotwo 1.000~(0.000)}} & {\footnotesize {\picotwo 1.000~(0.000)}} & {\footnotesize {\picotwo 1.000~(0.000)}} & {\footnotesize {\picotwo 1.000~(0.00)}} & {\footnotesize {\picotwo 1.000~(0.000)}} \\ 
			{MCC}& {\footnotesize {\picotwo 0.904~(0.076)}} & {\footnotesize {\picotwo 0.918~(0.053)}} & {\footnotesize {\picotwo 0.849~(0.063)}} & {\footnotesize {\picotwo 0.947~(0.046)}} & {\footnotesize {\picotwo 0.960~(0.037)}} & {\footnotesize {\picotwo 0.966~(0.029)}}  \\
        \cmidrule(lr){2-4} \cmidrule(lr){5-7}
	\end{tabular}}\\[.1in]
	\resizebox{.95\textwidth}{!}{\begin{tabular}{rcccccc} 
			\multicolumn{7}{c}{Scenario 2: continuous, binary, and count ($q=6$)} \\[.05in]
			&\multicolumn{3}{c}{One-step method}&\multicolumn{3}{c}{Two-step method}  \\ 
        \cmidrule(lr){2-4} \cmidrule(lr){5-7}
				{\picotwo $(p,n,s_0)$}  & {$\picotwo (500,100,6)$} & {\picotwo $(1000,150,8)$} & {\picotwo $(2000,200,10)$} & {\picotwo $(500,100,6)$} & {\picotwo $(1000,150,8)$} & {\picotwo $(2000,200,10)$} \\ 
        \cmidrule(lr){2-4} \cmidrule(lr){5-7}
			{rMSE}   & {\footnotesize {\picotwo 0.030~(0.008)}}  & {\footnotesize {\picotwo 0.024~(0.005)}} & {\footnotesize {\picotwo 0.021~(0.004)}} & {\footnotesize {\picotwo 0.024~(0.007)}} & {\footnotesize {\picotwo 0.017~(0.005)}} & {\footnotesize {\picotwo 0.011~(0.003)}} \\
                  {$\text{rMSE}(S_0)$} & {\footnotesize {\picotwo 0.315~(0.058)}} & {\footnotesize {\picotwo 0.284~(0.076)}} & {\footnotesize {\picotwo 0.281~(0.058)}} & {\footnotesize {\picotwo 0.224~(0.062)}} & {\footnotesize {\picotwo 0.197~(0.060)}} & {\footnotesize {\picotwo  0.167~(0.038)}} \\
			{CP}  & {\footnotesize {\picotwo 0.999~(0.002)}} & {\footnotesize {\picotwo 0.999~(0.000)}} & {\footnotesize {\picotwo 0.999~(0.000)}} & {\footnotesize {\picotwo 0.999~(0.000)}} & {\footnotesize {\picotwo 0.999~(0.000)}} & {\footnotesize {\picotwo 0.999~(0.000)}} \\ 
			{Sens} &{\footnotesize {\picotwo 0.908~(0.082)}} & {\footnotesize {\picotwo 0.942~(0.054)}} & {\footnotesize {\picotwo 0.905~(0.083)}} & {\footnotesize {\picotwo 0.933~(0.067)}} & {\footnotesize {\picotwo 0.966~(0.035)}} & {\footnotesize {\picotwo 0.980~(0.025)}} \\  
			{Spec} & {\footnotesize {\picotwo 1.000~(0.000)}} & {\footnotesize {\picotwo 1.000~(0.000)}} & {\footnotesize {\picotwo 1.000~(0.000)}} & {\footnotesize {\picotwo 1.000~(0.000)}} & {\footnotesize {\picotwo 1.000~(0.000)}} & {\footnotesize {\picotwo 1.000~(0.000)}} \\ 
			{MCC}&  {\footnotesize {\picotwo 0.949~(0.046)}} & {\footnotesize {\picotwo 0.967~(0.028)}} & {\footnotesize {\picotwo 0.947~(0.044)}} & {\footnotesize {\picotwo 0.960~(0.037)}} & {\footnotesize {\picotwo 0.978~(0.020)}} & {\footnotesize {\picotwo 0.984~(0.015)}} \\
           \cmidrule(lr){2-4} \cmidrule(lr){5-7}
	\end{tabular}}\\[.1in]
	\resizebox{.95\textwidth}{!}{\begin{tabular}{rcccccc}
			\multicolumn{7}{c}{Scenario 3: count and continuous ($q=5$)}\\[.05in]
			&\multicolumn{3}{c}{One-step method} &\multicolumn{3}{c}{Two-step method}  \\ 
           \cmidrule(lr){2-4} \cmidrule(lr){5-7}
		{\picotwo $(p,n,s_0)$}  & {$\picotwo (500,100,6)$} & {\picotwo $(1000,150,8)$} & {\picotwo $(2000,200,10)$} & {\picotwo $(500,100,6)$} & {\picotwo $(1000,150,8)$} & {\picotwo $(2000,200,10)$} \\ 
              \cmidrule(lr){2-4} \cmidrule(lr){5-7}
			{rMSE}   & {\footnotesize {\picotwo 0.010~(0.002)}} & {\footnotesize {\picotwo 0.006~(0.001)}} & {\footnotesize {\picotwo 0.005~(0.001)}} & {\footnotesize {\picotwo 0.009~(0.002)}} & {\footnotesize {\picotwo 0.006~(0.001)}} & {\footnotesize {\picotwo 0.004~(0.001)}} \\
                  {$\text{rMSE}(S_0)$} & {\footnotesize {\picotwo 0.100~(0.018)}} & {\footnotesize {\picotwo 0.074~(0.011)}} & {\footnotesize {\picotwo 0.069~(0.010)}} & {\footnotesize {\picotwo 0.094~(0.018)}} & {\footnotesize {\picotwo 0.073~(0.009)}} & {\footnotesize {\picotwo 0.066~(0.009)}} \\
			{CP}  & {\footnotesize {\picotwo 0.999~(0.000)}} & {\footnotesize {\picotwo 0.999~(0.000)}} & {\footnotesize {\picotwo 1.000~(0.000)}} & {\footnotesize {\picotwo 0.999~(0.000)}} & {\footnotesize {\picotwo 0.999~(0.000)}} & {\footnotesize {\picotwo 0.999~(0.000)}} \\ 
			{Sens} & {\footnotesize {\picotwo 0.975~(0.052)}} & {\footnotesize {\picotwo 0.990~(0.046)}} & {\footnotesize {\picotwo 0.961~(0.057)}} & {\footnotesize {\picotwo 0.992~(0.024)}} & {\footnotesize {\picotwo 0.999~(0.003)}} & {\footnotesize {\picotwo 1.000~(0.000)}} \\  
			{Spec} & {\footnotesize {\picotwo 1.000~(0.000)}} & {\footnotesize {\picotwo 1.000~(0.000)}} & {\footnotesize {\picotwo 1.000~(0.000)}} & {\footnotesize {\picotwo 1.000~(0.000)}} & {\footnotesize {\picotwo 1.000~(0.000)}} & {\footnotesize {\picotwo 1.000~(0.000)}} \\ 
			{MCC} & {\footnotesize {\picotwo 0.987~(0.027)}} & {\footnotesize {\picotwo 0.994~(0.026)}} & {\footnotesize {\picotwo 0.980~(0.030)}} & {\footnotesize {\picotwo 0.995~(0.014)}} & {\footnotesize {\picotwo 0.999~(0.004)}} & {\footnotesize {\picotwo 0.999~(0.004)}} \\
           \cmidrule(lr){2-4} \cmidrule(lr){5-7}
	\end{tabular}}\\[.1in]
\resizebox{.95\textwidth}{!}{\begin{tabular}{rcccccc} 
		\multicolumn{7}{c}{Scenario 4: all binary ($q=4$)}\\[.05in]
		&\multicolumn{3}{c}{One-step method} & \multicolumn{3}{c}{Two-step method} \\ 
          \cmidrule(lr){2-4} \cmidrule(lr){5-7}
	{\picotwo $(p,n,s_0)$}  & {$\picotwo (500,100,6)$} & {\picotwo $(1000,150,8)$} & {\picotwo $(2000,200,10)$} & {\picotwo $(500,100,6)$} & {\picotwo $(1000,150,8)$} & {\picotwo $(2000,200,10)$} \\ 
          \cmidrule(lr){2-4} \cmidrule(lr){5-7}
		{rMSE}  & {\footnotesize {\picotwo 0.150~(0.018)}} & {\footnotesize {\picotwo 0.109~(0.010)}} & {\footnotesize {\picotwo 0.073~(0.006)}} & {\footnotesize {\picotwo 0.095~(0.023)}} & {\footnotesize {\picotwo 0.050~(0.013)}} &
        {\footnotesize {\picotwo 0.028~(0.006)}} \\
        {$\text{rMSE}(S_0)$} & {\footnotesize {\picotwo 1.302~(0.176)}} & {\footnotesize {\picotwo 1.202~(0.114)}} & {\footnotesize {\picotwo 1.025~(0.082)}} & {\footnotesize {\picotwo 0.697~(0.179)}} & {\footnotesize {\picotwo 0.491~(0.113)}} & {\footnotesize {\picotwo 0.386~(0.080)}} \\
		{CP}   & {\footnotesize {\picotwo 0.998~(0.001)}} & {\footnotesize {\picotwo 0.999~(0.000)}} & {\footnotesize {\picotwo 0.999~(0.000)}} & {\footnotesize {\picotwo 0.998~(0.001)}} & {\footnotesize {\picotwo 0.999~(0.001)}} & {\footnotesize {\picotwo 0.999~(0.000)}} \\ 
		{Sens} & {\footnotesize {\picotwo 0.583~(0150)}} & {\footnotesize {\picotwo 0.636~(0.125)}} & {\footnotesize {\picotwo 0.510~(0.115)}} & {\footnotesize {\picotwo 0.822~(0.115)}} & {\footnotesize {\picotwo 0.877~(0.089)}} & {\footnotesize {\picotwo 0.881~(0.076)}} \\  
		{Spec}  & {\footnotesize {\picotwo  0.999~(0.000)}} & {\footnotesize {\picotwo 1.000~(0.000)}} & {\footnotesize {\picotwo 1.000~(0.000)}} & {\footnotesize {\picotwo 0.999~(0.001)}} & {\footnotesize {\picotwo 1.000~(0.000)}} & {\footnotesize {\picotwo 1.000~(0.000)}} \\ 
		{MCC}&  {\footnotesize {\picotwo 0.752~(0.107)}} & {\footnotesize {\picotwo 0.790~(0.085)}} & {\footnotesize {\picotwo 0.709~(0.085)}} & {\footnotesize {\picotwo 0.843~(0.096)}} & {\footnotesize {\picotwo 0.898~(0.063)}} & {\footnotesize {\picotwo 0.919~(0.042)}} \\
          \cmidrule(lr){2-4} \cmidrule(lr){5-7}
\end{tabular}}\\[.1in]
\resizebox{.95\textwidth}{!}{\begin{tabular}{rcccccc}
		\multicolumn{7}{c}{Scenario 5: binary and count ($q=5$)}\\[.05in]
		&\multicolumn{3}{c}{One-step method} & \multicolumn{3}{c}{Two-step method}   \\ 
          \cmidrule(lr){2-4} \cmidrule(lr){5-7}
	{\picotwo $(p,n,s_0)$}  & {$\picotwo (500,100,6)$} & {\picotwo $(1000,150,8)$} & {\picotwo $(2000,200,10)$} & {\picotwo $(500,100,6)$} & {\picotwo $(1000,150,8)$} & {\picotwo $(2000,200,10)$} \\ 
          \cmidrule(lr){2-4} \cmidrule(lr){5-7}
		{rMSE}   & {\footnotesize {\picotwo 0.062~(0.012)}} & {\footnotesize {\picotwo 0.055~(0.007)}} & {\footnotesize {\picotwo 0.047~(0.005)}} &
        {\footnotesize {\picotwo 0.033~(0.009)}} &
  {\footnotesize {\picotwo 0.023~(0.005)}} & {\footnotesize {\picotwo 0.017~(0.003)}} \\
                 {$\text{rMSE}(S_0)$} & {\footnotesize {\picotwo 0.705~(0.091)}} & {\footnotesize {\picotwo 0.640~(0.088)}} & {\footnotesize {\picotwo 0.594~(0.132)}} &
                 {\footnotesize {\picotwo 0.307~(0.086)}} &
                 {\footnotesize {\picotwo 0.262~(0.060)}} & {\footnotesize {\picotwo 0.253~(0.054)}} \\
		{CP}   & {\footnotesize {\picotwo 0.999~(0.001)}} & {\footnotesize {\picotwo 0.999~(0.000)}} & {\footnotesize {\picotwo 0.999~(0.001)}} & {\footnotesize {\picotwo 0.999~(0.000)}} & {\footnotesize {\picotwo 0.999~(0.000)}} & {\footnotesize {\picotwo 0.999~(0.000)}} \\ 
		{Sens}  & {\footnotesize {\picotwo 0.777~(0.124)}} & {\footnotesize {\picotwo 0.803~(0.125)}} & {\footnotesize {\picotwo 0.658~(0.115)}} & {\footnotesize {\picotwo 0.880~(0.094)}} & {\footnotesize {\picotwo 0.938~(0.057)}} & {\footnotesize {\picotwo 0.920~(0.078)}} \\  
		{Spec} & {\footnotesize {\picotwo 1.000~(0.000)}} & {\footnotesize {\picotwo 1.000~(0.000)}} & {\footnotesize {\picotwo 1.000~(0.000)}} & {\footnotesize {\picotwo 1.000~(0.001)}} & {\footnotesize {\picotwo 1.000~(0.000)}} & {\footnotesize {\picotwo 1.000~(0.000)}} \\ 
		{MCC}&  {\footnotesize {\picotwo 0.873~(0.075)}} & {\footnotesize {\picotwo 0.889~(0.075)}} & {\footnotesize {\picotwo 0.807~(0.070)}} & {\footnotesize {\picotwo 0.930~(0.053)}} & {\footnotesize {\picotwo 0.960~(0.034)}} & {\footnotesize {\picotwo 0.948~(0.046)}} \\
          \cmidrule(lr){2-4} \cmidrule(lr){5-7}
\end{tabular}}\\[.1in]
\resizebox{.95\textwidth}{!}{\begin{tabular}{rcccccc} 
		\multicolumn{7}{c}{Scenario 6: all count ($q=3$)}\\[.05in]
		&\multicolumn{3}{c}{One-step method} & \multicolumn{3}{c}{Two-step method}  \\ 
          \cmidrule(lr){2-4} \cmidrule(lr){5-7}
		{\picotwo $(p,n,s_0)$}  & {$\picotwo (500,100,6)$} & {\picotwo $(1000,150,8)$} & {\picotwo $(2000,200,10)$} & {\picotwo $(500,100,6)$} & {\picotwo $(1000,150,8)$} & {\picotwo $(2000,200,10)$} \\ 
          \cmidrule(lr){2-4} \cmidrule(lr){5-7}
		{rMSE}   & {\footnotesize {\picotwo 0.009~(0.003)}} & {\footnotesize {\picotwo 0.006~(0.002)}} & {\footnotesize {\picotwo 0.005~(0.002)}} & 
        {\footnotesize {\picotwo 0.008~(0.004)}} & {\footnotesize {\picotwo 0.006~(0.004)}} & {\footnotesize {\picotwo 0.004~(0.002)}} \\
                 {$\text{rMSE}(S_0)$} & {\footnotesize {\picotwo 0.117~(0.073)}} & {\footnotesize {\picotwo 0.093~(0.032)}} & {\footnotesize {\picotwo 0.067~(0.023)}} & 
                 {\footnotesize {\picotwo 0.101~(0.035)}} & {\footnotesize {\picotwo 0.089~(0.045)}} & {\footnotesize {\picotwo 0.069~(0.024)}} \\
		{CP}   & {\footnotesize {\picotwo 1.000~(0.000)}} & {\footnotesize {\picotwo 1.000~(0.000)}} & {\footnotesize {\picotwo 1.000~(0.000)}} & {\footnotesize {\picotwo 1.000~(0.000)}} & {\footnotesize {\picotwo 1.000~(0.000)}} & {\footnotesize {\picotwo 1.000~(0.000)}} \\ 
		{Sens} & {\footnotesize {\picotwo 0.875~(0.214)}} & {\footnotesize {\picotwo 0.948~(0.101)}} & {\footnotesize {\picotwo 0.695~(0.188)}} & {\footnotesize {\picotwo 0.972~(0.082)}} & {\footnotesize {\picotwo 0.993~(0.030)}} & {\footnotesize {\picotwo 0.921~(0.115)}} \\  
		{Spec}  & {\footnotesize {\picotwo 1.000~(0.000)}} & {\footnotesize {\picotwo 1.000~(0.000)}} & {\footnotesize {\picotwo 1.000~(0.000)}} & {\footnotesize {\picotwo 1.000~(0.000)}} & {\footnotesize {\picotwo 1.000~(0.000)}} & {\footnotesize {\picotwo 1.000~(0.000)}} \\ 
		{MCC}&  {\footnotesize {\picotwo 0.932~(0.078)}} & {\footnotesize {\picotwo 0.972~(0.059)}} & {\footnotesize {\picotwo 0.825~(0.122)}} & {\footnotesize {\picotwo 0.985~(0.046)}} & {\footnotesize {\picotwo 0.996~(0.016)}} & {\footnotesize {\picotwo 0.956~(0.070)}} \\
          \cmidrule(lr){2-4} \cmidrule(lr){5-7}
\end{tabular}}  \label{table:results}
\end{table}

Table \ref{table:results} reports our results averaged across the 100 replications for the different scenarios. {\pico In all of the simulation settings, we see that as {\picotwo $p$ and $n$} increased with $p$ growing faster than $n$, the average rMSE \emph{decreased} for both the one-step estimator and the two-step estimator. This provides numerical evidence of the posterior contraction results in Theorems \ref{one-step_contraction_rate} and \ref{thm_D2}.} Table \ref{table:results} also shows that the two-step method consistently had a lower average overall rMSE than the one-step method, especially in higher dimensions when {\picotwo $p=2000$}. {\pico This corroborates Theorem \ref{thm_D2}, which implies that when $p > n$, the two-step Mt-MBSP estimator should have a \emph{lower} estimation error than the one-step estimator.} {\pico Furthermore, Table \ref{table:results} also indicates that not only was the average overall rMSE lower for two-step Mt-MBSP, but so was the average rMSE($S_0$). This demonstrates that the two-step estimator's improvement over one-step Mt-MBSP could \emph{not} be solely attributed to the fact that the one-step estimator was not exactly sparse. Instead, two-step Mt-MBSP \emph{also} estimated the true nonzero entries in $\mathbf{B}_0$ with greater accuracy than one-step Mt-MBSP.}

{\picotwo Table \ref{table:results} also shows that the variable selection performance of the two-step estimator was consistently better than that of the one-step estimator, with higher average sensitivity and higher average MCC for all combinations of $(p, n, s_0)$. This is because the one-step thresholding rule \eqref{a0} tended to be more conservative and selected fewer variables, leading to more false negatives for one-step Mt-MBSP.} This {\picotwo provides numerical} evidence of Theorem \ref{thm_C}, which proved that two-step Mt-MBSP has the sure screening property with a well-chosen threshold $\gamma$ in \eqref{setAn}. Since Step 1 of two-step Mt-MBSP typically resulted in a set of variables $\mathcal{J}_n$ that properly contained the true set of non-null variables $S_0$, the two-step approach was {\emph less} likely to misclassify true signals than one-step Mt-MBSP. Overall, our results suggest that adding the variable screening step in our two-step approach is useful for improving the final estimation and variable selection performance.

Scenario 6 (all counts) is a particularly interesting case study. {\picotwo In this scenario, most} of the true nonzero signals in $\mathbf{B}_0$ were relatively weak. Table \ref{table:results} shows that the sensitivity {\picotwo and MCC were} considerably higher for two-step Mt-MBSP than for one-step Mt-MBSP {\picotwo in Scenario 6 when $p=2000$}. This suggests another advantage of our two-step approach. It appears as though two-step Mt-MBSP is able to select \emph{more} weak signals which are erroneously classified as null signals by the selection rule \eqref{a0} for one-step Mt-MBSP. By choosing an initial candidate set \eqref{setJn} with \emph{more} variables whose posterior credible intervals actually contain zero in Step 1, the two-step model is then able to detect these weak signals in Step 2.

{\pico \begin{remark}
    Theorem \ref{thm_inconsistent} implies that the one-step Mt-MBSP estimator is inconsistent when $p$ grows exponentially fast with $n$. We found it challenging to construct an artificial simulation demonstrating this posterior inconsistency. Our simulations in this section still corroborate our theoretical findings that when $p \gg n$, the two-step Mt-MBSP method improves upon the one-step Mt-MBSP method through lower estimation error and better variable selection. 
\end{remark}}

\subsection{Application with small $p$: Chronic kidney disease data}
\label{applications1}

We used the chronic kidney disease (CKD) dataset analyzed in \cite{EbiaredohSEM2022} to evaluate the proposed method. 
CKD is a disease in which the kidneys are unable to filter blood properly. As a result, excess fluids and waste in the blood remain in the body and may lead to other health problems such as heart disease and stroke. 
This dataset contains medical records from $n=400$ patients, 250 of whom had CKD. There are $24$ independent variables (11 numerical and 13 nominal) and a class variable (CKD status). These 24 variables are described in the Appendix \ref{S3:RealData}.

We studied whether the specific gravity (SG) (a continuous outcome) and the CKD status (a binary outcome) jointly vary with the other 23 attributes. SG is the concentration of all chemical particles in urine and is commonly used to quantify kidney function. We modeled the SG as a Gaussian response $y_1$ and the CKD status as a binary response $y_2$ and fit the Mt-MBSP model with the other $p=23$ variables as covariates. 
Our method selected the following three variables as being significantly associated with the CKD status: albumin (AL), hemoglobin (HEMO), and diabetes mellitus (DM) and only HEMO is significant for the SG. The study by \cite{EbiaredohSEM2022} also ranked AL and HEMO as the top two features. 

\begin{figure}[t!]
	\centering
	\includegraphics[scale=0.45]{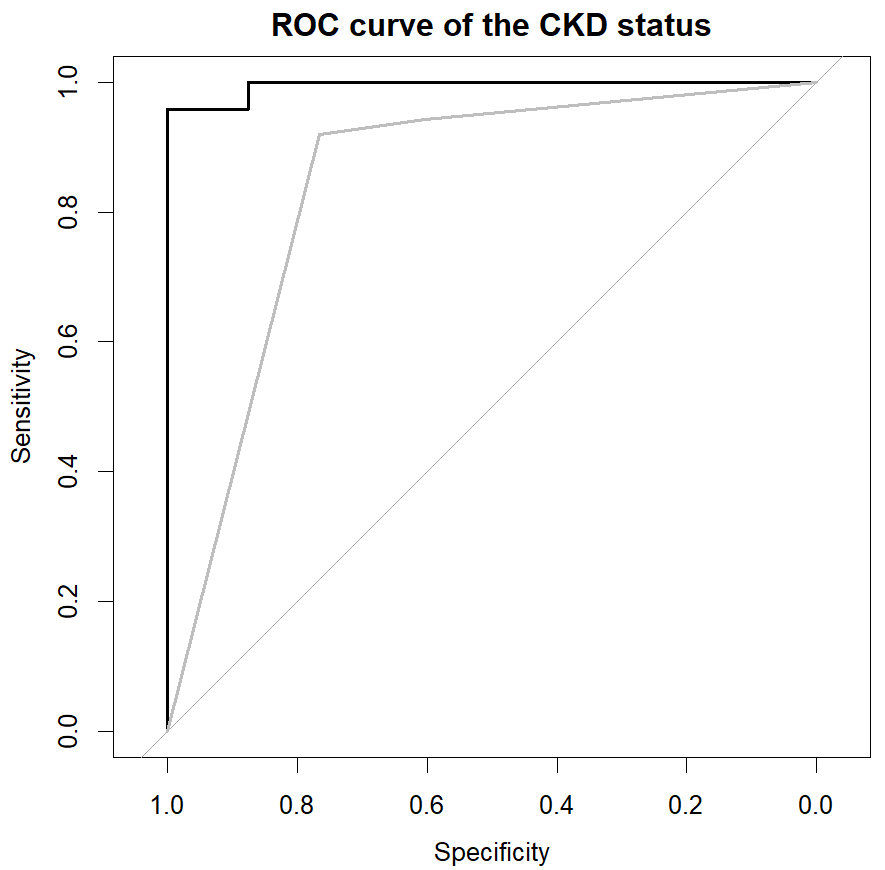} 
	\caption{Plots of the five-fold average ROC curves using the proposed method (black) and the marginal logistic model (gray).}
	\label{fig:CKD_ROCs}
\end{figure}

Using five-fold cross-validation (CV), we compared the results from Mt-MBSP to the results from fitting two separate marginal models,
\begin{equation*}
	y_1 = \bm{\beta}^\top_1 \mathbf{x} +\bm{\varepsilon}\mbox{~~and~~}
	P(y_2=y \mid x ) =\frac{\exp\{y(\bm{\beta}^\top_2 \mathbf{x} )\}}{1+\exp(\bm{\beta}^\top_2 \mathbf{x})},~y\in\{0,1\}.
\end{equation*}
We evaluated the average five-fold CV rMSE for the SG ($y_1$) and average five-fold CV area under the receiver operator characteristic (ROC) curve (AUC) for the CKD status ($y_2$). For $y_1$, the rMSE for the Mt-MBSP model ($0.876$)  was lower than the rMSE for the marginal linear regression model ($0.974$). For $y_2$, Mt-MBSP also had a much higher out-of-sample AUC ($0.995$) than the AUC for the marginal logistic model ($0.731$), indicating far better discriminative ability (see Figure \ref{fig:CKD_ROCs}). More detailed results from our analysis are in Appendix \ref{S3:RealData}. {\pico This example shows that the Mt-MBSP method performs well when $p<n$.}

\subsection{Application with very large $p$: Carcinomas data} \label{applications2}

{\pico Carcinoma is a type of cancer that originates in cells that make up the skin or  tissue lining organs. Correctly classifying different carcinomas based on their primary anatomical site (e.g. prostate, liver, etc.) is an important problem in medicine, since this allows clinicians to formulate optimal treatment plans for cancer patients \cite{su2001molecular}. In this section, we analyze the carcinomas dataset (U95a GeneChip) in \cite{su2001molecular}. This dataset contains a total of $n=174$ samples from 11 different carcinoma types: prostate, bladder/ureter, breast, colorectal, gastroesophagus, kidney, liver, ovary, pancreas, lung adenocarcinomas, and lung squamous cell carcinoma. The distribution of the $n$ samples across these 11 respective categories is as follows: 26, 8, 26, 23, 12, 11, 7, 27, 6, 14, and 14.  Collectively, these carcinomas account for about 70\% of all cancer-related deaths in the United States \cite{su2001molecular}.

In addition to the carcinoma type, each sample in our dataset contains $p=9183$ gene expression levels, with a maximum hybridization intensity of $\leq 200$ in at least one sample. All hybridization intensity values below 20 were adjusted to 20, followed by log transformation of the data. Our objectives are to: 1) classify carcinoma type based on gene expression levels, and 2) identify genes that are significantly associated with the different types of carcinoma.   

We applied the proposed two-step Mt-MBSP method to this carcinomas dataset with $q=11$ binary responses (one for each category of carcinoma) and all $p=9183$ genes. Thus, we needed to estimate a total of $pq = 101{,}002$ unknown regression coefficients. Through experimentation, we found that tuning the threshold $\gamma$ in the two-step algorithm from $\{ 0.010, 0.011, \ldots, 0.020 \}$ yielded the best results. We ran the two-step algorithm for 2000 MCMC iterations in each stage, discarding the first 1000 samples as burnin. To ensure that we had run the Gibbs sampler for enough iterations, we used the \textsf{R} package \texttt{mcmcse} \cite{mcmcseRpackage} to estimate the effective sample size (ESS) and Monte Carlo standard error (MCSE) for the posterior median of all $100{,}002$ entries in the $p \times q$ regression coefficients matrix $\mathbf{B}$. Step 1 had an average ESS of 789.94 and an average MCSE of 0.0063 across all entries of $\mathbf{B}$, while Step 2 had an average ESS of 580.37 and average MCSE of 0.0043. These MCMC diagnostics were reassuring.

\begin{table}[t!]
\caption{{\pico Results from our carcinomas data analysis. Carcinomas of the bladder/ureter, kidney, and pancreas are not shown because none of the genes was found to be associated with them. }}
\begin{center}
\footnotesize
\begin{tabular}{ l | l}
\hline 
   Carcinoma & Significant genes \\
   \hline 
  prostate & 3951, 7777 \\
  breast & 2062, 3951, 7777 \\
  colorectal & 2062, 7605 \\
 gastroesophagus & 5166 \\
 liver & 1893 \\
 ovary & 2516, 5166, 7777 \\
 lung adenocarcinomas  & 1848 \\
 lung squamous & 3951, 5166 \\ \hline
\end{tabular}
\end{center}
\label{carcinoma:genes}
\end{table}

Our method identified eight significant genes enumerated as 1848, 1893, 2062, 2516, 3951, 5166, 7065, and 7777 in \cite{su2001molecular}. Interestingly, Mt-MBSP found that each of these genes was only significantly associated with some but not all of the carcinomas. For example, gene 7777 was found to be significantly associated with  cancers of the prostate, breast, and ovary but not with the other eight carcinomas. In addition, carcinomas of the bladder/ureter, kidney, or pancreas were not found to be significantly associated with any of the 9183 genes. Table~\ref{carcinoma:genes} lists the results for the eight carcinomas that were found to be significantly associated with at least one of the genes in the dataset.

We also examined the ROC curves and AUC for the eight types of carcinomas exhibiting significant regression coefficients. Our results are shown in  Figure~\ref{fig:carcin_ROCs} and demonstrate exceptional classification performance. For comparison, \cite{cai2007selecting} utilized SVMs and had to identify 10 significant genes for \emph{each} class (or 110 total genes) in order to achieve comparable classification performance as Mt-MBSP. Despite choosing a rather sparse model (eight out of 9183 genes), our method still managed to achieve high classification accuracy. } 

\begin{figure}[t!]
\centering
\includegraphics[scale=0.57]{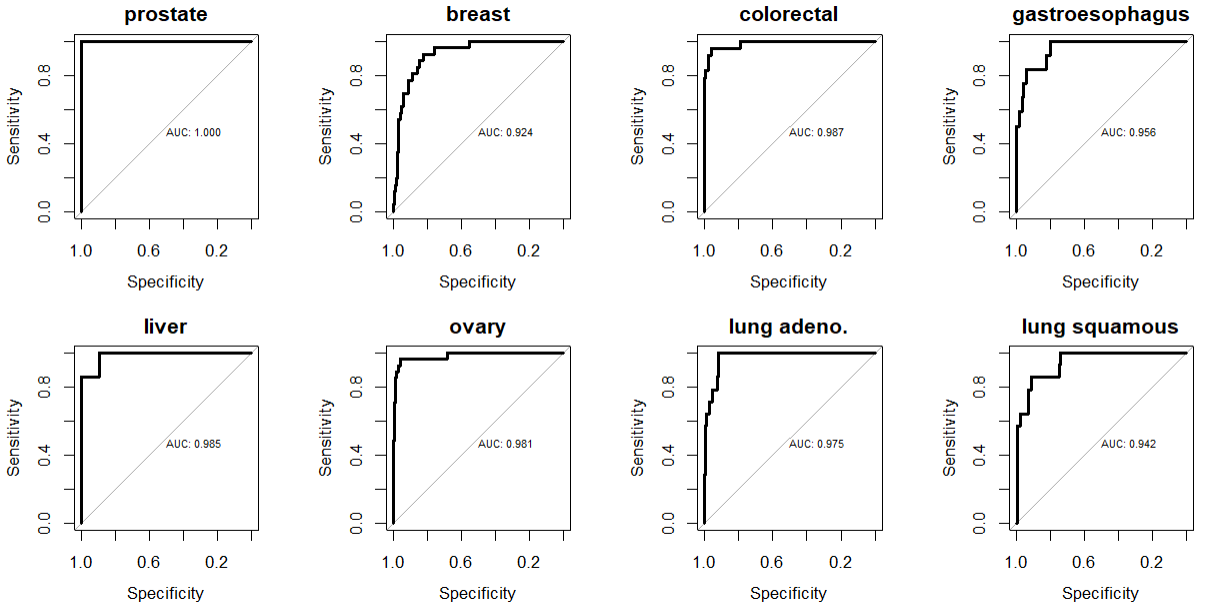} 
\caption{{\pico The ROC curves show that our method classifies eight types of carcinomas very well. In the figure, ``lung adenocarcinomas'' is abbreviated as ``lung adeno.'' Carcinomas of the bladder/ureter, kidney, and pancreas are not shown because none of the genes in the dataset was found to be associated with them.}}
\label{fig:carcin_ROCs}
\end{figure}


\section{Discussion} \label{Discussion}

We have introduced the Mt-MBSP method for joint modeling of mixed-type (i.e. continuous and discrete) responses. Our method accommodates correlations between the mixed responses through latent random effects. By employing a P\'{o}lya-gamma data augmentation approach \citep{polson2013bayesian} for the discrete responses, Mt-MBSP can be implemented with an exact Gibbs sampling algorithm where all the full conditionals are available in closed form. GL priors are placed on the regression coefficients to model sparsity and facilitate variable selection. 

Until now, theoretical results for Bayesian \textit{mixed-type} multivariate regression models have been unavailable. We have taken a step towards resolving this gap in the literature. We first derived the posterior contraction rate for the one-step Mt-MBSP model (i.e. the model fit to all $p$ variables without variable screening). We have also shown that subexponential growth of $p$ with $n$ is \textit{necessary} for posterior consistency of the one-step method. To overcome this limitation, we introduced a novel two-step algorithm that screens out a large number of variables in the first step. We established the sure screening property of our two-stage approach and the posterior contraction rate of the two-step estimator when $p$ is allowed to grow \textit{exponentially} with $n$. The two-step estimator has a sharper contraction rate than the one-step estimator if $p$ diverges faster than $n$, making it especially suitable for analysis of mixed-typed responses when $p>n$. 
The asymptotic regime of $\log p = O(n^{\alpha}), \alpha \geq 1$, has attracted theoretical interest \citep{LahiriAOS2021}, in part due to the abundance of high-throughput biological and genomic data. However, this regime has not been studied before in the Bayesian framework. Our paper makes an advancement in this important direction.

{\pico In this article, we have not considered situations where the response dimension $q$ could be large or divergent with $n$. Therefore, the choice of an inverse-Wishart prior \eqref{IWpriorSigma} for the covariance matrix $\bm{\Sigma}$ is adequate. However, the ``large $q$'' scenario also arises in practice, for example, in microbiome differential abundance analysis where there are large number of environmental outcomes \citep{chen2013variable,banerjee2022adaptive}. In these scenarios, we may need to impose additional sparsity assumptions on the precision matrix $\bm{\Sigma}^{-1}$ \citep{deshpande2019simultaneous, LiDattaCraigBhadra2021} or perform \textit{response} variable selection \citep{an2017simultaneous,kharesu2022}, in addition to predictor variable selection. The joint mean–covariance estimation procedure of \cite{LiDattaCraigBhadra2021} would be especially beneficial to use in the ``large $q$'' regime, because this method preserves conditional conjugacy. We will investigate the interesting problem of ``large $p$ \emph{and} large $q$'' methodologically and theoretically in future work.} We also plan to examine other important theoretical issues such as the asymptotic coverage of the posterior credible sets or a Bernstein von-Mises type theorem for the shape of the posterior in ultra high-dimensional settings. The present article intends to act as a springboard for future theoretical developments in Bayesian mixed-response models.

{\picotwo Our methodology and theoretical results assume row sparsity in the regression coefficients matrix $\mathbf{B}$. Although this is a common assumption in the literature, it may not hold in practice. For example, many of the $p$ rows could be non-null but each nonzero row only contains a few nonzero entries. In this scenario, if $s_0 > n$ (i.e. the number of nonzero rows is large), then we do not expect that the Mt-MBSP method will perform well. To accommodate more general sparsity structures, we can replace the $q$-variate priors \eqref{Mt-MBSP} on the rows of $\mathbf{B}$ with \emph{univariate} GL priors on the \emph{individual} regression coefficients $b_{jk}$ as
\begin{align} \label{univariateGL}
    b_{jk} \mid \xi_{jk} \sim \mathcal{N} (0, \tau \xi_{jk}),~~ \xi_{jk} \sim p(\xi_{jk}),~~ j = 1, \ldots, p,~ k = 1, \ldots, q,
\end{align}
similarly as \cite{LiDattaCraigBhadra2021, deshpande2019simultaneous}. Under modified assumptions, e.g. $\sum_{j=1}^{p} s_{0j} = o(n / \log p)$ where $s_{0j}$ is the number of non-null elements in the $j$th row of $\mathbf{B}_0$, Theorem \ref{one-step_contraction_rate} can be suitably modified to obtain the posterior contraction rate of $[ (\log p) \sum_{j=1}^{p} s_{0j} ] /n$. Likewise, we can modify the two-step algorithm in Section \ref{twostep} so that Step 1 first chooses the set $\mathcal{A}_n$ as
	\begin{equation*} 
		{\cal A}_n= \left\{~
		(j, k) \mid\mbox{
			$q_{0.025}(b_{jk}) > -\gamma$~~or~~ 
			$q_{0.975}(b_{jk}) <\,\gamma$
		}
		\right\},
\end{equation*}
and then chooses the $K_n = \min \{ n-1, |\mathcal{A}_n| \}$ variables with the largest magnitude posterior medians $|q_{0.5}(b_{jk})|$ as the final candidate set $\mathcal{J}_n$. We could then still obtain a sure screening property in a modified version of Theorem \ref{thm_C}. Finally, by modifying Theorem \ref{thm_D2}, the posterior contraction rate for this modified two-step estimator would then be $[ (\log K_n) \sum_{j=1}^{p} s_{0j} ]/n$. We anticipate that this extension to Mt-MBSP would work well under more general sparsity structures that may not entail row sparsity. This should be explored in future work.} 

\section*{Acknowledgments}

{\pico We are grateful to the Associate Editor and two anonymous reviewers whose thoughtful feedback helped us greatly improve our paper.}

\appendix

\renewcommand{\thesection}{\Alph{section}}
\renewcommand{\theproposition}{\thesection.\arabic{proposition}}
\renewcommand{\thetheorem}{\thesection.\arabic{theorem}}
\renewcommand{\thelemma}{\thesection.\arabic{lemma}}
\renewcommand{\theequation}{\thesection.\arabic{equation}}
\renewcommand{\thetable}{\thesection.\arabic{table}}	

		\setcounter{table}{0}

\section{One-step Gibbs sampler for Mt-MBSP} \label{S1:Gibbs}

\setcounter{equation}{0}

{\picotwo Let} $\mathbf{Z}_n=(\mathbf{z}_1,\dots,\mathbf{z}_n)^\top$ and $\mathbf{W}_n=(\bm{\omega}_1,\dots,\bm{\omega}_n)^\top$ {\picotwo denote} $n\times q$ augmented data matrices with rows $\mathbf{z}_i$ and $\bm{\omega}_i$, $i = 1, \ldots, n$, where $z_{ik}$ and $\omega_{ik}, k = 1, \ldots, q$, are defined as in \eqref{p4}, and $\mathbf{X}_n = (\mathbf{x}_1, \ldots, \mathbf{x}_n)^\top$ is the $n \times p$ design matrix. 
Let $\mathbf{U}_n=(\mathbf{u}_1,\dots, \mathbf{u}_n)^\top$ be the $n \times q$ matrix where the $\mathbf{u}_i$'s are the random effect vectors defined in \eqref{blatent}. 
{\picotwo Since $\mathbf{Y}_n$ is determined by $\mathbf{Z}_n$ and $\mathbf{W}_n$, the joint density of $(\mathbf{Y}_n, \mathbf{Z}_n, \mathbf{W}_n, \mathbf{U}_n, \mathbf{B}, \bm{\Sigma})$ is the same as the joint density for $(\mathbf{Z}_n, \mathbf{W}_n, \mathbf{U}_n, \mathbf{B}, \bm{\Sigma})$.}

Let $\bm{\beta}_k$ denote the $k$th column of $\mathbf{B}$ so that $\mathbf{B} = (\bm{\beta}_1, \ldots, \bm{\beta}_q)$. Meanwhile, let $\bm{\nu} = (\nu_1, \ldots, \nu_p)^\top$ and $\bm{\eta} = (\eta_1, \ldots, \eta_p)^\top$, and let ${\cal GIG}(a,b,c)$ denote the generalized inverse Gaussian distribution with density function $f(x ; a, b, c) \propto x^{c-1}e^{-(a/x+bx)/2}$.
Assuming the prior \eqref{Mt-MBSP}-\eqref{global_local_prior} for $\mathbf{B}$, the TPBN prior \eqref{TPBN} for the local scale parameters $p(\xi_j)$'s in \eqref{global_local_prior}, and the inverse-Wishart prior \eqref{IWpriorSigma} for $\bm{\Sigma}$, we have the following Gibbs sampling algorithm for Mt-MBSP.

\begin{enumerate}
	\item Initialize $\mathbf{B}, \mathbf{W}_n, \mathbf{U}_n, \bm{\nu}, \bm{\eta}, \bm{\Sigma}$.
	\item Repeat for a large number of iterations:
	\begin{enumerate}[label=\roman*.]
		\item Sample $\mathbf{B} \mid \mathbf{Z}_n, \mathbf{W}_n, \mathbf{U}_n, \bm{\nu}, \bm{\eta}, \bm{\Sigma}$ as 
		\begin{equation*}
			\bm{\beta}_k \sim {\cal N}_{p}\left(
			\bm{\Delta}_k^{-1} 
			\sum^n_{i=1} (z_{ik} - u_{ik}) \omega_{ik} \mathbf{x}_i,~\bm{\Delta}_k^{-1}
			\right),~~k=1, \ldots, q,
		\end{equation*}
		where
		\begin{equation*}
			\bm{\Delta}_k = {\rm diag}(1/\nu_1,\dots,1/\nu_p)+\sum^n_{i=1}
			\omega_{ik} \mathbf{x}_i \mathbf{x}_i^\top.
		\end{equation*}
		If $p>n$, then the fast sampling method of \cite{bhattacharya2016fast} is used to sample from the full conditionals for the columns of $\mathbf{B}$ in $\mathcal{O}(n^2p)$ rather than in $\mathcal{O}(p^3)$ time.
		\item Sample $\mathbf{U}_n \mid \mathbf{Z}_n, \mathbf{W}_n, \mathbf{B}, \bm{\nu}, \bm{\eta}, \bm{\Sigma}$ as
		\begin{equation*}
			\mathbf{u}_i \sim {\cal N}_{q}\left( \bm{\Psi}_i^{-1} \bm{\Omega}_i \left( \mathbf{z}_i - \mathbf{B}^\top \mathbf{x}_i \right), \bm{\Psi}_i^{-1} \right),~~i = 1, \ldots, n,
		\end{equation*}
		where 
		\begin{equation*}
			\bm{\Psi}_i = \bm{\Omega}_i + \bm{\Sigma}^{-1}.
		\end{equation*}
		\item Sample $\mathbf{W}_n \mid \mathbf{Z}_n, \mathbf{U}_n, \mathbf{B}, \bm{\nu}, \bm{\eta}, \bm{\Sigma}$ by sampling the $(i,k)$th entry $\omega_{ik}$ as
		\begin{equation*}
			\omega_{ik}=
			\left\{\begin{array}{ll}
				1 &\mbox{if $y_{ik}$ is continuous as in \eqref{gaussianresponse},}\\[.1in]
				{\cal PG}\left( f_{ik}^{(2)},~ (\mathbf{X}_n \mathbf{B}+\mathbf{U}_n)_{ik}\right)&
				\mbox{if $y_{ik}$ is discrete as in \eqref{polson_model},}
			\end{array}\right.
		\end{equation*}
		where $f_{ik}^{(2)}$ is defined as in \eqref{polson_model}.
        \item Update each $(i,k)$th entry in $\mathbf{Z}_n$, i.e. $z_{ik}$, as in \eqref{p4}.
		\item Sample $\bm{\nu} \mid \mathbf{Z}_n, \mathbf{W}_n, \mathbf{U}_n, \mathbf{B}, \bm{\eta}, \bm{\Sigma}$ as 
		\[ \nu_j \sim {\cal GIG}(\| \mathbf{b}_j \|^2, 2\eta_j,u-q/2),~~~j = 1,\dots,p.
		\]
		\item Sample $\bm{\eta} \mid \mathbf{Z}_n, \mathbf{W}_n, \mathbf{U}_n, \mathbf{B}, \bm{\nu}, \bm{\Sigma}$ as 
		\[ \eta_j \sim
		{\cal G}(a,~\tau+\nu_j),~~~j = 1, \ldots, p.
		\]
		\item
		Sample $\bm{\Sigma} \mid \mathbf{Z}_n, \mathbf{W}_n, \mathbf{U}_n, \mathbf{B}, \bm{\eta}, \bm{\nu}$ 
		\[ 
		\bm{\Sigma} \sim {\cal IW}(n+d_1,~\mathbf{U}_n^\top \mathbf{U}_n+d_2 \mathbf{I}_q).
		\]
	\end{enumerate}
\end{enumerate}

\section{Additional details and results for the chronic kidney disease (CKD) application} \label{S3:RealData}

Table~\ref{CKD24var} lists all 25 variables in the chronic kidney disease (CKD) application from Section \ref{applications1}. The two response variables that we considered were specific gravity (a continuous outcome abbreviated as SG) and CKD status (a binary outcome for ``yes'' or ``no''). The other 23 variables were used as predictors in the Mt-MBSP model. 

In Table \ref{tab:ckd}, we provide the posterior median estimates for the parameter $\mathbf{B} = (\bm{\beta}_1, \bm{\beta}_2)$ under the Mt-MBSP model. Here, $\bm{\beta}_1$ denotes the parameters corresponding to SG ($y_1$), while $\bm{\beta}_2$ denotes the parameters corresponding to CKD status ($y_2$).

\section{Proofs} \label{S4:Proofs}

\setcounter{equation}{0}

Here, we give the proofs of all the theoretical results from Sections \ref{theory} and \ref{twostep}. Propositions \ref{PropositionA1} and \ref{PropositionA2} are suitably adapted from Lemmas 1 and 2 of an unpublished technical report \citep{MBSPcorrigendum} written by the same authors of this manuscript. 

{\picotwo Throughout this section, $\mathbf{Z}_n=(\mathbf{z}_1,\dots,\mathbf{z}_n)^\top$ and $\mathbf{W}_n=(\bm{\omega}_1,\dots,\bm{\omega}_n)^\top$ denote the $n\times q$ matrices with rows $\mathbf{z}_i$ and $\bm{\omega}_i$, $i = 1, \ldots, n$, where $z_{ik}$ and $\omega_{ik}, k = 1, \ldots, q$, are defined as in \eqref{p4}.}

\begin{table}[t!] 
	\centering
	\caption{Estimated regression coefficients $\mathbf{B} = (\bm{\beta}_1, \bm{\beta}_2)$ in the Mt-MBSP model fit to the CKD dataset. Here, $\bm{\beta}_1$ is the vector of parameters associated with the Gaussian response $y_1$ (SG), while $\bm{\beta}_2$ is the vector of parameters associated with the binary response $y_2$ (CKD status).}~~\\
		\begin{tabular}{ccc|ccc}
			Variables & $\bm{\beta}_1$
			& $\bm{\beta}_2$ & Variables &$\bm{\beta}_1$ 
			& $\bm{\beta}_2$  \\[.1in]
                BGR&~0.0000&-0.0000&RBC&0.0000& ~0.0000\\
                BU&-0.0000&~0.0000&SU &{0.0000}&~0.0000\\
                SOD&~0.0001&{ -0.0002}&PC  &~0.0000&-0.0000\\
                SC &{~-0.0017}&{-0.0097} &PCC  &0.0000&~0.0000\\
                POT&~0.0006&{~-0.0017} &BA &0.0000&-0.0000\\
                PCV&-0.0011&{~0.0009}&HEMO&{~0.7794}&{-3.2464}\\
                RBCC  &{~0.0019}&{-0.0050}&DM&{-0.2162}&{~3.1619}\\
                GRF&~0.0000&~0.0000&CAD&~0.0000&~0.0000\\
			AGE &-0.0001&~0.0001 &APPET&~0.0000& ~0.0000\\
			BP  &~0.0000&~0.0000 &PE&-0.0000& ~0.0000\\
                HTN&~0.0000&~0.0000&ANE&~0.0000& ~0.0000\\
			AL  &{-0.1641}&{1.5406}&&&
	\end{tabular} \label{tab:ckd}
\end{table}

\begin{table}[t!]
	\caption{The 25 variables in the CKD dataset analyzed in Section \ref{applications1}. In our analysis, we used a continuous variable specific gravity (SG) and a binary variable CKD status (CKD) as the response variables for our analysis. The other 23 covariates were used as the covariates.}
	\vspace{.1in}
	\centering
		\begin{tabular}{ll}
			\textbf{Variable (type)} & \textbf{Abbreviation (description)}\\
			\hline
			Age(numerical)&
			AGE in years\\
			Blood Pressure(numerical)&
			BP in mm/Hg\\
			Specific Gravity(nominal)&
			SG $\in \{1.005,1.010,1.015,1.020,1.025\}$\\
			Albumin(nominal)&
			AL $\in \{0,1,2,3,4,5\}$\\
			Sugar(nominal)&
			SU $\in \{0,1,2,3,4,5\}$\\
			Red Blood Cells(nominal)&
			RBC $\in$ \{normal, abnormal\}\\
			Pus Cell (nominal)&
			PC $\in$ \{normal, abnormal\}\\
			Pus Cell clumps(nominal)&
			PCC $\in$ \{present,not present\}\\
			Bacteria(nominal)&
			BA $\in$ \{present, not present\}\\
			Blood Glucose Random(numerical)&
			BGR in mgs/dl\\
			Blood Urea(numerical)&
			BU in mgs/dl\\
			Serum Creatinine (numerical)&
			SC in mgs/dl\\
			Sodium (numerical)&
			SOD in mEq/L\\
			Potassium (numerical)&
			POT in mEq/L\\
			Hemoglobin (numerical)&
			HEMO in gms\\
			Packed Cell Volume (numerical)&PCV $\in [0,1]$\\
			Glomerular Filtration Rate (numerical)&
			  GFR in ml/min\\
			Red Blood Cell Count (numerical)&
			RC in millions/cmm\\
			Hypertension (nominal)&
			HTN $\in$ \{yes, no\}\\
			Diabetes Mellitus (nominal)&
			DM $\in$ \{yes, no\}\\
			Coronary Artery Disease (nominal)&
			CAD $\in$ \{yes, no\}\\
			Appetite(nominal)&
			APPET $\in$ \{good, poor\}\\
			Pedal Edema(nominal)&
			PE $\in$ \{(yes, no\}\\
			Anemia(nominal)&
			ANE $\in$ \{yes, no\}\\
			CKD status (nominal)&
			CKD $\in$ \{yes, no\}\\
			\hline
	\end{tabular}
	\label{CKD24var}
\end{table}

\subsection{Preliminary Lemmas and Propositions}

\begin{proposition} 
	Assume that a positive measurable function $L(x)$ is slowly varying, i.e. for each fixed $a>0$, $L(ax) / L(x) \rightarrow 1$ as $x \rightarrow \infty$. Then for each $\varepsilon \in (0, 1)$ and $a > 1$, there exists $x_0$ such that
	\begin{align*}
		\widetilde{k}_1 x^{\log(1-\varepsilon) / \log(a)} \leq L(x) \leq \widetilde{k}_2 x^{\log(1+\varepsilon)/\log(a)} 
	\end{align*}
	whenever $x > x_0$, for some constants $\widetilde{k}_1, \widetilde{k}_2 > 0$.\label{PropositionA1}
\end{proposition}

\begin{proof}[Proof of Proposition \ref{PropositionA1}]
	By the definition of a slowly varying function, for each $\varepsilon \in (0,1)$ and $a > 1$, there exists $u_0$ so that $\lvert L(au) / L(u) - 1 \rvert < \varepsilon$ whenever $u > u_0$. Thus, we have
	\begin{align*}
		L(u_0)(1- \varepsilon) \leq L(au_0) \leq L(u_0) (1+\varepsilon).
	\end{align*}
	By induction, for all $k \in \mathbb{N}$,
	\begin{equation} \label{slowlyvarying}
		L(u_0) (1-\varepsilon)^k \leq L(a^k u_0) \leq L(u_0) (1+\varepsilon)^k.
	\end{equation}
	Note that $a^k u_0 > u_0$ for all $k \in \mathbb{N}$ since $a > 1$. Take $x = a^k u_0$ so that $k = \log(x/u_0) / \log(a)$. We can then rewrite \eqref{slowlyvarying} as
	\begin{align*}
		L(u_0)(1-\varepsilon)^{\log(x/u_0)/\log(a)} \leq L(x) \leq L(u_0)(1+\varepsilon)^{\log(x/u_0)/\log(a)},
	\end{align*}
	and thus,
	\begin{align*}
		L(u_0) \left( \frac{x}{u_0} \right)^{\frac{\log(1-\varepsilon)}{\log(a)}} \leq L(x) \leq L(u_0)  \left( \frac{x}{u_0} \right)^{\frac{\log(1+\varepsilon)}{\log(a)}}.
	\end{align*}
	~
\end{proof}

\begin{proposition} 
	Suppose that $\xi$ has the hyperprior density $p(\xi)$ in \eqref{global_local_prior}, i.e. $p(\xi) = K \xi^{-a-1} L(\xi)$, where $L(\cdot)$ is slowly varying. {\pico Let  $a_n =O(n^t)$ and $\log\,b_n \asymp n^{s}$ be increasing sequences of $n$, where $t$ and $s$ are finite positive numbers.
 } Then
	{\pico \begin{align*}
		P ( \xi > a_n b_n) > \exp(-C\, \log b_n ),
	\end{align*}}
	for some finite $C > 0$.\label{PropositionA2}
\end{proposition}
{\pico
\begin{proof}[Proof of Proposition \ref{PropositionA2}]
	By Proposition \ref{PropositionA1}, there exists finite $b > 0$ and $K_1 > 0$ so that 
	\begin{align*}
		P (\xi > a_n b_n) & = \int_{a_n b_n}^{\infty} K u^{-a-1} L(u) du \\
		& \geq \int_{a_n b_n}^{\infty} K u^{-a-1} K_1 u^{-b} du \\
		& = \frac{KK_1}{a+b} (a_n b_n)^{-a-b} \\
		& = \exp \left\{ - (a+b) \log b_n - (a+b) \log a_n   + \log(KK_1 / (a+b)) \right\} \\
		& > \exp(-C\, \log b_n )
	\end{align*}
	for a sufficiently large $C > 0$. In the last line of the display, we used the fact that $a_n =O(n^t)$ and $\log b_n \asymp n^{s}$. 
\end{proof}
}

{\picotwo 
\begin{proposition}\label{prop:NB}
Assume $Y_1,\dots, Y_n$ are i.i.d random variables from a negative binomial distribution  ${\rm NB}(r,p)$ with PMF,
\begin{equation*}
P(Y=y\mid \theta,~r) = \frac{\Gamma(y+r)}{y!~\Gamma(r)}
	(1-p)^{y} p^r,
\end{equation*}
where $0<r<\infty$. Then for any $t>0$,
\[
P(\max_{1\leq i\leq n}Y_i > n^t,\mbox{~a.s.})=0 ~~{ as }~~ n \rightarrow \infty.
\]
\end{proposition}
\begin{proof}
Let $[x]$ be the integral part of $x \in \mathbb{R}$.  As $n$ is large, we obtain
\begin{footnotesize}
\begin{eqnarray}
P(\max_{1\leq i\leq n}Y_i > n^t)&\leq& 
n\,P(Y_1 > n^t)\notag\\
&=&n\sum^\infty_{\ell=[n^t]+1}
\frac{\Gamma(y+r)}{y!~\Gamma(r)}
(1-p)^\ell p^r\notag\\
&=&n\sum^\infty_{\ell=[n^t]+1}
\frac{
\sqrt{2\pi(\ell+r-1)}
\left(\frac{\ell+r-1}{e}\right)^{(\ell+r-1)}}{
\sqrt{2\pi\ell}
\left(\frac{\ell}{e}\right)^{\ell}
\sqrt{2\pi(r-1)}
\left(\frac{r-1}{e}\right)^{r-1}
}(1-p)^\ell p^r(1+o(1))\notag\\
&=&
\frac{1}{\sqrt{2\pi}}
e^{r-1}p^r
n\sum^\infty_{\ell=[n^t]+1}
\left(
\frac{\ell+r-1}{r-1}
\right)^{r-1}(1-p)^\ell(1+o(1))\notag\\
&\leq&
\frac{1}{\sqrt{2\pi}}
e^{r-1}p^r
\sum^\infty_{\ell=[n^t]+1}
(1-p)^{\ell/2}(1+o(1))\notag\\
&=&\frac{1}{\sqrt{2\pi}}
e^{r-1}p^{r-1}(1-p)^{([n^t]+1)/2}(1+o(1)).
\label{nb_tail}
\end{eqnarray}
\end{footnotesize}

\noindent The third line of \eqref{nb_tail} can be derived from Stirling's formula. For large enough $n$, we have 
$n(\ell+r-1/(r-1))^{r-1}\leq (1-p)^{\ell/2}$ for $\ell > n^t$. Thus, 
the fifth line of \eqref{nb_tail} holds. By \eqref{nb_tail}, 
$\sum^\infty_{n=1} P(\max_{1\leq i\leq n}Y_i > n^t)<\infty$, and thus the proof is completed by applying the Borel-Cantelli Lemma.
\end{proof}}

To derive the asymptotic behavior for $p(\mathbf{B}_n \mid \mathbf{Z}_n, \mathbf{W}_n)$, we need to bound the eigenvalues of $\bm{\Omega}_i = \textrm{diag}(\omega_{i1}, \ldots, \omega_{iq})$ for each $i = 1, \ldots, n$. Without loss of generality, we prove Proposition \ref{pgW} for all discrete responses. Namely, we show that the eigenvalues of $\sum^n_{i=1}(\bm{\Sigma}+\bm{\Omega}_i^{-1})^{-1}/n$ are bounded away from zero and infinity when all of the outcomes in $\mathbf{y}_i$ are discrete. 
However, if there are continuous (Gaussian) responses in $\mathbf{y}_i$, then the assertion that the eigenvalues of $\sum^n_{i=1}(\bm{\Sigma}+\bm{\Omega}_i^{-1})^{-1}/n$ are bounded away from zero and infinity still holds. This is because $\omega_{ik} = 1$ for all $k \in \{1, \ldots, q \}$ where $y_{ik}$ is Gaussian. {\pico Let $S \subset \mathcal{I} = \{ 1, \ldots, p_n \}$ be a set of row indices for the regression coefficients matrix.}
\begin{proposition}\label{pgW}
Let $\omega_{ik},~i=1,\dots,n$ are i.i.d. ${\cal PG}(b_{1k},0)$, $b_{1k}>0$, $k\in\{1,\dots,q\}$. Then as $n\rightarrow\infty$, for any $v>0$, 
 $\min_{1\leq i\leq n}\omega_{ki}> n^{-v}$ ~a.s. 
\end{proposition}

\begin{proof}[Proof of proposition \ref{pgW}] 
Based on the definition of a
	P\'{o}lya-gamma random variable
	\citep{polson2013bayesian},  
	if $\omega\sim {\cal PG}(b,c)$, then 
	\begin{eqnarray*}
		\omega =
		\frac{1}{2\pi^2}
		\sum^\infty_{\ell=1}
		\frac{g_\ell}{\left(
			\ell-\frac12\right)^2+\left(\frac{c}{2\pi}\right)^2},~~ g_\ell \overset{ind.}{\sim} {\cal G}(b,1).
	\end{eqnarray*}
For any given $k\in\{1,\dots,q\}$, $\omega_{k1}\sim {\cal PG}(b_{1k},0)$, and for $b_{1k}>0$, there exists large $N\in \mathbb{N}$ so that $\mbox{$Nb_{1k}>1$}$ and 
$\mbox{$Nb_{1k}v>2$}$.  {\picotwo Then,
letting $g^*=\sum_{\ell=1}^N g_\ell,~~ g_\ell \overset{ind.}{\sim} {\cal G}(b_{1k},1)$, we have 
$g^*\sim {\cal G}(Nb_{1k},1)$ and 
\[
\frac{1}{2\pi^2}
		\sum^\infty_{\ell=1}
		\frac{g_\ell}{\left(
			\ell-1/2\right)^2}
   \geq 
   \frac{1}{2\pi^2}
		\sum^N_{\ell=1}
		\frac{g_\ell}{\left(
			 N-1/2\right)^2}
    =
       \frac{1}{2\pi^2\left(
			 N-1/2\right)^2}g^*.
\]
Thus, 
 \begin{eqnarray}
P\left(\min_{1\leq i\leq n}\omega_{ki}\leq n^{-v}\right)
&=& 1- P\left(\min_{1\leq i\leq n}\omega_{ki} > n^{-v}\right)\notag\\
&=&1- P\left(\omega_{k1} > n^{-v}\right)^n\notag\\
&=& 1- P\left(
\frac{1}{2\pi^2}
		\sum^\infty_{\ell=1}
		\frac{g_\ell}{\left(
			\ell-1/2\right)^2}
> n^{-v}\right)^n\notag\\
&\leq& 1- P\left(
\frac{1}{2\pi^2(N-1/2)^2}g^*
> n^{-v}\right)^n\notag\\
&=& 1- \left\{1-P\left(
g^*
\leq \frac{2\pi^2(N-1/2)^2}{n^{v}}\right)\right\}^n. \label{ww}
 \end{eqnarray}
Note that $g^*\sim {\cal G}(Nb_{1k},1)$ and 
$u^{Nb_{1k}-1}e^{-u}$ is an increasing function of $u$ if $\mbox{$Nb_{1k}>1$}$, $u\in(0,1)$. We have that for sufficiently large $n$, 
 \begin{align*}
& \left\{1-P\left(g^*\leq \frac{2\pi^2(N-1/2)^2}{n^{v}}\right)\right\}^n \\
& \qquad = \left\{1-\int^{2\pi^2(N-1/2)^2/n^{v}}_0
\frac{1}{\Gamma(Nb_{1k})}u^{Nb_{1k}-1}e^{-u}
du\right\}^n\notag\\
& \qquad > 
\left\{1-
\frac{1}{\Gamma(Nb_{1k})}
\left(\frac{2\pi^2(N-1/2)^2}{n^{v}}\right)^{Nb_{1k}}e^{-2\pi^2(N-1/2)^2/n^{v}}
\right\}^n\notag\\
& \qquad = 
1+O(n^{1-Nb_{1k}v}).
  \end{align*}
  }
  Using equation~\eqref{ww}, we have
  \[
P\left(\min_{1\leq i\leq n}\omega_{ki}\leq n^{-v}\right)=O(n^{1-Nb_{1k}v}).
\] 
Since $1-Nb_{1k}v<-1$, we have  $P\left(\min_{1\leq i\leq n}\omega_{ki}\leq n^{-v}\right)=o(n^{-1})$. Therefore, by the Borel-Cantelli Lemma, the proof is established.
\end{proof}

\begin{lemma}\label{lemmaC1}
{\pico
 Suppose that Assumptions (A3)-(A4) hold, so that the diagonal entries of $\bm{\Omega}_i$ follow ${\cal PG}(b_{ik},0)$, $b_{ik}>0$, $k=1,\dots,q$, and
	$\bm{\Sigma} \mid \mathbf{U}_n\sim {\cal IW}(n+d_2,\mathbf{U}_n^\top \mathbf{U}_n+d_1 \mathbf{I}_q)$. Then for a positive number ${k_2}\in [0,1)$ in Assumption (B1) {\picotwo and all $i \in \{1, \ldots, n\}$, we have that} as $n\rightarrow\infty$,
	\begin{align*}
	n^{-{k_2}}
\leq \inf_{|S|=o(n/\log p_n)}  \lambda_{\min}\left(
n^{-1}\sum^n_{i=1}\left\{(\bm{\Sigma}+\bm{\Omega}^{-1}_i)^{-1} \otimes \mathbf{x}^S_i (\mathbf{x}_i^S)^\top\right\}\right)  \mbox{~~a.s.}
	\end{align*} 
 and
 \begin{align*}
\sup_{|S|=o(n/\log p_n)}  \lambda_{\max}\left(
n^{-1}\sum^n_{i=1}\left\{(\bm{\Sigma}+\bm{\Omega}^{-1}_i)^{-1} \otimes \mathbf{x}^S_i (\mathbf{x}_i^S)^\top\right\}\right)\leq n^{{k_2}}  \mbox{~~a.s.}
	\end{align*} 
    	where {\picotwo $\mathbf{x}_i^S$ denotes the $p^S\times 1$ sub-vector of the $i$th row $\mathbf{x}_i$ of the design matrix $\mathbf{X}_n$ containing the entries with indices in $S$, while} $\lambda_{\min}(\mathbf{A})$ and $\lambda_{\max}(\mathbf{A})$ denote the smallest eigenvalue and the largest eigenvalue
     for a matrix $\mathbf{A}$ respectively.}
\end{lemma}

\begin{proof}[Proof of Lemma \ref{lemmaC1}]

 {\pico Recall that $\bm{\Omega}_i=\textrm{diag}(\omega_{1i},\dots,\omega_{qi}),~i=1,\dots,n$. By 
 Proposition \ref{pgW}, 
 we obtain that 
 for any $v>0$,
 \begin{eqnarray}
 \min_{1\leq i\leq n}\omega_{ki}> n^{-v} ~a.s. \mbox{~~~~for each~$k \in \{ 1,\dots,q \}$},
 \label{bbd}  
 \end{eqnarray}
  as $n\rightarrow\infty$. Thus, for each $k\in\{1,\dots,q\}$, we take 
$v_k=d$ so that $0\leq d+k_0<1$ and 
$\min_{1\leq i\leq n}w_{ki} > n^{-d}~a.s.$,  where $k_0$ and $d$ are defined  in Assumption (A3).  Then we have 
\begin{eqnarray*}
(\bm{\Sigma}+n^{d}\mathbf{I}_q)^{-1}\otimes\left(
n^{-1}\sum_{i=1} \mathbf{x}^S_i (\mathbf{x}_i^S)^\top
\right)
\leq 
n^{-1}\sum^n_{i=1}\left\{(\bm{\Sigma}+\bm{\Omega}^{-1}_i)^{-1} \otimes \mathbf{x}^S_i (\mathbf{x}_i^S)^\top\right\}~a.s.
\end{eqnarray*}
By Assumption (A3),  for each $S$ satisfying $|S|=o(n/\log p_n)$, 
\begin{eqnarray}
\varepsilon_n(\bm{\Sigma}+n^{d}\mathbf{I}_q)^{-1}\otimes \mathbf{I}_{|S|}
\leq 
n^{-1}\sum^n_{i=1}\left\{(\bm{\Sigma}+\bm{\Omega}^{-1}_i)^{-1} \otimes \mathbf{x}^S_i (\mathbf{x}_i^S)^\top\right\}~a.s.\notag\\
\label{ccc}
\end{eqnarray}
Now, it is enough to show that $\bm{\Sigma}$ is bounded. Using the fact that
	$\bm{\Sigma} \mid \mathbf{U}_n \sim {\cal IW}(n+d_2, \mathbf{U}_n^\top \mathbf{U}_n+d_1 \mathbf{I}_q)$, we have that as $n\rightarrow\infty$,  
	\begin{eqnarray}
		\frac{1}{n+d_2} (\mathbf{U}_n^\top \mathbf{U}_n +d_1\mathbf{I}_q) \longrightarrow \bm{\Sigma}_0,~a.s.\label{s2}
	\end{eqnarray}
On the other hand, it is implied by Theorem 1 in \citep{zhu2012}
that 
	\begin{eqnarray}
  \frac{
\lambda_\ell(\bm{\Sigma})}{
 \lambda_{\ell}[E(\bm{\Sigma} \mid \mathbf{U}_n)]}=
 \frac{
\lambda_\ell(\bm{\Sigma})}
{\lambda_\ell[(\mathbf{U}_n^\top\mathbf{U}_n+d_1\mathbf{I}_q)/(n+d_2)]}\rightarrow 1
,~~\ell=1,\dots,q,~~a.s.\label{s3}
	\end{eqnarray}
so that  $\bm{\Sigma}\rightarrow \bm{\Sigma}_0~a.s.$  as $n\rightarrow\infty$. By Assumption (A4), the eigenvalues of $\bm{\Sigma}$ are bounded away from zero almost surely. Thus, 
\begin{eqnarray*}
n^{-d-k_0}=\varepsilon_n n^{-d}
\leq \lambda_{\min}\left(
n^{-1}\sum^n_{i=1}\left\{(\bm{\Sigma}+\bm{\Omega}^{-1}_i)^{-1} \otimes \mathbf{x}^S_i (\mathbf{x}_i^S)^\top\right\}\right)~a.s.
\end{eqnarray*}
Take ${ k_2}=d+k_0$, where ${ k_2}\in [0,1)$.  
By Assumption (A3), we have that as $n\rightarrow \infty$,
\begin{eqnarray*}
\lambda_{\max}[(n^{-1}\sum^n_{i=1}
(\bm{\Sigma}+\bm{\Omega}^{-1}_i)^{-1} \otimes \mathbf{x}^S_i (\mathbf{x}_i^S)^\top]
&\leq& \lambda_{\max}[n^{-1}\sum^n_{i=1}
\bm{\Sigma}^{-1} \otimes \mathbf{x}^S_i (\mathbf{x}_i^S)^\top]\\
&\leq& k_1\varepsilon^{-1}_n
\leq n^{{k_2}}.
\end{eqnarray*} 
The proof is done.
 }
\end{proof}

In order to prove the main results in the paper, we will first characterize the asymptotic behavior of the Mt-MBSP model when $\log p_n = O(n^{\alpha}), \alpha \in (0, \infty)$. We need the following assumptions, which are looser than assumptions (A1)-(A2).

\begin{enumerate}[label=(S\arabic*)]
	\item $n\ll p_n$ and $\log p_n=O(n^\alpha)$, $\alpha \in(0,\infty)$. 
	\item Let $S_0 \subset \{1, \dots, p_n \}$ denote the set of indices of the rows in $\mathbf{B}_0$ with at least one nonzero entry. Then $|S_0|=s_0$ satisfies $1\leq s_0$ and $ s_0=o(n)$.
\end{enumerate}

We need to consider these more general assumptions (S1)-(S2) in order to prove Theorems \ref{thm_inconsistent} and \ref{thm_C}. To derive the main results, we also need to construct a suitable test function. For the test function given in Lemma \ref{pg_prop} below, we need to first introduce the following notation. For each $S$ satisfying that $|S|=o(n/\log p_n)$, define the set $\mathcal{M}$ as 
\begin{eqnarray}
	\mathcal{M}=\left\{ S:  S \supset S_0,~ S \neq S_0,~ |S|\leq m_n\right\}, \label{usingM}
\end{eqnarray}
where $S_0$ is the set of indices of the true model (i.e. the indices of the true nonzero rows in $\mathbf{B}_0$), and {$m_n$ is a positive number satisfying $s_0\leq m_n$ and $m_n \log p_n=o(n^{(1- k_2)}\delta^2_n)$}, where {\picotwo $k_2$ is defined as in Assumption (B1)
} and {$\delta_n^2 \in (s_0  n^{-1} \log p_n,~s_0 \log p_n)$.} Define the set $\mathcal{T}$ as 
\begin{eqnarray}
	\mathcal{T}=\left\{ S:  S \subset (\mathcal{I} \setminus \mathcal{M}),\,|S|\leq n\right\}.\label{usingT}
\end{eqnarray}
Let $p^S=|S|$ and $s_0=|S_0|$, and let 
$\mathbf{B}^S$ and $\mathbf{B}_0^{S}$ denote respectively the $p^S\times q$ submatrices of $\mathbf{B}_n$ and $\mathbf{B}_0$ containing the rows with indices in $S$. {\picotwo Let the maximum likelihood estimator (MLE) for $\mathbf{B}^S_0$ be denoted by $\widehat{\mathbf{B}}^S$.} For $\mathbf{B}^S \in {\cal M}$,  define the set,
\begin{eqnarray}
	{\cal C}_n = \bigcup_{S\in {\cal M}}
	\left\{
	\| \widehat{\mathbf{B}}^S 
	- \mathbf{B}^S_0
	\|_F >\varepsilon\delta_n/2
	\right\}.
	\label{a13}
\end{eqnarray}

\begin{lemma}\label{pg_prop}
	Assume that we have a mixed-type response model \eqref{gaussianresponse}-\eqref{polson_model} where the parameters $\bm{\theta}_{i0}$'s satisfy \eqref{blatent} with true parameters $(\mathbf{B}_0, \bm{\Sigma}_0)$. Suppose that Assumptions (S1)-(S2) and Assumptions (A3)-(A4) in Section \ref{onestepcontraction} hold. Suppose that we endow the covariance matrix $\bm{\Sigma}$ with the inverse-Wishart prior \eqref{IWpriorSigma}. Let { $\delta_n^2 \in (s_0  n^{-1} \log p_n,~s_0  \log p_n)$}. Define $\Phi_n=\mathrm{1}(\mathbf{Z}_n\in {\cal C}_n)$ and 
	${\cal B}_\varepsilon=\{\mathbf{B}_n~:~\| \mathbf{B}_n-\mathbf{B}_0\|>\varepsilon\delta_n\}$, where ${\cal C}_n$ is defined as in \eqref{a13}. Then, for any arbitrary $\varepsilon>0$,  there exists a finite number $C>0$ such that as $n\rightarrow \infty$, \\[.1in]
	{\rm (i)$~$}~${\rm E}_{\mathbf{B}_0}(\Phi_n)
	\leq\exp(-C\,n^{\pico (1-k_2)}\delta_n^2/2)$; \\[.1in]
	{\rm (ii)}~$\sup_{\mathbf{B}_n\in{\cal B}_\varepsilon}{\rm E}_{\mathbf{B}_n}(1-\Phi_n)
	\leq\exp(-C\,n^{\pico (1-k_2)}\delta_n^2)$,
 where $k_2$ is defined as in Assumption (B1).
\end{lemma}
\begin{proof}[Proof of Lemma \ref{pg_prop}]
  {\picotwo Before proving (i), we first establish a tail bound \eqref{eta_sqr} that is needed for our proof. In what follows, $\mathbf{x}_i^S$ denotes the $p^S\times 1$ sub-vector of the $i$th row $\mathbf{x}_i$ of the design matrix $\mathbf{X}_n$ containing the entries with indices in $S$.

Without loss of generality, assume the first $\widetilde{q}$ response variables $y_1, \ldots, y_{\widetilde{q}}$ are discrete, where $0 < \widetilde{q} < q$, and the remaining response variables $y_{\widetilde{q}+1}, …, y_q$ are continuous. If all $q$ responses are discrete or if all $q$ responses are continuous, then our results can still be shown to hold by suitably modifying \eqref{baselinefn} and noting that the upper bound in \eqref{dvt} continues to hold.
It is implied by  \eqref{wyz} and \eqref{blatent} that the joint density function 
$\mbox{$(\mathbf{z}_i,~\bm{\omega}_i,~\mathbf{u}_i)\mid \mathbf{x}_i$}$ is
\begin{align*}  \label{jointLemmaC1}
&p( \mathbf{z}_i,~\bm{\omega}_i,~\mathbf{u}_i \mid \mathbf{x}_i)\notag\\
&=\{(2\pi)^q {\rm det}(\bm \Sigma)\}^{-1/2}
    \exp\left\{-\frac12 \mathbf u_i^\top \mathbf \Sigma^{-1} \mathbf u_i\right\}
\notag\\
&~~~~\times
p_0(\mathbf z_i, \bm \omega_i) 
\exp\left\{-\frac12
  (\mathbf z_i-(\mathbf B^S_0)^\top \mathbf x_i-\mathbf u_i)^\top\mathbf \Omega_i
  (\mathbf z_i-(\mathbf B^S_0)^\top \mathbf x_i-\mathbf u_i)
  \right\}, \numbereqn
\end{align*}
where the baseline function  
$p_0(\mathbf z, \bm \omega)$ 
in \eqref{jointLemmaC1} is
\begin{footnotesize}
\begin{eqnarray} \label{baselinefn}
p_0(\mathbf z, \bm \omega)
&=&\prod_{\ell=1}^q \left\{1\left(
z_{\ell}={\kappa_{\ell}}/{\omega_{\ell}},~\ell=1,\dots,\tilde q
\right)2^{-f_{\ell}^{(2)}}
\exp\left(\kappa^2_{\ell}/\omega_{\ell} 
\right){\cal PG}
(\omega_\ell \mid f^{(2)}_\ell,~0)\right.\notag\\
&&\left.+
1\left(
\omega_{\ell}=1,~\ell=\tilde q+1,\dots, q
\right)(
2\pi \omega_{\ell}
)^{-1/2}\right\}.
\end{eqnarray}
\end{footnotesize}
Then 
  \begin{eqnarray}
    p(\mathbf z_i, \bm \omega_i\mid \mathbf x_i)
    &=&
     p_0(\mathbf z_i, \bm \omega_i)
\{{\rm det}(\bm\Omega_i+\bm \Sigma)\}^{-1/2}
\{{\rm det}(\bm\Omega_i)\}^{1/2}
\notag\\
&&~~\times 
  \exp\left\{-\frac12
  (\mathbf z_i-(\mathbf B^S_0)^\top \mathbf x_i^S)^\top(\mathbf \Omega_i+
  \mathbf \Sigma)^{-1}
  (\mathbf z_i-(\mathbf B^S_0)^\top \mathbf x_i^S)
  \right\}.\notag\\
  \label{nongaussian}
  \end{eqnarray}
From Proposition \ref{pgW}, we have 
$\max_{1\leq i\leq n}{\omega_{ki}}^{-1} < n^{v}$ ~a.s. for any $v>0$.  
Moreover, by Assumption (B4), $\max_{i=1,\dots,n}\kappa^2_{ki}<n^{t}$ for any $t>0$ as $n$ is large. 
Take $\tilde d=v+t$ so that 
\begin{eqnarray}
\max_{1\leq i\leq n} \max_{1\leq \ell\leq q}z^2_{\ell i}= 
\max_{1\leq i\leq n} \max_{1\leq \ell\leq q}\frac{\kappa^2_{\ell i}}{\omega^2_{\ell i}} < n^{\tilde d}
\mbox{~~and~~} 
m_nn^{2\tilde d}=o(n^{1-k_2}\delta_n^2)
\label{dvt}
\end{eqnarray}
Based on \eqref{nongaussian} and the fact that $\textrm{vec}((\mathbf{x}_i^S)^\top \mathbf{B}^S_0)=(
  \mathbf{I}_q \otimes (\mathbf{x}_i^S)^\top
  )
  \textrm{vec}(\mathbf{B}_0^S)
  $, 
  the score function for our model is 
  \begin{eqnarray}
  s(\mathbf{B}^S)= 
  \sum^n_{i=1}
  ((\mathbf \Omega_i+
  \mathbf \Sigma)^{-1}\otimes \mathbf{x}_i^S)(\mathbf{z}_i - 
  (\mathbf{I}_q\otimes (\mathbf{x}_i^S)^\top)\textrm{vec}(
  \mathbf B^{S})). \label{score}
  \end{eqnarray}
  Recall that $\widehat{\mathbf{B}}^S$ is the MLE for for $\mathbf{B}^S_0$.  
  Since $s(\widehat{\mathbf{B}}^S)=\bm 0$
  and $E_{\mathbf{B}_0^S}[s({\mathbf{B}}^S_0)]=\bm 0$, we have 
\begin{eqnarray*}
  \frac{\sqrt{n}}{n}
 s({\mathbf B}^S_0)&=&
  \frac{\sqrt{n}}{n}
 s({\mathbf B}^S_0)-
 E_{\mathbf B_0^S}[s({\mathbf B}^S_0)]\notag\\
 &=&
   \frac{\sqrt{n}}{n}\{
  s({\mathbf B}^S_0)
  - s(\widehat{\mathbf B}^S_0)
   \} \notag\\
   &=&
     \left(\sum^n_{i=1}
  (\mathbf \Omega_i+
  \mathbf \Sigma)^{-1}\otimes \mathbf x_i^S(\mathbf x_i^S)^\top\right) \left(\textrm{vec}(
  \widehat{\mathbf B}^{S})
  -\textrm{vec}(
  \mathbf B^{S}_0)
  \right)\notag\\
  &=& \sqrt{n}\widetilde{\bm{\Sigma}}^{-1}_n \left(\textrm{vec}(
  \widehat{\mathbf B}^{S})
  -\textrm{vec}(
  \mathbf B^{S}_0)
  \right), \label{sbs}
\end{eqnarray*}
where 
	\begin{eqnarray*}
	\widetilde{\bm{\Sigma}}_n= 
	\left( \frac1n \sum^n_{i=1} \left\{(\bm{\Sigma}+\bm{\Omega}^{-1}_i)^{-1}\otimes 
	\left(\mathbf{x}^S_i (\mathbf{x}_i^S)^\top \right)\right\}\right)^{-1}.
	\end{eqnarray*}
	Let $\tilde \lambda_\ell,$ $\ell=1,\dots,p^Sq$ be the $\ell$th eigenvalues of $\widetilde{\bm{\Sigma}}_n$.  Based on Lemma \ref{lemmaC1}, we have
 	\begin{eqnarray}
 n^{-{ k_2}}<\min_{\ell} \lambda_\ell(\widetilde{\bm{\Sigma}}_n^{-1}),~{ k_2}\in[0,1).\label{kc} 
 	\end{eqnarray}
Let $\bm \eta =
\widetilde{\bm{\Sigma}}^{1/2}_ns({\mathbf B}^S_0)/\sqrt{n}
$. Then $
E(\bm \eta)=\bm 0
\mbox{~~and~~}
{\rm var}(\bm \eta)=\mathbf{I}_{p^Sq}
$ and 
\begin{eqnarray}
 \left\|\textrm{vec}(
  \widehat{\mathbf B}^{S})
  -\textrm{vec}(
  \mathbf B^{S}_0)
  \right\|^2
  =\frac{1}{n}
  {\rm tr} \left(\widetilde{\bm{\Sigma}}_n {\bm \eta}
{\bm \eta}^\top \right).
  \label{sbs2}
 \end{eqnarray}
 Form \eqref{dvt}  and \eqref{score}, we have $\max_{1\leq j\leq m_n}|\eta_{j}|^2<n^{\tilde d}~a.s.$ by noting that $\bm \eta$ can be 
 regarded as a weighted linear combination of 
 $\{\mathbf{z}_i - 
  (\mathbf{I}_q\otimes (\mathbf{x}_i^S)^\top)\textrm{vec}(
  \mathbf B^{S})\}_{i=1}^n$. 
It follows from Hoeffding's inequality that 
 for any $t>0$,
\begin{eqnarray}
P(\|\bm \eta\|^2>t)
\leq \exp\left(
-C\frac{(t-E\|\bm\eta\|^2)^2}{m_n n^{2\tilde d}}
\right)
\leq \exp\left(
-C\frac{(t-m_n)^2}{m_n n^{2\tilde d}}
\right),\notag\\
\label{eta_sqr}
\end{eqnarray}
where $C>0$ is a constant.}

To prove (i), {\picotwo note that} 
	\begin{align*}
		& E_{\mathbf{B}_0}(\Phi_n)=
		P(\mathbf{Z}_n\in {\cal C}_n)
		\leq \sum_{S\in{\cal M}}
		P \left(
		\| \widehat{\mathbf{B}}^S 
		-\mathbf{B}^S_0
		\|_F >\varepsilon\delta_n/2
		\right)\notag\\
		& \qquad  \leq \sum^{m_n}_{|S|=s_0+1}
		\binom{p_n}{|S|}
		P \left(
		\| 
		\textrm{vec} (\widehat{\mathbf{B}}^S)- \textrm{vec} (\mathbf{B}^S_0)
		\|^2_F  > \varepsilon^2\delta_n^2/4
		\right)\notag\\
		&\qquad \leq \sum^{m_n}_{|S|=s_0+1}
		\binom{p_n}{|S|}
		P \left( 
		{\rm tr}\left( \widetilde{\bm{\Sigma}}_n \bm{\eta} \bm{\eta}^\top \right) > \varepsilon^2n\delta^2_n/4\right)\notag\\
  & \qquad \leq \sum^{m_n}_{|S|=s_0+1}
		\binom{p_n}{|S|}
		P \left( 
		n^{ k_2}{\rm tr}\left(\bm{\eta} \bm{\eta}^\top \right) > \varepsilon^2n\delta^2_n/4\right)\notag\\
		& \qquad \leq \sum^{m_n}_{p^S=s_0+1}
		\binom{p_n}{p^S}
		P \left( 
  {\picotwo \|\bm \eta\|^2 >}\varepsilon^2n^{ (1-{ k_2} )}\delta^2_n/4\right)\label{tail_chisqr}, \numbereqn
	\end{align*}
 {\picotwo
where the third inequality of the display holds by using \eqref{sbs2} and
 the fourth inequality holds by using \eqref{kc}. Using}
 the fact that $\sum^m_{i=k}\binom{n}{i}\leq (m-k+1)\binom{n}{m}$,  the right-hand-side of \eqref{tail_chisqr} can be bounded as follows. From \eqref{eta_sqr}, we take 
 $t=\varepsilon^2n^{ (1-{ k_2} )}\delta^2_n/4$ and 
 obtain that
  for sufficiently large $n$,
 \begin{align*} 	\label{eq24}
		& \sum^{m_n}_{p^S=s_0+1}
		\binom{p_n}{p^S}
\exp\left(
	-{\picotwo C}\frac{n^{(1- k_2)}\varepsilon^2\delta_n^2}{4}\right) \\
		& \qquad \leq 
		(m_n-s_0)\left(\frac{ep_n}{m_n}\right)^{m_n} 
		\exp\left(
		-{\picotwo C}\frac{n^{ (1- k_2)}\varepsilon^2\delta_n^2}{4}
		\right). \numbereqn
	\end{align*}
 Using the facts that {\pico $m_n\ll n^{\picotwo (1-k_2)}\delta^2_n/\log\,p_n$} {\picotwo and $ m_n n^{2\tilde d}=o(n^{1-k_2}\delta_n^2)$ in \eqref{dvt},} it is implied by
	\eqref{eq24} that
	\begin{align*} \label{LemmaA1finalinequality}
		& \log\left[
\sum^{m_n}_{p^S=s_0+1}
		\binom{p_n}{p^S}
p^Sq\,\exp\left(
	-{\picotwo C}\frac{n^{(1-k_2 )}\varepsilon^2\delta_n^2}{4}\right)
		\right] \\
		&\qquad \leq \log(m_n)
		+m_n(1+\log p_n)-{\picotwo C}
		\frac{1}{4}
		n^{ (1- k_2)}\varepsilon^2\delta^2_n. \numbereqn
	\end{align*}
	Altogether, combining \eqref{tail_chisqr}-\eqref{LemmaA1finalinequality},  we have
	\[
	E_{\mathbf{B}_0}(\Phi_n) \leq
	\exp\left(
	-Cn^{\pico (1-k_2)}\delta_n^2/2
	\right). 
	\]
	The proof of (i) is done. 
	The proof of (ii) follows almost identical arguments as the proof of Proposition 2 of \cite{BaiGhosh2018}, and thus, we omitted it here.
\end{proof}

  {\picotwo Next, we have Proposition
\ref{pg_thm}
} which provides an upper bound on the radius of the ball in which the posterior {\picotwo (conditionally on $\mathbf{Z}_n$ and $\mathbf{W}_n$)} asymptotically puts all of its mass when $\log p_n = O(n^{\alpha}), \alpha \in (0, \infty)$. 

\begin{proposition} \label{pg_thm}
	Assume the same setup as Lemma \ref{pg_prop} where Assumptions (S1)-(S2) and (A3)-(A4) hold. Suppose that we endow the regression coefficients matrix $\mathbf{B}_n$ with a prior $p(\mathbf{B}_n)$ satisfying 
 \begin{align}
 		P \left( \| \mathbf{B}_n-\mathbf{B}_0\|_F <  \widetilde{C} n^{-\rho/2}\delta_n
		\right)> \exp(-Dn^{\pico (1-k_2)}\delta_n^2)  \label{ssv}
	\end{align}
{\pico for any $(\tilde C, D, \rho)$ with $\tilde C>0$, $D>0$, and $\rho>k_2$. Moreover, $\bm{\Sigma}$ follows an inverse-Wishart prior \eqref{IWpriorSigma}.} Then for any arbitrary $\varepsilon>0$, 
	\begin{eqnarray*}
		\sup_{\mathbf{B}_0}  E_{\mathbf{B}_0} P \left(  
		\|\mathbf{B}_n-\mathbf{B}_0\|_F > \varepsilon\delta_n~\bigg|~\mathbf{Z}_n, \mathbf{W}_n
		\right)\longrightarrow  0\mbox{~as~} n\rightarrow \infty.
	\end{eqnarray*}  
\end{proposition}
\begin{proof}[Proof of Proposition \ref{pg_thm}]
	
	The proof is similar to the proof of \cite{BaiGhosh2018}. Based on Lemma \ref{pg_prop}, both the type I and type II errors 
	can control an exponential decay rate. Let $\Phi_n=1(\mathbf{Z}_n\in {\cal C}_n)$, where 
	${\cal C}_n$ is defined as in \eqref{a13}. Let 
	\[
	J_{ {\cal B}_\varepsilon}=\int_{\mathbf{B}_n\in{\cal B}_\varepsilon}
	\frac{p(\mathbf{Z}_n, {\picotwo \mathbf{W}_n} \mid \mathbf{B}_n)}{p(\mathbf{Z}_n, {\picotwo \mathbf{W}_n} \mid \mathbf{B}_0)}
	dP(\mathbf{B}_n)
	\mbox{~~and~~}
	J_n = \int
	\frac{p(\mathbf{Z}_n, {\picotwo \mathbf{W}_n} \mid \mathbf{B}_n)}{p(\mathbf{Z}_n, {\picotwo \mathbf{W}_n} \mid \mathbf{B}_0)}
	dP(\mathbf{B}_n).
	\]
	The posterior probability of ${\cal B}_\varepsilon=\{\| \mathbf{B}_n -\mathbf{B}_0 \|_F>\varepsilon\delta_n\}$, where $\varepsilon>0$, is given by 
	\[
	P({\cal B}_\varepsilon\mid \mathbf{Z}_n, {\picotwo \mathbf{W}_n})=
	\frac{J_{ {\cal B}_\varepsilon}
	}{J_n}\leq \Phi_n+(1-\Phi_n)
	\frac{J_{ {\cal B}_\varepsilon}
	}{J_n}\overset{def}{=}
	A_1+\frac{A_2}{J_n}.
	\]
	By using Markov's inequality and part (i) of Lemma \ref{pg_prop}, we have 
	\begin{align*}
		P(A_1\geq 
		\exp\{
		-C\,n^{\pico (1-k_2)}\delta^2_n/4
		\}) & \leq  \exp(C\,n^{\pico (1-k_2)}\delta^2_n/4) E_{\mathbf{B}_0}(\Phi_n) \\
		& \leq \exp(-C\,n^{\pico (1-k_2)}\delta^2_n/4).
	\end{align*}
	Thus, the Borel-Cantelli lemma yields
	\[
	P(A_1\geq 
	\exp\{-C\,n^{\pico (1-k_2)}\delta^2_n/4\}\mbox{~~infinity often}
	)=0,
	\]
	since $\sum^\infty_{n=1}
	P(A_1\geq 
	\exp\{-C\,n^{\pico (1-k_2)}\delta^2_n/4\})
	\leq \sum^\infty_{n=1} \exp(-C\,n^{\pico (1-k_2)}\delta^2_n/4)<\infty. 
	$ We have $A_1\rightarrow 0$ a.s. in $P_{\mathbf{B}_0}$-probability as $n\rightarrow \infty$.
	
	Next we will look at the term $A_2$. By Fubini's theorem, 
	\begin{align*}
		 E_{\mathbf{B}_0}(A_2) &=  E_{\mathbf{B}_0}[(1-\Phi_n)J_{ {\cal B}_\varepsilon}] \\
		&= E_{\mathbf{B}_0} \left(
		(1-\Phi_n)\int_{\mathbf{B}_n\in {\cal B}_\varepsilon}
		\frac{p(\mathbf{Z}_n, {\picotwo \mathbf{W}_n} \mid \mathbf{B}_n)}{p(\mathbf{Z}_n, {\picotwo \mathbf{W}_n} \mid \mathbf{B}_0) }dP(\mathbf{B}_n)
		\right)\notag\\
		&=\int_{\mathbf{B}_n\in {\cal B}_\varepsilon}
		E_{\mathbf{B}_n}(1-\Phi_n)dP(\mathbf{B}_n)\notag\\
		&\leq P(\mathbf{B}_n\in {\cal B}_\varepsilon)
		\sup_{\mathbf{B}_n \in {\cal B}_\varepsilon} E_{\mathbf{B}_0}(1-\Phi_n)\notag\\
		&\leq \sup_{\mathbf{B}_n \in {\cal B}_\varepsilon} E_{\mathbf{B}_0}(1-\Phi_n)
		\leq \exp( -C\,n^{\pico 1-k_2}\delta^2_n).
	\end{align*}
	The last inequality follows from part (ii) of Lemma \ref{pg_prop}. Again, by the similar argument for $A_1$, we apply the Borel-Cantelli lemma to obtain that 
	\[
	P(A_2\geq 
	\exp\{-C\,n^{\pico (1-k_2)}\delta^2_n/2\}\mbox{~~infinity often}
	)=0,
	\]
	since $\sum^\infty_{n=1}
	P(A_2\geq 
	\exp\{-C\,n^{\pico (1-k_2)}\delta^2_n/2\})
	\leq \sum^\infty_{n=1} \exp(-C\,n^{\pico (1-k_2)}\delta^2_n/2)<\infty.$ We have $A_1\rightarrow 0$ a.s. in ${P}_{\mathbf{B}_0}$-probability as $n\rightarrow \infty$. To finish the proof, we will show that 
	\begin{eqnarray*}
		\exp(C\,n^{\pico (1-k_2)}\delta^2_n/2)J_n\rightarrow\infty~\mbox{a.s. in}~ P_{\mathbf{B}_0}\text{-probability as } n\rightarrow\infty.
	\end{eqnarray*}
	Define $${\cal D}_n\overset{def}{=}\left\{ \mathbf{B}_n~:~
	\frac{1}{n^{\pico (1-k_2)}\delta^2_n}\log\frac{p(\mathbf{Z}_n, {\picotwo \mathbf{W}_n} \mid \mathbf{B}_n)}{p(\mathbf{Z}_n, {\picotwo \mathbf{W}_n} \mid \mathbf{B}_0)}<v\right\}\mbox{~~for~~}0<v<C/2.$$
	We have
 \begin{footnotesize}
	\begin{align*} \label{Dn}
		& \exp(C\,n^{\pico (1-k_c)}\delta^2_n/2)J_n \\
        & \hspace{.2cm} = \exp(C\,n^{\pico (1-k_2)}\delta^2_n/2)
		\int \exp\left\{
		-n^{\pico (1-k_2)}\delta^2_n \frac{1}{n^{\pico (1-k_2)}\delta^2_n}\log\frac{p(\mathbf{Z}_n, {\picotwo \mathbf{W}_n} \mid \mathbf{B}_n)}{p(\mathbf{Z}_n, {\picotwo \mathbf{W}_n} \mid \mathbf{B}_0)}\right\}dP(\mathbf{B}_n)\notag\\
		& \hspace{.2cm} \geq \exp\left\{
		(C/2-v)\,n^{\pico (1-k_2)}\delta^2_n
		\right\} P({\cal D}_n\mid {\picotwo \mathbf{Z}_n, \mathbf{W}_n} ). \numbereqn
	\end{align*}
 \end{footnotesize}
	{\picotwo Integrating out the $\mathbf{u}_i$'s in \eqref{jointLemmaC1}, we obtain}
	\begin{align*} \label{a12}
		& \log\frac{p(\mathbf{Z}_n, {\picotwo \mathbf{W}_n} \mid \mathbf{B}_n)}{p(\mathbf{Z}_n, {\picotwo \mathbf{W}_n} \mid \mathbf{B}_0)} = -\frac{1}{2}\sum^n_{i=1}\left\{
		(\mathbf{z}_i-\mathbf{B}^\top_n \mathbf{x}_i)^\top(\bm{\Omega}^{-1}_i+\bm{\Sigma})^{-1}(\mathbf{z}_i-\mathbf{B}^\top_n \mathbf{x}_i)\right.\notag \\
		& \qquad \qquad \qquad \qquad \qquad \hspace{.3cm}  -\left. 
		(\mathbf{z}_i-\mathbf{B}^\top_0 \mathbf{x}_i)^\top(\bm{\Omega}^{-1}_i+\bm{\Sigma})^{-1}(\mathbf{z}_i-\mathbf{B}_0^\top \mathbf{x}_i)
		\right\}\notag\\
		& \qquad = 
		-\frac12\sum^n_{i=1}{\rm tr}\left\{
		(\bm{\Omega}^{-1}_i+\bm{\Sigma})^{-1}(\mathbf{z}_i-\mathbf{B}^\top_n \mathbf{x}_i)(\mathbf{z}_i-\mathbf{B}^\top_n \mathbf{x}_i)^\top\right\}\notag\\
		& \qquad \qquad +\frac12\sum^n_{i=1}
		{\rm tr}\left\{
		(\bm{\Omega}^{-1}_i+\bm{\Sigma})^{-1}(\mathbf{z}_i-\mathbf{B}_0^\top \mathbf{x}_i)(\mathbf{z}_i-\mathbf{B}^\top_0 \mathbf{x}_i)^\top
		\right\}\notag\\
		& \qquad \leq {\pico -\frac12\sum^n_{i=1}
		{\rm tr}
		\left\{
		(\bm{\Omega}^{-1}_i+\bm{\Sigma})^{-1}
		(\mathbf{B}_n^\top \mathbf{x}_i-\mathbf{B}_0^\top \mathbf{x}_i)(\mathbf{B}_n^\top \mathbf{x}_i-\mathbf{B}^\top_0 \mathbf{x}_i)^\top\right\}}\notag\\
		& \qquad \qquad +\sum^n_{i=1}
		\left[
		{\rm tr}\left\{
		(\bm{\Omega}^{-1}_i+{\bm\Sigma})^{-1}
		(\mathbf{B}_n^\top \mathbf{x}_i-\mathbf{B}_0^\top \mathbf{x}_i)(\mathbf{B}_n^\top \mathbf{x}_i-\mathbf{B}^\top_0 \mathbf{x}_i)^\top\right\}\right. \notag\\
		& \qquad \qquad \qquad \hspace{.2cm} \times\left.
		{\rm tr}\left\{
		(\bm{\Omega}^{-1}_i+\bm{\Sigma})^{-1}(\mathbf{z}_i-\mathbf{B}_0^\top \mathbf{x}_i)(\mathbf{z}_i-\mathbf{B}^\top_0 \mathbf{x}_i)^\top
		\right\}\right]^{1/2}. \numbereqn
	\end{align*}
 {\pico Let $\kappa_n=n^{(1-k_2+\rho)/2}$, $\rho>k_2$, where $k_2\in [0,1)$ is defined in Proposition \ref{pgW}.} {\pico
 We have 
 {\footnotesize
\begin{align*}    	\label{eq28}
 & P({\cal D}_n   \mid \mathbf{Z}_n, {\picotwo \mathbf{W}_n}) = P\left( \mathbf{B}_n~:~
	\frac{1}{n^{\pico (1-k_2)}\delta^2_n}\log\frac{p(\mathbf{Z}_n, {\picotwo \mathbf{W}_n} \mid \mathbf{B}_n)}{p(\mathbf{Z}_n, {\picotwo \mathbf{W}_n} \mid \mathbf{B}_0)}<v\right)\notag\\
 &\qquad \geq P \left(
		\frac{1}{n^{\pico (1-k_2)}\delta^2_n}
		\left\{ 
		{\rm tr}\sum^n_{i=1}
		(\bm{\Omega}^{-1}_i+\bm{\Sigma})^{-1}
		(\mathbf{B}_n^\top \mathbf{x}_i-\mathbf{B}_0^\top \mathbf{x}_i)(\mathbf{B}_n^\top \mathbf{x}_i-\mathbf{B}^\top_0 \mathbf{x}_i)^\top\right\}^{1/2}\right.\notag\\
		& \qquad \qquad \times 
		\left.
		\left\{ 
		{\rm tr}\sum^n_{i=1}
(\bm{\Omega}^{-1}_i+\bm{\Sigma})^{-1}
		(\mathbf{z}_i-\mathbf{B}_0^\top \mathbf{x}_i)(\mathbf{z}_i-\mathbf{B}^\top_0 \mathbf{x}_i)^\top\right\}^{1/2}< v
		\right)  \notag\\
 &\qquad \geq P \left(
		\frac{1}{n^{\pico (1-k_2)}\delta^2_n}
		\left\{ 
		{\rm tr}\sum^n_{i=1}
		(\bm{\Omega}^{-1}_i+\bm{\Sigma})^{-1}
		(\mathbf{B}_n^\top \mathbf{x}_i-\mathbf{B}_0^\top \mathbf{x}_i)(\mathbf{B}_n^\top \mathbf{x}_i-\mathbf{B}^\top_0 \mathbf{x}_i)^\top\right\}^{1/2}< \frac{v}{\kappa_n},\right.\notag\\
		& \qquad \qquad \qquad
		\left. 
		\left\{ 
		{\rm tr}\sum^n_{i=1}
		(\bm{\Omega}^{-1}_i+\bm{\Sigma})^{-1}
		(\mathbf{z}_i-\mathbf{B}_0^\top \mathbf{x}_i)(\mathbf{z}_i-\mathbf{B}^\top_0 \mathbf{x}_i)^\top\right\}^{1/2}< \kappa_n
		\right)\notag\\
		& \qquad = P\left( \frac{1}{n^{\pico (1-k_2)}\delta^2_n}
		\left\{
		{\rm tr}\sum^n_{i=1}
		(\bm{\Omega}^{-1}_i+\bm{\Sigma})^{-1}
		(\mathbf{B}_n^\top \mathbf{x}_i-\mathbf{B}_0^\top \mathbf{x}_i)(\mathbf{B}_n^\top \mathbf{x}_i-\mathbf{B}^\top_0 \mathbf{x}_i)^\top\right\}^{1/2}< \frac{v}{\kappa_n}\right), \numbereqn
	\end{align*} }
	by noting that 
 {\footnotesize 
	\[
	P\left(
	\left\{ 
	{\rm tr}\sum^n_{i=1}
	(\bm{\Omega}^{-1}_i+\bm{\Sigma})^{-1}
	(\mathbf{z}_i-\mathbf{B}_0^\top \mathbf{x}_i)(\mathbf{z}_i-\mathbf{B}^\top_0 \mathbf{x}_i)^\top\right\}^{1/2}> \kappa_n 
	\mbox{~~infinity often}\right)=0. 
	\]}
	Thus, from \eqref{eq28}, Assumption (A3), and  
	Proposition \ref{pgW}, we have that 
 {\footnotesize
	\begin{align*}
		& P({\cal D}_n\mid \mathbf{Z}_n, {\picotwo \mathbf{W}_n}) \\
		& \qquad \geq P \left( 
		\frac{1}{n^{\pico (1-k_2)}\delta^2_n}\left\{ 
		{\rm tr}\sum^n_{i=1}
		(\bm{\Omega}^{-1}_i+\bm{\Sigma})^{-1}
		(\mathbf{B}_n^\top \mathbf{x}_i-\mathbf{B}_0^\top \mathbf{x}_i)(\mathbf{B}_n^\top \mathbf{x}_i-\mathbf{B}^\top_0 \mathbf{x}_i)^\top\right\}^{1/2}< \frac{v}{\kappa_n}\right)\notag\\
		& \qquad \geq 
		P \left( 
	\left\{ 
		{\rm tr}\sum^n_{i=1}
		(\mathbf{B}_n-\mathbf{B}_0)
		(\bm{\Omega}^{-1}_i+\bm{\Sigma})^{-1}
		(\mathbf{B}_n^\top-\mathbf{B}_0^\top)\mathbf{x}_i\mathbf{x}_i^\top\right\}^{1/2}< \frac{n^{\pico (1-k_2)}\delta_nv}{\kappa_n}\right)\notag\\
		& \qquad \geq P \left( 
		\|\mathbf{B}_n-\mathbf{B}_0\|_F< \frac{v\delta_n}{n^{\rho/2} }\right).
	\end{align*}}}
 Note that the second inequality holds since $\delta_n\geq 1$. For sufficiently large $n$ and any $D > 0$, we have 
	\[
	P\left(
	\|\mathbf{B}_n-\mathbf{B}_0\|_F< \frac{2v\delta_n}{3n^{\rho/2}}\right)
	>\exp(-D n^{\pico (1-k_2)}\delta_n^2),
	\]
	we let $D \in (0,C/2-v)$ so that  
	$\exp(C\,n^{\pico (1-k_2)}\delta^2_n/2)J_n\rightarrow \infty$  a.s. in $P_{\mathbf{B}_0}$-probability as $n\rightarrow\infty$. This establishes the result.
\end{proof}

{\picotwo Recall that $\mathbf{Z}_n=(\mathbf{z}_1,\dots,\mathbf{z}_n)^\top$ and $\mathbf{W}_n=(\bm{\omega}_1,\dots,\bm{\omega}_n)^\top$ are $n\times q$ matrices with rows $\mathbf{z}_i$ and $\bm{\omega}_i$, $i = 1, \ldots, n$, where $z_{ik}$ and $\omega_{ik}, k = 1, \ldots, q$, are defined as in \eqref{p4}. Meanwhile, $\mathbf{X}_n = (\mathbf{x}_1, \ldots, \mathbf{x}_n)^\top$ is the $n \times p$ design matrix, and $\mathbf{Y}_n = (\mathbf{y}_1, \ldots, \mathbf{y}_n)^\top$ is the $n \times q$ matrix of mixed-type responses. 

When there are discrete responses in $\mathbf{Y}_n$, the matrix $\mathbf{W}_n$ is an unobserved and random matrix. In this case, it is more desirable to analyze the posterior conditioned on the actually observed data $\mathbf{X}_n$ and $\mathbf{Y}_n$ rather than on $\mathbf{Z}_n$ and $\mathbf{W}_n$. The next proposition allows us to relate events defined on the posterior $p(\mathbf{B}_n \mid \mathbf{Z}_n, \mathbf{W}_n)$ with the same events on the posterior $p(\mathbf{B}_n \mid \mathbf{X}_n, \mathbf{Y}_n)$.

\begin{proposition} \label{Prop:C.6}
 Assume the same setup as Proposition \ref{pg_thm}. Let $\mathcal{E}_n$ be an event depending on $\mathbf{B}_n$ and $\mathbf{B}_0$. 
 \begin{enumerate}[label=(\roman*)]
     \item If $\sup_{\mathbf{B}_0} E_{\mathbf{B}_0} P( \mathcal{E}_n \mid \mathbf{Z}_n, \mathbf{W}_n ) \rightarrow 0$ as $n \rightarrow \infty$, then we also have
     \begin{equation*}
     \sup_{\mathbf{B}_0} E_{\mathbf{B}_0} P( \mathcal{E}_n \mid \mathbf{X}_n, \mathbf{Y}_n ) \rightarrow 0~ \text{ as } ~n \rightarrow \infty.
     \end{equation*}
     \item If $\inf_{\mathbf{B}_0} E_{\mathbf{B}_0} P( \mathcal{E}_n \mid \mathbf{Z}_n, \mathbf{W}_n ) > \delta$ as $n \rightarrow \infty$, where $\delta \in (0, 1)$, then
     \begin{equation*}
      \inf_{\mathbf{B}_0} E_{\mathbf{B}_0} P( \mathcal{E}_n \mid \mathbf{X}_n, \mathbf{Y}_n ) > \delta ~ \text{ as } ~n \rightarrow \infty.
     \end{equation*}
 \end{enumerate}
\end{proposition}

\begin{proof}
We first prove (i). We have
\begin{align*}
	&\hspace{-0in}\sup_{\mathbf{B}_0}  E_{\mathbf{B}_0} P \left( 
		\mathcal{E}_n ~\bigg|~\mathbf{X}_n, \mathbf{Y}_n
		\right) \\
  &\hspace{-0in}= \sup_{\mathbf{B}_0}  E_{\mathbf{B}_0} 
  \left\{\int
  P \left( 
		\mathcal{E}_n ~\bigg|~\mathbf{W}_n,  \mathbf{X}_n, \mathbf{Y}_n
		\right) 
  p(\mathbf{W_n}\mid \mathbf{X}_n, \mathbf{Y}_n)
  d\mathbf{W}_n\right\}\\
  &\hspace{-0in}= \sup_{\mathbf{B}_0}  E_{\mathbf{B}_0} 
  \left\{\int
  P \left( 
		\mathcal{E}_n ~\bigg|~\mathbf{Z}_n,  \mathbf{W}_n
		\right) 
  p(\mathbf{W_n}\mid \mathbf{X}_n, \mathbf{Y}_n)~
  d\mathbf{W}_n\right\}\\
    &\hspace{-0in}\leq \sup_{\mathbf{B}_0}  E_{\mathbf{B}_0} 
  \left\{\int
  \sup_{\mathbf{B}_0}  E_{\mathbf{B}_0}  P \left( 
		\mathcal{E}_n ~\bigg|~\mathbf{Z}_n, \mathbf{W}_n
		\right) 
  p(\mathbf{W_n}\mid \mathbf{X}_n, \mathbf{Y}_n)
  ~d\mathbf{W}_n\right\}\\
  &\hspace{-0in}\longrightarrow 
  0\mbox{~as~} n\rightarrow \infty.
	\end{align*}
  To prove (ii), note that 
\begin{align*}
	&\hspace{-0in}\inf_{\mathbf{B}_0}  E_{\mathbf{B}_0} P \left( 
		\mathcal{E}_n ~\bigg|~\mathbf{X}_n, \mathbf{Y}_n
		\right) \\
  &\hspace{-0in}= \inf_{\mathbf{B}_0}  E_{\mathbf{B}_0} 
  \left\{\int
  P \left( 
		\mathcal{E}_n ~\bigg|~\mathbf{W}_n,  \mathbf{X}_n, \mathbf{Y}_n
		\right) 
  p(\mathbf{W_n}\mid \mathbf{X}_n, \mathbf{Y}_n)
  d\mathbf{W}_n\right\}\\
  &\hspace{-0in}= \inf_{\mathbf{B}_0}  E_{\mathbf{B}_0} 
  \left\{\int
  P \left( 
		\mathcal{E}_n ~\bigg|~\mathbf{Z}_n,  \mathbf{W}_n
		\right) 
  p(\mathbf{W_n}\mid \mathbf{X}_n, \mathbf{Y}_n)~
  d\mathbf{W}_n\right\}\\
    &\hspace{-0in}\geq \inf_{\mathbf{B}_0}  E_{\mathbf{B}_0} 
  \left\{\int
  \inf_{\mathbf{B}_0}  E_{\mathbf{B}_0}  P \left( 
		\mathcal{E}_n ~\bigg|~\mathbf{Z}_n, \mathbf{W}_n
		\right) 
  p(\mathbf{W_n}\mid \mathbf{X}_n, \mathbf{Y}_n)
  ~d\mathbf{W}_n\right\}\\
  &\hspace{-0in} > \delta \mbox{~as~} n\rightarrow \infty.
	\end{align*}
\end{proof}
}

\subsection{Proof of Theorem \ref{one-step_contraction_rate}}

We first prove a more general result about the asymptotic behavior of the posterior for $\mathbf{B}_n$ under the Mt-MBSP model when Assumptions (S1)-(S2) hold. The reason for first obtaining Theorem \ref{thm_D} below is because we need Theorem \ref{thm_D} in order to prove Theorem \ref{thm_C}. The proof of Theorem \ref{one-step_contraction_rate} is {\picotwo also} a straightforward modification of Theorem \ref{thm_D}. 


\begin{theorem} \label{thm_D}
	Assume the same setup as Proposition \ref{pg_thm}. Suppose that conditions (S1)-(S2), (A3)-(A4) and (B1)-(B3) hold, and we endow $\mathbf{B}_n$ with the Mt-MBSP prior \eqref{Mt-MBSP}-\eqref{global_local_prior} and $\bm{\Sigma}$ with the inverse-Wishart prior \eqref{IWpriorSigma}. Then, for any arbitrary $\varepsilon>0$, 
	\begin{equation*}
		\sup_{\mathbf{B}_0}  E_{\mathbf{B}_0} P \left( 
		\|\mathbf{B}_n-\mathbf{B}_0\|_F> \varepsilon \delta_n~\bigg|~ \mathbf{Z}_n, \mathbf{W}_n
		\right)\longrightarrow  0\mbox{~as~} n\rightarrow \infty,
	\end{equation*}
	where $\delta^2_n \in (s_0\log p_n/n,~
	s_0\log p_n)$.
\end{theorem}

\begin{proof}
	In light of Proposition \ref{pg_thm}, it is enough to prove that the Mt-MBSP prior \eqref{Mt-MBSP}-\eqref{global_local_prior} for $\mathbf{B}_n$ satisfies condition \eqref{ssv}. That is, 
	for any constants $C>0$ and $D >0$, 
	\begin{eqnarray}\label{C19}
		P \left( \|\mathbf{B}_n-\mathbf{B}_0\|_F <  C \frac{\delta_n}{n^{\rho/2}} \right)
		>\exp(-D n^{\pico (1-k_2)}\delta_n^2). 
	\end{eqnarray}
	under the Mt-MBSP prior on \eqref{Mt-MBSP}-\eqref{global_local_prior} on $\mathbf{B}_n$. Let $\mathbf{B}_n=( \bm{\beta}_1,\dots, \bm{\beta}_q)=(\mathbf{b}^n_1,\dots,\mathbf{b}^n_{p_n})^\top$. 
	Since the rows of $\mathbf{B}_n$ are independent, the marginal 
	prior of the Mt-MBSP model can also be written as 
	\begin{eqnarray}
		p(\mathbf{B}_n) =\prod^{p_n}_{j=1}p(\mathbf{b}_j^n),
		\label{PP}
	\end{eqnarray}
	where $\mathbf{b}^n_{j}\mid \xi_j,\bm{\Sigma} \overset{ind}{\sim}{\cal N}_q( \bm{0}_q,~\tau_n\xi_j\bm{\Sigma})$ for $j = 1,\dots,p_n$. Therefore, the marginal prior for $p(\mathbf{b}^n_{j})$ is 
	\[
	p(\mathbf{b}^n_{j}) = \int^\infty_{0}
	(2\pi\tau_n\xi_j)^{-q/2}|\bm{\Sigma}|^{-1/2}\exp\left(
	-\frac{1}{2\xi_j\tau_n}\|\bm{\Sigma}^{-1/2} \mathbf{b}^n_j\|^2
	\right)p(\xi_j)d\xi_j,
	\]
	where $p(\xi_j)$ is the hyperprior density \eqref{global_local_prior}. We have 
	\begin{align*} \label{BB0}
		& P\left( \|\mathbf{B}_n-\mathbf{B}_0\|^2_F < C^2\frac{\delta_n^2}{n^\rho} \right) =
		P\left( \sum_{j\in S_0}\|
		\mathbf{b}^n_j-\mathbf{b}^0_j
		\|^2 + 
		\sum_{j\notin S_0}\|\mathbf{b}^n_j\|^2
		< C^2\frac{\delta_n^2}{n^\rho} 
		\right) \\
		& \qquad \qquad \geq P\left( \sum_{j\in S_0}\|
		\mathbf{b}^n_j-\mathbf{b}^0_j
		\|^2 < C^2\frac{\delta_n^2}{2n^\rho},~~ 
		\sum_{j\notin S_0}\|\mathbf{b}^n_j\|^2
		<C^2 \frac{\delta_n^2}{2n^\rho} 
		\right) \\
		& \qquad \qquad \geq \prod_{j\in S_0} P\left( \|
		\mathbf{b}^n_j-\mathbf{b}^0_j
		\|^2 < C^2\frac{\delta_n^2}{2s_0n^\rho}\right) P\left( \sum_{j\notin S_0}
		\|\mathbf{b}^n_j\|^2
		< C^2\frac{\delta_n^2}{2n^\rho} 
		\right) \\
		& \qquad \qquad \overset{def}{=} {\cal I}_1\times {\cal I}_2.  \numbereqn
	\end{align*}
	We first bound ${\cal I}_1$ from below. In the following, 
	we let $C_1$ and $C_2$ be appropriate constants that do not depend on $n$, while $C_3$ is the constant from  Proposition \ref{PropositionA2}. Let $\Delta^2_n=C^2\delta_n^2/(2s_0n^\rho)$. Then $\Delta^2_n= O(\log p_n)$. From \eqref{BB0}, we have 
 {\footnotesize
	\begin{flalign}  
		\mathcal{I}_1 & = \prod_{j \in S_0} P \left( \lVert \mathbf{b}_j^n - \mathbf{b}_j^0 \rVert^2 < {\Delta}^2_n \right) \notag\\
		& = \prod_{j \in S_0} \left[ \int_{0}^{\infty} \int_{\lVert \mathbf{b}_j^n - \mathbf{b}_j^0 \rVert^2 < {\Delta}^2_n
		} \frac{1}{(2 \pi \tau_n \xi_j )^{q/2} \lvert \bm{\Sigma} \rvert^{1/2}} \exp \left( - \frac{\lVert \bm{\Sigma}^{-1/2} \mathbf{b}_j^n \rVert^2}{2 \xi_j \tau_n}  \right) p(\xi_j) d \mathbf{b}_j^n d \xi_j \right] \notag\\
		& \geq \prod_{j \in S_0} \left[ \int_{0}^{\infty}  \int_{\lVert \mathbf{b}_j^n - \mathbf{b}_j^0 \rVert^2 <{\Delta}^2_n} (2 \pi \tau_n \xi_j k_1 )^{-q/2} \exp \left( - \frac{\lVert \mathbf{b}_j^n \rVert^2}{2 \xi_j \tau_n k_1} \right) p(\xi_j) d \mathbf{b}_j^n d \xi_j \right] \notag\\
		& \geq \prod_{j \in S_0} \prod^q_{k=1}\left[ \int_{0}^{\infty}  \int_{ \lvert b_{jk}^n - b_{jk}^0 \rvert^2 < \frac{\Delta^2_n}{q}} \frac{1}{ (2 \pi \tau_n \xi_j k_1)^{1/2}} \exp \left( - \frac{(b_{jk}^n)^2}{2 \xi_j \tau_n k_1} \right) p(\xi_j) d b_{jk}^n d \xi_j \right]\notag\\
		&\overset{\Delta}{=}A,\label{ff}
	\end{flalign} }
	where the third inequality in \eqref{ff} comes from the boundedness of $\bm{\Sigma}$ in \eqref{s2} and Assumption (A4) that $0 < k_1^{-1} < \lambda_{\min} (\bm{\Sigma}) \leq \lambda_{\max}(\bm{\Sigma}) < k_1 < \infty$. 
	
	We continue to lower-bound the term $A$ in \eqref{ff}. {\pico Let $p'_n=\exp\{n^{(1-k_2)}\}$ be a sequence where $k_2\in [0,1)$ is defined in Proposition \ref{pgW}.} In the fifth line of the following display \eqref{I1lowerboundpt1}, we use Assumption (B3) that $\max_{j,k} \lvert (\mathbf{B}_{0})_{ij}\rvert^2=O(\log p_n)$ so that the term $\{|b^0_{jk}|\vee (\Delta_n/q^{1/2} )\}^2\leq M \log p_n$ for some finite number $M>0$. We also apply Proposition \ref{PropositionA2} to the integral term by the fact that $p^{\star} (\xi_j) = C_2 \xi_j^{-1/2} \pi (\xi_j) = C_2 \xi_j^{-(a+1/2)-1} L(\xi)$, where $C_2$ is the normalizing constant. 
 
 We have
\begin{footnotesize}
	\begin{align*}  
		 A & = \prod_{j \in S_0} \prod_{k=1}^{q} \left[ \int_{0}^{\infty}  \int_{ b_{jk}^{0} - \Delta_n/q^{1/2} }^{b_{jk}^{0} + \Delta_n/q^{1/2} } \frac{1}{(2 \pi \tau_n \xi_j k_1 )^{1/2}} \exp \left( - \frac{(b_{jk}^n)^2}{2 \xi_j \tau_n k_1} \right) p(\xi_j) d b_{jk}^n d \xi_j \right] \notag\\
		& \geq \prod_{j \in S_0} \prod_{k=1}^{q} \left[ \int_{0}^{\infty} \frac{C_1\,\Delta_n}{(\tau_n \xi_j)^{1/2}} \exp \left( - \frac{\{|b^0_{jk}|\vee (\Delta_n/q^{1/2} )\}^2}{2 \xi_j \tau_n k_1} \right) p(\xi_j) d \xi_j \right] (1+o(1)) \notag\\
		&> \prod_{j \in S_0} \prod_{k=1}^{q} \left[ \int_{s_0 {\pico p'_n} n^{\rho}\log p_n/\delta^2_n}^{\infty} \frac{C_1\Delta_n}{(\tau_n \xi_j)^{1/2}} \exp \left( - \frac{(\{|b^0_{jk}|\vee (\Delta_n/q^{1/2} )\}^2}{2 \xi_j \tau_n k_1} \right) p(\xi_j) d \xi_j \right] (1+o(1)) \notag\\
		& \geq \prod_{j \in S_0} \prod_{k=1}^{q} \left[ \frac{C_1\Delta_n}{ C_2 ( \tau_n)^{1/2}} \exp \left( - \frac{\{|b^0_{jk}|\vee (\Delta_n/q^{1/2} )\}^2}{2 k_1 \tau_n s_0 {\pico p'_n} n^{\rho}\log p_n/\delta_n^2} \right) \int_{s_0 {\pico p'_n} n^{\rho}\log p_n/\delta^2_n}^{\infty} p^{\star}(\xi_j) d \xi_j \right] \\
        &\qquad \qquad \qquad \times (1+o(1)), \notag\\
		& \qquad \qquad \qquad  \textrm{ where } p^{\star}(\xi_j) = C_2 \xi_j^{-1/2} p(\xi_j) \textrm{ and } \int_{0}^{\infty} p^{\star} (\xi_j) d \xi_j = 1 \notag  \\
		& > \prod_{j \in S_0} \prod_{k=1}^{q} \left[ \frac{C_1\Delta_n}{C_2 ( \tau_n)^{1/2}} \exp \left( - \frac{M\delta^2_n}{2 k_1 \tau_n s_0 {\pico p'_n} n^{\rho}} \right) \exp \left(-C_3\,\log {\pico p'_n} \right) \right] (1+o(1)) \notag\\
		& = \exp \left( q s_0 \left[ \log \left( \frac{C_1}{C_2} \right) - \frac{1}{2} \log\,\tau_n +  \log\,\Delta_n - \frac{M\delta^2_n}{2 k_1 \tau_n s_0 {\pico p'_n} n^{\rho}} - C_3\, \log {\pico p'_n} \right] \right)(1+o(1)) 
    \end{align*}
  \end{footnotesize}
  \begin{footnotesize}
  \begin{align*} \label{I1lowerboundpt1} 
		& \hspace{.2cm} > \exp \left( -q C_3 n^{(\pico 1-k_2)} - \frac{q Mn^{\pico (1-k_2)} \delta^2_n}{2 k_1 \tau_n {\pico p'_n} n^{\pico \rho+1-k_2}} + qs_0 \left[ \log \left( \frac{C_1}{C_2} \right) - \frac{1}{2} \log\,\tau_n + \log\,\Delta_n \right] \right)  \\
    &    \qquad \qquad \times (1+o(1)) \\
  	& \hspace{.2cm} > \exp \left( -q C_3 n^{(\pico 1-k_2)}\delta^2_n - \frac{q Mn^{\pico (1-k_2)} \delta^2_n}{2 k_1 \tau_n {\pico p'_n} n^{\pico \rho+1-k_2}} + qs_0 \left[ \log \left( \frac{C_1}{C_2} \right) - \frac{1}{2} \log\,\tau_n + \log\,\Delta_n \right]\right) \\
   &\qquad \qquad  \times (1+o(1)). \\
  \numbereqn
 	\end{align*}
  \end{footnotesize}
	\noindent	By Assumption (B1), {\pico $\tau_n n^{ \rho+1-k_2}p'_n=c$}, and further, by Assumptions (S1)-(S2), $s_0 \log p_n /n \leq \delta_n^2$. Meanwhile, all the other terms in the exponent of \eqref{I1lowerboundpt1} are of order ${\pico o(n^{(1-k_2)}\delta^2_n)}$ except ${\pico-q C_3 n^{(1-k_2)}\delta^2_n}$. Therefore, \eqref{I1lowerboundpt1} implies
	{\pico \begin{eqnarray}
		\mathcal{I}_1 > \exp\left(- D n^{(1-k_2)}\delta_n\right), \label{lowerboundpt2}
	\end{eqnarray} }
	for sufficiently large $n$. Further, to derive the lower bound of $\mathcal{I}_2$ in \eqref{BB0}, we apply Markov's inequality and the fact that $\mathbf{b}^n_j\mid \xi_j \sim {\mathcal N}_q(0_q, \tau_n\xi_j \bm{\Sigma})$ for $j = 1, \ldots, p_n$. We have 
	\begin{eqnarray}
		{\cal I}_2 &=&
		P\left(
		\sum_{j \notin S_0}\|\mathbf{b}^n_{j}\|^2
		< s_0 \Delta_n^2 \right)\notag\\
		&=&
		1-P\left( \sum_{j\notin S_0}\|\mathbf{b}^n_{j}\|^2
		>s_0 \Delta_n^2\right)\notag\\ 
		&=&1- E \left[P\left(
		\sum_{j\notin S_0}\|\mathbf{b}^n_{j}\|^2
		>s_0 \Delta_n^2 \hspace{.1cm} \bigg| \hspace{.1cm} \xi_1,\dots,\xi_{p_n} \right)\right]\notag\\
		&\geq&1- E \left[
		\frac{ E \left( \sum_{j\notin S_0}\|\mathbf{b}^n_{j}\|^2\mid \xi_1,\dots,\xi_{p_n}  \right) }{s_0 \Delta^2_n}\right]\notag\\
		&=&1-
		\frac{\tau_n {\rm tr}(\Sigma) E \left( \sum_{j\notin S_0} \xi_j \right)}{s_0\Delta^2_n} \notag \\
		& \geq &
		1-
		\frac{\tau_n (p_n-s_0) E (\xi_1)}{ k_1 s_0\Delta^2_n}\rightarrow 1 \textrm{ as } n \rightarrow \infty,  \label{lowerboundI3}
	\end{eqnarray}
	where the last inequality of \eqref{lowerboundI3} follows from Assumption (A4), Assumption (B1), and the fact $E(\xi_1) < \infty$. Note that the argument of boundedness of ${\rm tr}(\bm{\Sigma})$ is similar to that of \eqref{s2}. {\pico In addition, it is implied  by the assumptions (A3) and (B1) that $p_n \leq \exp\{n^{1-k_2}\}=p'_n$
 since  $n^{k_2}=o(n/\log p_n)$. 
} Thus, 
	{\pico \[
	\frac{\tau_n(p_n-s_0) E(\xi_1)}{k_1s_0\Delta_n^2}=
	\frac{2\tau_n(p_n-s_0) E(\xi_1)}{k_1C^2n^{(1-k_2)}\delta_n^2}
	\rightarrow 0 \mbox{~~as $n \rightarrow \infty$ }.
	\] } 
 By \eqref{BB0}, \eqref{lowerboundpt2}, and \eqref{lowerboundI3}, it is clear that for sufficiently large $n$, \eqref{C19} holds.
\end{proof}

\begin{proof}[Proof of Theorem \ref{one-step_contraction_rate}]
	
	To prove Theorem \ref{one-step_contraction_rate}, we can modify the proofs of Lemma \ref{pg_prop}, Proposition \ref{pg_thm}, and Theorem \ref{thm_D} {\picotwo to show that
\begin{equation*}
		\sup_{\mathbf{B}_0}  E_{\mathbf{B}_0} P \left( 
		\|\mathbf{B}_n-\mathbf{B}_0\|_F> \varepsilon \left( \frac{s_0 \log p_n}{n} \right)^{1/2} ~\bigg|~{\picotwo \mathbf{Z}_n, \mathbf{W}_n}
		\right)\longrightarrow  0\mbox{~as~} n\rightarrow \infty.
	\end{equation*}
 }
 The proofs go through with $\delta_n$ replaced by $(s_0 \log p_n / n)^{1/2}$ as long as we assume that $s_0 = o(n / \log p_n)$, as in Assumption (A2). Since $s_0$ is allowed to diverge as $n \rightarrow \infty$, we must have $\log p_n = o(n)$, as in Assumption (A1). Thus, replacing conditions (S1)-(S2) in Theorem \ref{thm_D} with conditions (A1)-(A2) gives the desired posterior contraction rate when $\log p_n = o(n)$, {\picotwo conditional on $\mathbf{Z}_n$ and $\mathbf{W}_n$. 

To obtain the posterior contraction rate conditional on $\mathbf{X}_n$ and $\mathbf{Y}_n$, we define the event $\mathcal{E}_n = \{ \|\mathbf{B}_n-\mathbf{B}_0\|_F> \varepsilon [(s_0 \log p_n )/n]^{1/2} \}$ and use part (i) of Proposition \ref{Prop:C.6}.}
 
\end{proof}

\subsection{Proofs of Theorem \ref{thm_inconsistent} and Corollary \ref{corollary_inconsistent}} 

\begin{proof}[Proof of Theorem \ref{thm_inconsistent}]
{To give a counterexample, let $k_2=0$ for Proposition \ref{pgW}.}
	Let $\Phi_n=1(\mathbf{Z}_n\in {\cal C}_n)$, where 
	${\cal C}_n$ is defined as in \eqref{a13}. Let 
	\[
	J_{ {\cal B}_\varepsilon}=\int_{\mathbf{B}_n\in{\cal B}_\varepsilon}
	\frac{p(\mathbf{Z}_n, {\picotwo \mathbf{W}_n} \mid \mathbf{B}_n)}{p(\mathbf{Z}_n, {\picotwo \mathbf{W}_n} \mid \mathbf{B}_0)}
	dP(\mathbf{B}_n)
	\mbox{~~and~~}
	J_n = \int
	\frac{p(\mathbf{Z}_n, {\picotwo \mathbf{W}_n} \mid \mathbf{B}_n)}{p(\mathbf{Z}_n, {\picotwo \mathbf{W}_n} \mid \mathbf{B}_0)}
	dP(\mathbf{B}_n).
	\]
	The posterior probability of ${\cal B}_\varepsilon=\{\|\mathbf{B}_n -\mathbf{B}_0\|_F>\varepsilon/4\}$, where $\varepsilon>0$, can be rewritten as 
	\[
	P({\cal B}_\varepsilon\mid \mathbf{Z}_n, {\picotwo \mathbf{W}_n})=
	\frac{J_{ {\cal B}_\varepsilon}
	}{J_n}\geq 
	\frac{(1-\Phi_n)J_{ {\cal B}_\varepsilon}
	}{J_n}. 
	\]
	To prove the theorem, we need to show that under the assumption that $\log p_n=Cn$ where $C>0$,
	\begin{equation} \label{thm3pti}
		\inf_n{E}_{\mathbf{B}_0}\{(1-\Phi_n){\cal B}_\varepsilon\}>0,
	\end{equation}
	and 
	\begin{equation} \label{thm3ptii}
		J_n<\infty~a.s.~\textrm{ in } P_{\mathbf{B}_0}\textrm{-probability}. 
	\end{equation} 
	To prove \eqref{thm3pti}, we will use a similar argument as in the proof of part (i) of Lemma \ref{pg_prop}. On the set 
	${\cal B}^*_\varepsilon=\{\mathbf{B}_n~:~\|\mathbf{B}^S_n-\mathbf{B}^S_0\|_F > \varepsilon/4,~\forall S\in{\cal M}\}$, we have
	\begin{eqnarray*}
		&&{E}_{\mathbf{B}_n}(\Phi_n)\notag\\
		&\leq&
		\sum_{S\in {\cal M}}
		P \left(
		\|\widehat{\mathbf{B}}^S 
		-\mathbf{B}^S_0
		\|_F >\varepsilon/2,~~\|\mathbf{B}^S_n-\mathbf{B}^S_0\|_F > \varepsilon/4,~\forall S\in{\cal M}
		\mid~ \mathbf{B}_n\right)\notag\\
		&\leq&
		\sum_{S\in {\cal M}}
		P \left(
		\|\widehat{\mathbf{B}}^S 
		-\mathbf{B}^S_n
		\|_F >\varepsilon/4 \mid~ \mathbf{B}_n
		\right).
	\end{eqnarray*}
	Similar to  
	\eqref{tail_chisqr} and \eqref{eq24}, 
	\[
	\log\,
	{E}_{\mathbf{B}_n}(\Phi_n)
	\leq \log\,(m_n-s_0)+
	m_n(1+\log\,p_n)
	-\frac{n\varepsilon^2}{32k_2^2},
	\]
	where $s_0$, and $k_2$ are defined as in Lemma \ref{pg_prop}. Note that $\log\,p_n=Cn$ where $C>0$. Thus, we let $m_n$ be a positive number such that $s_0\leq m_n$ and $m_n$ satisfies that $
	\log\,(m_n-s_0)+
	m_n(1+\log p_n)-n\varepsilon^2/(32k_2^2)\in (-\infty,0)$. This implies that 
	${\rm E}_{\mathbf{B}_n}(1-\Phi_n)$ has a nonzero (constant) lower bound. Therefore, 
	\begin{eqnarray*}
		E_{\mathbf{B}_0}[(1-\Phi_n)J_{ {\cal B}_\varepsilon}]
		&\geq& E_{\mathbf{B}_0}[(1-\Phi_n)J_{ {\cal B}^*_\varepsilon}] \notag\\
		&=&
		E_{\mathbf{B}_0} \left(
		(1-\Phi_n)\int_{\mathbf{B}_n\in {\cal B}^*_\varepsilon}
		\frac{p(\mathbf{Z}_n, {\picotwo \mathbf{W}_n} \mid \mathbf{B}_n)}{p(\mathbf{Z}_n, {\picotwo \mathbf{W}_n} \mid \mathbf{B}_0) }dP(\mathbf{B}_n)
		\right)\notag\\
		&=&\int_{\mathbf{B}_n\in {\cal B}^*_\varepsilon}
		E_{\mathbf{B}_n}(1-\Phi_n)dP(\mathbf{B}_n)\notag\\
		&\geq& P(\mathbf{B}_n\in {\cal B}^*_\varepsilon)
		\inf_{\mathbf{B}_n \in {\cal B}_\varepsilon} E_{\mathbf{B}_0}(1-\Phi_n).
	\end{eqnarray*}
	Obviously, $\inf_n P(\mathbf{B}_n\in {\cal B}^*_\varepsilon)>0$. Thus, the proof of \eqref{thm3pti} is done. 
	
	In order to prove \eqref{thm3ptii}, we let $\mathbf{B}_n=( \bm{\beta}_1,\dots, \bm{\beta}_q)=(\mathbf{b}^n_1,\dots,\mathbf{b}^n_{p_n})^\top$. From \eqref{PP}, 
	the marginal prior for $p(\mathbf{b}^n_{j})$ is 
	\[
	p(\mathbf{b}^n_{j}) = \int^\infty_{0}
	(2\pi\tau_n\xi_j)^{-q/2}|\bm{\Sigma}|^{-1/2}\exp\left(
	-\frac{1}{2\xi_j\tau_n}\|\bm{\Sigma}^{-1/2}\mathbf{b}^n_j\|^2
	\right)p(\xi_j)d\xi_j,
	\]
	where $p(\xi_j)$ is the hyperprior density \eqref{global_local_prior}. We have 
	\begin{align*} \label{Jnequiv}
		& J_n = \int
		\frac{p(\mathbf{Z}_n, {\picotwo \mathbf{W}_n} \mid \mathbf{B}_n)}{p(\mathbf{Z}_n, {\picotwo \mathbf{W}_n} \mid \mathbf{B}_0)}
		dP(\mathbf{B}_n) \\
		& = \int \exp\left\{
		\log\,\frac{p(\mathbf{Z}_n, {\picotwo \mathbf{W}_n} \mid \mathbf{B}_n)}{p(\mathbf{Z}_n, {\picotwo \mathbf{W}_n} \mid \mathbf{B}_0)} \right. \\
        & \qquad \left. -\sum^{p_n}_{j=1}\left(\frac{\|\bm{\Sigma}^{-1/2} \mathbf{b}_j^n\|^2}{2\xi_j\tau_n}+
		\frac{q}{2}\log(2\pi\tau_n\xi_j)+\frac12\log\,|\bm{\Sigma}|-\log\,p(\xi_j)\right)
		\right\} d\mathbf{B}_n \\
		& \overset{\Delta}{=}\int A\,d \mathbf{B}_n. \numbereqn
	\end{align*}
	Similar to \eqref{a12}, 
	\begin{align*} \label{ratioupperbound}
		& \log\frac{p(\mathbf{Z}_n, {\picotwo \mathbf{W}_n} \mid \mathbf{B}_n)}{p(\mathbf{Z}_n, {\picotwo \mathbf{W}_n} \mid \mathbf{B}_0)} = -\frac{1}{2}\sum^n_{i=1}\left\{
		(\mathbf{z}_i-\mathbf{B}^\top_n \mathbf{x}_i)^\top(\bm{\Omega}^{-1}_i+\bm{\Sigma})^{-1}(\mathbf{z}_i-\mathbf{B}^\top_n \mathbf{x}_i)\right.\notag\\
		& \qquad \qquad \qquad \qquad \qquad \hspace{.3cm} -\left. 
		(\mathbf{z}_i-\mathbf{B}^\top_0 \mathbf{x}_i)^\top(\bm{\Omega}^{-1}_i+\bm{\Sigma})^{-1}(\mathbf{z}_i-\mathbf{B}_0^\top \mathbf{x}_i)
		\right\}\notag\\
		&\qquad \leq \frac12\sum^n_{i=1}
		{\rm tr}
		\left\{
		(\bm{\Omega}^{-1}_i+\bm{\Sigma})^{-1}
		(\mathbf{B}_n^\top \mathbf{x}_i-\mathbf{B}_0^\top \mathbf{x}_i)(\mathbf{B}_n^\top \mathbf{x}_i-\mathbf{B}^\top_0 \mathbf{x}_i)^\top\right\}\notag\\
		& \qquad \qquad +\sum^n_{i=1}
		\left[
		{\rm tr}\left\{
		(\bm{\Omega}^{-1}_i+\bm{\Sigma})^{-1}
		(\mathbf{B}_n^\top \mathbf{x}_i-\mathbf{B}_0^\top \mathbf{x}_i)(\mathbf{B}_n^\top \mathbf{x}_i-\mathbf{B}^\top_0 \mathbf{x}_i)^\top\right\}\right. \notag\\
		& \qquad \qquad \qquad \hspace{.2cm} \times\left.
		{\rm tr}\left\{
		(\bm{\Omega}^{-1}_i+\bm{\Sigma})^{-1}(\mathbf{z}_i-\mathbf{B}_0^\top \mathbf{x}_i)(\mathbf{z}_i-\mathbf{B}^\top_0 \mathbf{x}_i)^\top
		\right\}\right]^{1/2}. \numbereqn 
	\end{align*}
	Let $\kappa_n=n^{(1+\rho)/2}$ for $\rho>0$. Since  
	{\footnotesize \[
	P \left(
	\left\{ 
	{\rm tr}\sum^n_{i=1}
	(\bm{\Omega}^{-1}_i+\bm{\Sigma})^{-1}
	(\mathbf{z}_i-\mathbf{B}_0^\top \mathbf{x}_i)(\mathbf{z}_i-\mathbf{B}^\top_0 \mathbf{x}_i)^\top\right\}^{1/2}> \kappa_n 
	\mbox{~~~~infinity often}\right)=0,
	\]}
	and $\tau_n=O(p_n^{-1})$, 
	we have 
 {\footnotesize
	\begin{align*} \label{logAupperbound}
		& \log\,A  \leq \frac12\sum^n_{i=1}
		{\rm tr}
		\left\{
		(\bm{\Omega}^{-1}_i+\bm{\Sigma})^{-1}
		(\mathbf{B}_n^\top \mathbf{x}_i-\mathbf{B}_0^\top \mathbf{x}_i)(\mathbf{B}_n^\top \mathbf{x}_i-\mathbf{B}^\top_0 \mathbf{x}_i)^\top\right\}\notag\\
		& \qquad \qquad  +\left[\sum^n_{i=1}
		{\rm tr}\left\{
		(\bm{\Omega}^{-1}_i+\bm{\Sigma})^{-1}
		(\mathbf{B}_n^\top \mathbf{x}_i-\mathbf{B}_0^\top \mathbf{x}_i)(\mathbf{B}_n^\top \mathbf{x}_i-\mathbf{B}^\top_0 \mathbf{x}_i)^\top\right\}\right]^{1/2}\kappa_n \notag \\
		& \qquad \qquad  -\sum^{p_n}_{j=1}\frac{\|\bm{\Sigma}^{-1/2}b_j^n\|^2}{2\xi_j\tau_n}-
		\sum^{p_n}_{j=1}\frac{q}{2}\log(2\pi\tau_n\xi_j)-\sum^{p_n}_{j=1}\frac12\log\,|\bm{\Sigma}|+\sum^{p_n}_{j=1}\log\,p(\xi_j)\notag\\
		& \qquad = 
		O_p(np_n)+
		O_p(n^{1+\frac{\rho}{2}}p^{1/2}_n)
		-O_p(p_n^2)-O_p(p_n\log\,p_n)-O(p_n)+O(p_n)
		\notag\\
		& \qquad = -O_p(p_n^2)<0. \numbereqn
	\end{align*}}
	Note that the term $\sum^{p_n}_{j=1}\|\Sigma^{-1/2} \mathbf{b}_j^n\|^2/(2\xi_j\tau_n)$ is the dominating term of order $O_p(p_n^2)$ in \eqref{logAupperbound}. Combining \eqref{Jnequiv}-\eqref{logAupperbound}, we have that $J_n=\int A d \mathbf{B}_n$ is bounded, i.e. we have established \eqref{thm3ptii}. Combining \eqref{thm3pti}-\eqref{thm3ptii} gives that $E_{\mathbf{B}_0} P({\cal B}_\varepsilon \mid \mathbf{Z}_n, {\picotwo \mathbf{W}_n} )$ is bounded away from zero. {\picotwo To finish the proof and obtain the result conditional on $\mathbf{X}_n$ and $\mathbf{Y}_n$, we apply part (ii) of Proposition \ref{Prop:C.6} to the event ${\cal B}_{\varepsilon}$.}
\end{proof}

\begin{proof}[Proof of Corollary \ref{corollary_inconsistent}]
	By Theorem \ref{one-step_contraction_rate} and Theorem \ref{thm_inconsistent}, the one-step Mt-MBSP posterior for $\mathbf{B}_n$ is consistent when $\log p_n = o(n)$ but inconsistent if $\log p_n = C n, C > 0$. 
\end{proof}

\subsection{Proofs of Theorems \ref{thm_C} and \ref{thm_D2}}

\begin{proof}[Proof of Theorem \ref{thm_C}]
	Based on Theorem \ref{thm_D}
	and our proposed two-step procedure, we have that for $\varepsilon>0$, 
 {\footnotesize
	\begin{align*}
		\sup_{\mathbf{B}_0}E_{\mathbf{B}_0}
		P \left(\sum^{p_n}_{j=1}
		\sum^{q}_{k=1}\{(\mathbf{B}_n)_{jk}-(\mathbf{B}_0)_{jk}\}^2 > \varepsilon^2
		\frac{r_n\log p_n}{n} \mid \mathbf{Z}_n,~\mathbf{W}_n
		\right)\longrightarrow 0\mbox{~~as $n\rightarrow\infty$},
	\end{align*}}
	for $r_n\in (s_0,~n)$. 
	Thus, for each $\mathbf{B}_0$,
 {\footnotesize
	\begin{eqnarray*}
		\sup_{\mathbf{B}_0} E_{\mathbf{B}_0}
		P \left(\sum^{p_n}_{j=1}
		\sum^{q}_{k=1}\{(\mathbf{B}_n)_{jk}-(\mathbf{B}_0)_{jk}\}^2 \leq \varepsilon^2
		\frac{r_n\log p_n}{n} \mid \mathbf{Z}_n,~\mathbf{W}_n
		\right)\longrightarrow 1\mbox{~~as $n\rightarrow\infty$}.
	\end{eqnarray*}}
	Then it follows that
	\begin{align*} 	\label{s20}
		& \sup_{\mathbf{B}_0} E_{\mathbf{B}_0}
		P \left(~
		\textstyle\sum^{q}_{k=1}\{(\mathbf{B}_n)_{jk}-(\mathbf{B}_0)_{jk}\}^2 \leq \varepsilon^2
		(r_n\log p_n / n ) \right. \\
		& \left. \qquad \qquad \qquad \textrm{ for all } j\in \{1,\dots,p_n\}\mid \mathbf{Z}_n,~\mathbf{W}_n
		\right)\longrightarrow 1\mbox{~~as $n\rightarrow\infty$}.\numbereqn
	\end{align*}

{\color{black}
 \noindent We have 
    { 
	\begin{eqnarray*}
		&&\left\{
		\sum^{q}_{k=1}\left((\mathbf{B}_n)_{jk}-(\mathbf{B}_0)_{jk}\right)^2 \leq \varepsilon^2
		\frac{r_n\log p_n}{n} \textrm{ for all } j
		\right\} \notag\\
		&\bigsubset&
		\left\{
		\max_{k=1,
			\dots,q}((\mathbf{B}_n)_{jk}-(\mathbf{B}_0)_{jk}) \leq \varepsilon
		\left(\frac{r_n\log p_n}{n}\right)^{1/2} \textrm{ for all } j
		\right\} \notag\\
		&&\bigcap 
		\left\{
		\min_{k=1,
			\dots,q}((\mathbf{B}_n)_{jk}-(\mathbf{B}_0)_{jk}) \geq -\varepsilon
		\left(\frac{r_n\log p_n}{n}\right)^{1/2} \textrm{ for all } j
		\right\} \notag\\
		&\bigsubset&
		\left\{
		\min_{k=1,
			\dots,q}((\mathbf{B}_n)_{jk}-(\mathbf{B}_0)_{jk}) \leq \varepsilon
		\left(\frac{r_n\log p_n}{n}\right)^{1/2} \textrm{ for all } j
		\right\} \notag\\
		&&\bigcap 
		\left\{
		\max_{k=1,
			\dots,q}((\mathbf{B}_n)_{jk}-(\mathbf{B}_0)_{jk}) \geq -\varepsilon
		\left(\frac{r_n\log p_n}{n}\right)^{1/2} \textrm{ for all } j
		\right\}
	\end{eqnarray*}}
From \eqref{s20}, 
	\begin{align*}
		& E_{\mathbf{B}_0}
		P \left\{
		\min_{k=1,
			\dots,q}((\mathbf{B}_n)_{jk}-(\mathbf{B}_0)_{jk}) \leq \varepsilon
		\left(\frac{r_n\log p_n}{n}\right)^{1/2} \textrm{ for all } j
		\right\}
		\longrightarrow 1
	\end{align*}
	and 
	\begin{eqnarray} 
		&& E_{\mathbf{B}_0}
		P \left\{
		\max_{k=1,
			\dots,q}((\mathbf{B}_n)_{jk}-(\mathbf{B}_0)_{jk}) \geq -\varepsilon
		\left(\frac{r_n\log p_n}{n}\right)^{1/2} \textrm{ for all } j
		\right\}
		\longrightarrow 1.\notag\\	
 \label{s221_new} 
 \end{eqnarray}
Thus, 
{\footnotesize{
	\begin{align*}
		& E_{\mathbf{B}_0}
		P \left\{
		\min_{k=1,
			\dots,q}(q_{0.975}[(\mathbf{B}_n)_{jk}]-(\mathbf{B}_0)_{jk})\leq \varepsilon
		\left(\frac{r_n\log p_n}{n}\right)^{1/2} \textrm{ for all } j
		\right\}
		\longrightarrow 1
	\end{align*}}}
	and 
{\footnotesize{	\begin{eqnarray} 
		&& E_{\mathbf{B}_0}
		P \left\{
		\max_{k=1,
			\dots,q}(q_{0.025}[(\mathbf{B}_n)_{jk}]-(\mathbf{B}_0)_{jk}) \geq -\varepsilon
		\left(\frac{r_n\log p_n}{n}\right)^{1/2} \textrm{ for all } j
		\right\}
		\longrightarrow 1.\notag\\
\label{s221_new2} 	
 \end{eqnarray}}}
 By the condition that
	$\min_{j\in S_0}\max_{k=1,
		\dots,q}|(\mathbf{B}_0)_{jk}| 
	\geq \frac{c_3}{n^{\zeta}}$ for some  $2\zeta < \alpha$ so that
	$ \max_{k=1,
		\dots,q}|(\mathbf{B}_0)_{jk}| 
	\gg ( r_n\log p_n / n )^{1/2}$,~~$j\in S_0$ and $r_n\in (s_0,n)$, it then follows from \eqref{s221_new2} that 
 {\footnotesize
	\begin{align*} \label{s222new}
		&E_{\mathbf{B}_0}
		P \left(~j\in S_0  \mbox{~~implies that~~}\min_{k=1,
			\dots,q} q_{0.975}[(\mathbf{B}_n)_{jk}]\
		\leq \varepsilon \left(\frac{r_n\log p_n}{n}\right)^{1/2}
		\right.\\
		&\left. \qquad \qquad \mbox{or~~}\max_{k=1,
			\dots,q} q_{0.025}[(\mathbf{B}_n)_{jk}]\
		\geq -\varepsilon \left(\frac{r_n\log p_n}{n}\right)^{1/2}
		\mid \mathbf{Z}_n,~\mathbf{W}_n
		\right)\rightarrow 1. \numbereqn
	\end{align*}}
 This implies that ${\cal A}_n$ contains the true set $S_0$, where  ${\cal A}_n$ is defined as in \eqref{setAn}. On the other hand, 
we have 
    {
	\begin{eqnarray*}
		&&
		\left\{
		\max_{k=1,
			\dots,q}((\mathbf{B}_n)_{jk}-(\mathbf{B}_0)_{jk}) \leq \varepsilon
		\left(\frac{r_n\log p_n}{n}\right)^{1/2} \textrm{ for all } j
		\right\} \notag\\
		&&\bigcap 
		\left\{
		\min_{k=1,
			\dots,q}((\mathbf{B}_n)_{jk}-(\mathbf{B}_0)_{jk}) \geq -\varepsilon
		\left(\frac{r_n\log p_n}{n}\right)^{1/2} \textrm{ for all } j
		\right\} \notag\\
		& &\bigsubset
		\left\{
		-\varepsilon
		\left(\frac{r_n\log p_n}{n}\right)^{1/2}
		\leq 
		\max_{k=1,
			\dots,q} (\mathbf{B}_n)_{jk} 
		-\max_{k=1,
			\dots,q}(\mathbf{B}_0)_{jk} \right. \\
        && ~~~~~~~~~~~~~~ \left.   
		\leq 
		\varepsilon
		\left(\frac{r_n\log p_n}{n}\right)^{1/2} \textrm{ for all } j
		\right\}.
	\end{eqnarray*}}
	Based on \eqref{s20}, 
	\begin{align*} 		
		& E_{\mathbf{B}_0}
		P \left(
		-\varepsilon
		\left(\frac{r_n\log p_n}{n}\right)^{1/2}
		\leq 
		\max_{k=1,
			\dots,q} (\mathbf{B}_n)_{jk} 
		-\max_{k=1,
			\dots,q}(\mathbf{B}_0)_{jk} \right. \\
		& \left.  \qquad \qquad \qquad \leq 
		\varepsilon
		\left(\frac{r_n\log p_n}{n}\right)^{1/2} \textrm{ for all } j \mid \mathbf{Z}_n,~\mathbf{W}_n
		\right)
		\longrightarrow 1\notag
	\end{align*}
 so that 
 	\begin{align*} 		
		& E_{\mathbf{B}_0}
		P \left(
		-\varepsilon
		\left(\frac{r_n\log p_n}{n}\right)^{1/2}
		\leq 
		\max_{k=1,
			\dots,q} q_{0.5}[(\mathbf{B}_n)_{jk}] 
		-\max_{k=1,
			\dots,q}(\mathbf{B}_0)_{jk} \right. \\
		& \left.  \qquad \qquad \qquad \leq 
		\varepsilon
		\left(\frac{r_n\log p_n}{n}\right)^{1/2} \textrm{ for all } j \mid \mathbf{Z}_n,~\mathbf{W}_n
		\right)
		\longrightarrow 1.\notag
	\end{align*}}
 Since $ \max_{k=1,
		\dots,q}|(\mathbf{B}_0)_{jk}| 
	\gg ( r_n\log p_n / n )^{1/2}$,~~$j\in S_0$ and $r_n\in (s_0,n)$,
%
 {\footnotesize
	\begin{align*} \label{s222}
		&E_{\mathbf{B}_0}
		P \left(~j\in S_0  \mbox{~~implies that~~}\left|\max_{k=1,
			\dots,q} q_{0.5}[(\mathbf{B}_n)_{jk}]\right|
		> \varepsilon \left(\frac{r_n\log p_n}{n}\right)^{1/2}
		, \right.\\
		&\left. \qquad \qquad j\notin S_0
		\mbox{~~implies that~~}
		\left|\max_{k=1,
			\dots,q} q_{0.5}[(\mathbf{B}_n)_{jk}]\right|
		\leq \varepsilon \left(\frac{r_n\log p_n}{n}\right)^{1/2}
		\mid \mathbf{Z}_n,~\mathbf{W}_n
		\right)\rightarrow 1. \numbereqn
	\end{align*}}
	Recall that the set ${\cal J}_n$ collects the $K_n = \min \{ n-1, | \mathcal{A}_n | \} < n $ largest values among the elements in the set
	$\{|\max_{k=1,
		\dots,q} q_{0.5}[(\mathbf{B}_n)_{jk}]|\}^{p_n}_{j=1}$.
	Moreover, ${\cal J}_n$ satisfies
	${\cal J}_n\subset {\cal A}_n$ and $|{\cal J}_n|\geq |S_0|$. Therefore, it is implied  by \eqref{s222} that 
	\begin{eqnarray*}
		E_{\mathbf{B}_0}
		P \left(~j\in S_0 \mbox{~~implies that~~} j\in {\cal J}_n\mid \mathbf{Z}_n,~\mathbf{W}_n\right)\rightarrow 1. 
	\end{eqnarray*}
	{\picotwo A straightforward extension of the proof of part (i) in Proposition \ref{Prop:C.6} gives us the same result conditional on $\mathbf{X}_n$ and $\mathbf{Y}_n$.}
\end{proof}

\begin{proof}[Proof of Theorem \ref{thm_D2}]
	Let $\mathcal{J}_n^{c} = \{1, \ldots, p_n \} \setminus \mathcal{J}_n$. Let  $\widetilde{\mathbf{B}}^{\mathcal{J}_n}$ and $\widetilde{\mathbf{B}}^{\mathcal{J}_n^c}$ denote the submatrices of $\widetilde{\mathbf{B}}_n$ with row indices in $\mathcal{J}_n$ and $\mathcal{J}_n^c$ respectively, and define $\mathbf{B}_0^{\mathcal{J}_n}$ and $\mathbf{B}_0^{\mathcal{J}_n^c}$ analogously. We have that, with probability tending to one,
	\begin{align*}
		\lVert \widetilde{\mathbf{B}}_n - \mathbf{B}_0 \rVert_F^2 = \lVert \widetilde{\mathbf{B}}^{\mathcal{J}_n} - \mathbf{B}_0^{\mathcal{J}_n} \rVert_F^2 + \lVert \widetilde{\mathbf{B}}^{\mathcal{J}_n^c} - \mathbf{B}_0^{\mathcal{J}_n^c} \rVert_F^2 = \lVert \widetilde{\mathbf{B}}^{\mathcal{J}_n} - \mathbf{B}_0^{\mathcal{J}_n} \rVert_F^2,  
	\end{align*}
	where the final equality is because $\widetilde{\mathbf{B}}^{\mathcal{J}_n^c}$ is a zero matrix by construction and because by Theorem \ref{thm_C}, $\mathcal{J}_n^c \subset S_0^c$ as $n \rightarrow \infty$. Thus $\mathbf{B}_0^{\mathcal{J}_n^c}$ is also a zero matrix as $n \rightarrow \infty$. The result now follows from Theorem \ref{one-step_contraction_rate}.
\end{proof}

\section{MCMC Convergence Diagnostics} \label{S5:Diagnostics}

{\pico When we fit the Mt-MBSP model to the full data (i.e. all $p$ predictors) in Section \ref{simulation}, we ran the Gibbs sampling algorithm in Appendix \ref{S1:Gibbs} for 1100 iterations, discarding the first 100 samples as burnin. In this section, we verify that the Gibbs sampler had adequate convergence for the six numerical experiments from Section \ref{simulation}. Our analysis ensures that the superior performance of {\picotwo the} two-step estimator (reported in Table \ref{table:results}) was \emph{not} due to the algorithm failing to converge for the one-step estimator. This also confirms that for the two-step estimator, we had run Step 1 for enough iterations to obtain accurate estimates of the posterior quantiles in \eqref{setAn}-\eqref{setJn} \emph{before} running Step 2. 

To empirically verify convergence, we ran one replication for each of the six scenarios in Section \ref{simulation} when {\picotwo $n=200$ and $p=2000$}. These scenarios contained a mixture of continuous and discrete responses, and therefore, all of them utilized P\'{o}lya-gamma data augmentation. We ran the Gibbs sampler for 1100 iterations, discarding the first 100 samples as burnin. Using the \textsf{R} package \texttt{mcmcse} \cite{mcmcseRpackage}, we recorded the effective sample size (ESS) and the Monte Carlo standard error (MCSE) for the posterior median $q_{0.5}$, the posterior 0.025 quantile $q_{0.025}$, and the posterior 0.975 quantile $q_{0.975}$ for each of the regression coefficients in the $p \times q$ matrix $\mathbf{B}$. 

\begin{table} 
	\caption{{\picotwo Average ESS, maximum MCSE for $q_{0.5}$, maximum MCSE for $q_{0.025}$, and maximum MSCE for $q_{0.975}$ for all entries in the $p \times q$ regression coefficients matrix $\mathbf{B}$. These results are based on 1000 saved MCMC samples following a burnin period of 100 initial samples which were discarded. The six scenarios pertain to the six simulation settings described in Section \ref{simulation}.}}
  \begin{tabular}{|| l | c | c | c | c||} 
 \hline
  & ESS & MCSE($q_{0.5}$) & MCSE($q_{0.025}$) & MCSE($q_{0.975}$) \\ [0.25ex] 
 \hline
 Scenario 1 & {\picotwo 714.56} & {\picotwo 0.158} & {\picotwo 0.353} & {\picotwo 0.225} \\ 
 Scenario 2 & {\picotwo 868.28} & {\picotwo 0.092} & {\picotwo 0.126} & {\picotwo 0.162} \\
 Scenario 3 & {\picotwo 673.37} & {\picotwo 0.030} & {\picotwo 0.058} & {\picotwo 0.088} \\
 Scenario 4 & {\picotwo 551.06} & {\picotwo 0.111} & {\picotwo 0.392} & {\picotwo 0.478} \\
 Scenario 5 & {\picotwo 677.74} & {\picotwo 0.188} & {\picotwo 0.318} & {\picotwo 0.366} \\ 
 Scenario 6 & {\picotwo 748.50} & {\picotwo 0.022} & {\picotwo 0.090} & {\picotwo 0.071} \\ [0.5ex] 
 \hline
 \end{tabular} \label{tab:MCMC_diagnostics}
\end{table}

Table \ref{tab:MCMC_diagnostics} reports the average ESS and {\picotwo maximum} MCSE for the $pq$ entries in $\mathbf{B}$ based on the 1000 MCMC samples that were saved after 100 warmup iterations. Table \ref{tab:MCMC_diagnostics} shows an acceptable ESS (greater than {\picotwo 500} on average) in all simulation settings and an {\picotwo acceptable maximum} MCSE, indicating {\picotwo decent} precision of our MCMC estimates for $q_{0.5}$, $q_{0.025}$, and $q_{0.975}$. Figure \ref{fig:trace_plots} shows the MCMC traceplots for the 1000 iterations after burnin for six randomly chosen coefficients in $\mathbf{B}$ from each of the six scenarios. The traceplots in Figure \ref{fig:trace_plots} {\picotwo also} demonstrate superb mixing, suggesting that the Gibbs sampling algorithm had already converged after the first 100 warmup iterations.

Our analysis demonstrates that even when $p>n$ and $p$ is fairly large (here, {\picotwo $p=2000$}), we can still obtain a high quality approximation of the posterior for $\mathbf{B}$ by running just 1100 iterations of our Gibbs sampling algorithm in Appendix \ref{S1:Gibbs}. While our empirical results are very good, a rigorous theoretical analysis of the Gibbs sampler (e.g. its convergence rate or ergodicity) is desirable.}

\begin{figure}[t!]
\centering
\includegraphics[width=.9\textwidth]{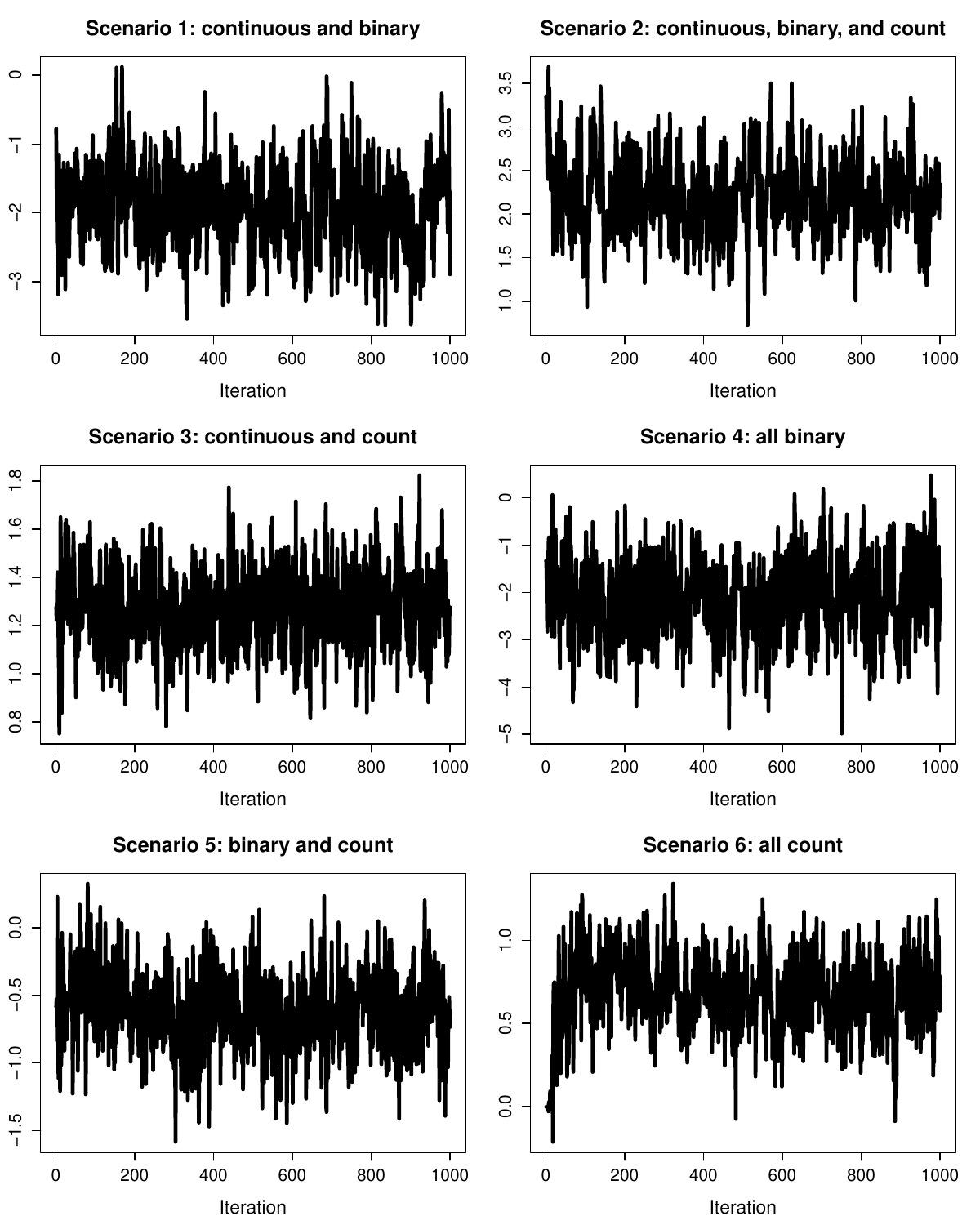} 
\caption{{\picotwo Traceplots for one regression coefficient from each of the six simulation experiments. The MCMC samples are plotted for the 1000 iterations following 100 warmup iterations.}}
\label{fig:trace_plots}
\end{figure}




\bibliographystyle{imsart-number} 
\bibliography{MBSP-Reference}       

\begin{thebibliography}{89}

\bibitem{an2017simultaneous}
\begin{barticle}[author]
\bauthor{\bsnm{An},~\bfnm{Baiguo}\binits{B.}} \AND
  \bauthor{\bsnm{Zhang},~\bfnm{Beibei}\binits{B.}}
(\byear{2017}).
\btitle{Simultaneous selection of predictors and responses for high dimensional
  multivariate linear regression}.
\bjournal{Statistics \& Probability Letters}
\bvolume{127}
\bpages{173--177}.
\end{barticle}
\endbibitem

\bibitem{ArmaganClydeDunson2011}
\begin{binproceedings}[author]
\bauthor{\bsnm{Armagan},~\bfnm{Artin}\binits{A.}},
  \bauthor{\bsnm{Clyde},~\bfnm{Merlise}\binits{M.}} \AND
  \bauthor{\bsnm{Dunson},~\bfnm{David}\binits{D.}}
(\byear{2011}).
\btitle{Generalized beta mixtures of {G}aussians}.
In \bbooktitle{Advances in Neural Information Processing Systems}
(\beditor{\bfnm{J.}\binits{J.}~\bsnm{Shawe-Taylor}},
  \beditor{\bfnm{R.}\binits{R.}~\bsnm{Zemel}},
  \beditor{\bfnm{P.}\binits{P.}~\bsnm{Bartlett}},
  \beditor{\bfnm{F.}\binits{F.}~\bsnm{Pereira}} \AND
  \beditor{\bfnm{K.~Q.}\binits{K.~Q.}~\bsnm{Weinberger}}, eds.)
\bvolume{24}
\bpages{523--531}.
\end{binproceedings}
\endbibitem

\bibitem{ArmaganDunsonLee2013}
\begin{barticle}[author]
\bauthor{\bsnm{Armagan},~\bfnm{Artin}\binits{A.}},
  \bauthor{\bsnm{Dunson},~\bfnm{David~B}\binits{D.~B.}} \AND
  \bauthor{\bsnm{Lee},~\bfnm{Jaeyong}\binits{J.}}
(\byear{2013}).
\btitle{Generalized double {P}areto shrinkage}.
\bjournal{Statistica Sinica}
\bvolume{23}
\bpages{119-143}.
\end{barticle}
\endbibitem

\bibitem{BaiGhosh2018}
\begin{barticle}[author]
\bauthor{\bsnm{Bai},~\bfnm{Ray}\binits{R.}} \AND
  \bauthor{\bsnm{Ghosh},~\bfnm{Malay}\binits{M.}}
(\byear{2018}).
\btitle{High-dimensional multivariate posterior consistency under
  global–local shrinkage priors}.
\bjournal{Journal of Multivariate Analysis}
\bvolume{167}
\bpages{157-170}.
\end{barticle}
\endbibitem

\bibitem{banerjee2022adaptive}
\begin{barticle}[author]
\bauthor{\bsnm{Banerjee},~\bfnm{Kalins}\binits{K.}},
  \bauthor{\bsnm{Chen},~\bfnm{Jun}\binits{J.}} \AND
  \bauthor{\bsnm{Zhan},~\bfnm{Xiang}\binits{X.}}
(\byear{2022}).
\btitle{Adaptive and powerful microbiome multivariate association analysis via
  feature selection}.
\bjournal{NAR Genomics and Bioinformatics}
\bvolume{4}
\bpages{lqab120}.
\end{barticle}
\endbibitem

\bibitem{BelitserGhosalvanZanten2012}
\begin{barticle}[author]
\bauthor{\bsnm{Belitser},~\bfnm{Eduard}\binits{E.}},
  \bauthor{\bsnm{Ghosal},~\bfnm{Subhashis}\binits{S.}} \AND
  \bauthor{\bparticle{van} \bsnm{Zanten},~\bfnm{Harry}\binits{H.}}
(\byear{2012}).
\btitle{Optimal two-stage procedures for estimating location and size of the
  maximum of a multivariate regression function}.
\bjournal{The Annals of Statistics}
\bvolume{40}
\bpages{2850--2876}.
\end{barticle}
\endbibitem

\bibitem{BhadraISR2020}
\begin{barticle}[author]
\bauthor{\bsnm{Bhadra},~\bfnm{Anindya}\binits{A.}},
  \bauthor{\bsnm{Datta},~\bfnm{Jyotishka}\binits{J.}},
  \bauthor{\bsnm{Li},~\bfnm{Yunfan}\binits{Y.}} \AND
  \bauthor{\bsnm{Polson},~\bfnm{Nicholas}\binits{N.}}
(\byear{2020}).
\btitle{Horseshoe regularisation for machine learning in complex and deep
  models}.
\bjournal{International Statistical Review}
\bvolume{88}
\bpages{302-320}.
\end{barticle}
\endbibitem

\bibitem{BhadraDattaPolsonWillard2017}
\begin{barticle}[author]
\bauthor{\bsnm{Bhadra},~\bfnm{Anindya}\binits{A.}},
  \bauthor{\bsnm{Datta},~\bfnm{Jyotishka}\binits{J.}},
  \bauthor{\bsnm{Polson},~\bfnm{Nicholas~G.}\binits{N.~G.}} \AND
  \bauthor{\bsnm{Willard},~\bfnm{Brandon}\binits{B.}}
(\byear{2017}).
\btitle{The horseshoe+ estimator of ultra-sparse signals}.
\bjournal{Bayesian Analysis}
\bvolume{12}
\bpages{1105 -- 1131}.
\end{barticle}
\endbibitem

\bibitem{bhattacharya2016fast}
\begin{barticle}[author]
\bauthor{\bsnm{Bhattacharya},~\bfnm{Anirban}\binits{A.}},
  \bauthor{\bsnm{Chakraborty},~\bfnm{Antik}\binits{A.}} \AND
  \bauthor{\bsnm{Mallick},~\bfnm{Bani~K}\binits{B.~K.}}
(\byear{2016}).
\btitle{Fast sampling with {G}aussian scale mixture priors in high-dimensional
  regression}.
\bjournal{Biometrika}
\bvolume{103}
\bpages{985–991}.
\end{barticle}
\endbibitem

\bibitem{BhattacharyaPatiPillaiDunson2015}
\begin{barticle}[author]
\bauthor{\bsnm{Bhattacharya},~\bfnm{Anirban}\binits{A.}},
  \bauthor{\bsnm{Pati},~\bfnm{Debdeep}\binits{D.}},
  \bauthor{\bsnm{Pillai},~\bfnm{Natesh~S.}\binits{N.~S.}} \AND
  \bauthor{\bsnm{Dunson},~\bfnm{David~B.}\binits{D.~B.}}
(\byear{2015}).
\btitle{Dirichlet–Laplace priors for optimal shrinkage}.
\bjournal{Journal of the American Statistical Association}
\bvolume{110}
\bpages{1479-1490}.
\end{barticle}
\endbibitem

\bibitem{BhattacharyaKharePal2022}
\begin{barticle}[author]
\bauthor{\bsnm{Bhattacharya},~\bfnm{Suman}\binits{S.}},
  \bauthor{\bsnm{Khare},~\bfnm{Kshitij}\binits{K.}} \AND
  \bauthor{\bsnm{Pal},~\bfnm{Subhadip}\binits{S.}}
(\byear{2022}).
\btitle{{Geometric ergodicity of Gibbs samplers for the horseshoe and its
  regularized variants}}.
\bjournal{Electronic Journal of Statistics}
\bvolume{16}
\bpages{1 -- 57}.
\end{barticle}
\endbibitem

\bibitem{BradleyBA2022}
\begin{barticle}[author]
\bauthor{\bsnm{Bradley},~\bfnm{Jonathan~R.}\binits{J.~R.}}
(\byear{2022}).
\btitle{Joint {B}ayesian analysis of multiple response-types using the
  hierarchical generalized transformation model}.
\bjournal{Bayesian Analysis}
\bvolume{17}
\bpages{127 -- 164}.
\end{barticle}
\endbibitem

\bibitem{cai2007selecting}
\begin{barticle}[author]
\bauthor{\bsnm{Cai},~\bfnm{Zhipeng}\binits{Z.}},
  \bauthor{\bsnm{Goebel},~\bfnm{Randy}\binits{R.}},
  \bauthor{\bsnm{Salavatipour},~\bfnm{Mohammad~R}\binits{M.~R.}} \AND
  \bauthor{\bsnm{Lin},~\bfnm{Guohui}\binits{G.}}
(\byear{2007}).
\btitle{Selecting dissimilar genes for multi-class classification, an
  application in cancer subtyping}.
\bjournal{BMC Bioinformatics}
\bvolume{8}
\bpages{1--15}.
\end{barticle}
\endbibitem

\bibitem{camastra2005novel}
\begin{barticle}[author]
\bauthor{\bsnm{Camastra},~\bfnm{Francesco}\binits{F.}} \AND
  \bauthor{\bsnm{Verri},~\bfnm{Alessandro}\binits{A.}}
(\byear{2005}).
\btitle{A novel kernel method for clustering}.
\bjournal{IEEE Transactions on Pattern Analysis and Machine Intelligence}
\bvolume{27}
\bpages{801--805}.
\end{barticle}
\endbibitem

\bibitem{CaoKhareGhosh2020}
\begin{barticle}[author]
\bauthor{\bsnm{Cao},~\bfnm{Xuan}\binits{X.}},
  \bauthor{\bsnm{Khare},~\bfnm{Kshitij}\binits{K.}} \AND
  \bauthor{\bsnm{Ghosh},~\bfnm{Malay}\binits{M.}}
(\byear{2020}).
\btitle{High-dimensional posterior consistency for hierarchical non-local
  priors in regression}.
\bjournal{Bayesian Analysis}
\bvolume{15}
\bpages{241 -- 262}.
\end{barticle}
\endbibitem

\bibitem{CarvahoPolsonScott2010}
\begin{barticle}[author]
\bauthor{\bsnm{Carvalho},~\bfnm{Carlos~M.}\binits{C.~M.}},
  \bauthor{\bsnm{Polson},~\bfnm{Nicholas~G.}\binits{N.~G.}} \AND
  \bauthor{\bsnm{Scott},~\bfnm{James~G.}\binits{J.~G.}}
(\byear{2010}).
\btitle{The horseshoe estimator for sparse signals}.
\bjournal{Biometrika}
\bvolume{97}
\bpages{465-480}.
\end{barticle}
\endbibitem

\bibitem{chakraborty2020bayesian}
\begin{barticle}[author]
\bauthor{\bsnm{Chakraborty},~\bfnm{Antik}\binits{A.}},
  \bauthor{\bsnm{Bhattacharya},~\bfnm{Anirban}\binits{A.}} \AND
  \bauthor{\bsnm{Mallick},~\bfnm{Bani~K}\binits{B.~K.}}
(\byear{2020}).
\btitle{Bayesian sparse multiple regression for simultaneous rank reduction and
  variable selection}.
\bjournal{Biometrika}
\bvolume{107}
\bpages{205--221}.
\end{barticle}
\endbibitem

\bibitem{ChakrabortyGhosal2021}
\begin{barticle}[author]
\bauthor{\bsnm{Chakraborty},~\bfnm{Moumita}\binits{M.}} \AND
  \bauthor{\bsnm{Ghosal},~\bfnm{Subhashis}\binits{S.}}
(\byear{2021}).
\btitle{Bayesian inference on monotone regression quantile: coverage and rate
  acceleration}.
\bjournal{preprint}.
\end{barticle}
\endbibitem

\bibitem{chen2013variable}
\begin{barticle}[author]
\bauthor{\bsnm{Chen},~\bfnm{Jun}\binits{J.}} \AND
  \bauthor{\bsnm{Li},~\bfnm{Hongzhe}\binits{H.}}
(\byear{2013}).
\btitle{Variable selection for sparse {D}irichlet-multinomial regression with
  an application to microbiome data analysis}.
\bjournal{The Annals of Applied Statistics}
\bvolume{7}
\bpages{418-442}.
\end{barticle}
\endbibitem

\bibitem{ChenHuang2012}
\begin{barticle}[author]
\bauthor{\bsnm{Chen},~\bfnm{Lisha}\binits{L.}} \AND
  \bauthor{\bsnm{Huang},~\bfnm{Jianhua~Z.}\binits{J.~Z.}}
(\byear{2012}).
\btitle{Sparse reduced-rank regression for simultaneous dimension reduction and
  variable selection}.
\bjournal{Journal of the American Statistical Association}
\bvolume{107}
\bpages{1533--1545}.
\end{barticle}
\endbibitem

\bibitem{chen2022bayesian}
\begin{barticle}[author]
\bauthor{\bsnm{Chen},~\bfnm{Xiaoyu}\binits{X.}},
  \bauthor{\bsnm{Kang},~\bfnm{Xiaoning}\binits{X.}},
  \bauthor{\bsnm{Jin},~\bfnm{Ran}\binits{R.}} \AND
  \bauthor{\bsnm{Deng},~\bfnm{Xinwei}\binits{X.}}
(\byear{2023}).
\btitle{Bayesian sparse regression for mixed multi-responses with application
  to runtime metrics prediction in fog manufacturing}.
\bjournal{Technometrics}
\bvolume{65}
\bpages{206-219}.
\end{barticle}
\endbibitem

\bibitem{ChoiHobert2013}
\begin{barticle}[author]
\bauthor{\bsnm{Choi},~\bfnm{Hee~Min}\binits{H.~M.}} \AND
  \bauthor{\bsnm{Hobert},~\bfnm{James~P.}\binits{J.~P.}}
(\byear{2013}).
\btitle{The {P}olya-{G}amma {G}ibbs sampler for {B}ayesian logistic regression
  is uniformly ergodic}.
\bjournal{Electronic Journal of Statistics}
\bvolume{7}
\bpages{2054 -- 2064}.
\end{barticle}
\endbibitem

\bibitem{ChunKeles2010}
\begin{barticle}[author]
\bauthor{\bsnm{Chun},~\bfnm{Hyonho}\binits{H.}} \AND
  \bauthor{\bsnm{Kele\c{s}},~\bfnm{S\"{u}nd\"{u}z}\binits{S.}}
(\byear{2010}).
\btitle{Sparse partial least squares regression for simultaneous dimension
  reduction and variable selection}.
\bjournal{Journal of the Royal Statistical Society Series B: Statistical
  Methodology}
\bvolume{72}
\bpages{3-25}.
\end{barticle}
\endbibitem

\bibitem{CoxWermuth1992}
\begin{barticle}[author]
\bauthor{\bsnm{Cox},~\bfnm{D.~R.}\binits{D.~R.}} \AND
  \bauthor{\bsnm{Wermuth},~\bfnm{Nanny}\binits{N.}}
(\byear{1992}).
\btitle{Response models for mixed binary and quantitative variables}.
\bjournal{Biometrika}
\bvolume{79}
\bpages{441--461}.
\end{barticle}
\endbibitem

\bibitem{DengJin2015}
\begin{barticle}[author]
\bauthor{\bsnm{Deng},~\bfnm{Xinwei}\binits{X.}} \AND
  \bauthor{\bsnm{Jin},~\bfnm{Ran}\binits{R.}}
(\byear{2015}).
\btitle{{QQ} Mmdels: Joint modeling for quantitative and qualitative quality
  responses in manufacturing systems}.
\bjournal{Technometrics}
\bvolume{57}
\bpages{320-331}.
\end{barticle}
\endbibitem

\bibitem{deshpande2019simultaneous}
\begin{barticle}[author]
\bauthor{\bsnm{Deshpande},~\bfnm{Sameer~K}\binits{S.~K.}},
  \bauthor{\bsnm{Ro{\v{c}}kov{\'a}},~\bfnm{Veronika}\binits{V.}} \AND
  \bauthor{\bsnm{George},~\bfnm{Edward~I}\binits{E.~I.}}
(\byear{2019}).
\btitle{Simultaneous variable and covariance selection with the multivariate
  spike-and-slab lasso}.
\bjournal{Journal of Computational and Graphical Statistics}
\bvolume{28}
\bpages{921--931}.
\end{barticle}
\endbibitem

\bibitem{DunsonJRSSB2000}
\begin{barticle}[author]
\bauthor{\bsnm{Dunson},~\bfnm{David~B.}\binits{D.~B.}}
(\byear{2000}).
\btitle{Bayesian latent variable models for clustered mixed outcomes}.
\bjournal{Journal of the Royal Statistical Society. Series B (Statistical
  Methodology)}
\bvolume{62}
\bpages{355--366}.
\end{barticle}
\endbibitem

\bibitem{EbiaredohSEM2022}
\begin{barticle}[author]
\bauthor{\bsnm{Ebiaredoh-Mienye},~\bfnm{S.~A.}\binits{S.~A.}},
  \bauthor{\bsnm{Swart},~\bfnm{T.~G.}\binits{T.~G.}},
  \bauthor{\bsnm{Esenogho},~\bfnm{E.}\binits{E.}} \AND
  \bauthor{\bsnm{Mienye},~\bfnm{I.~D.}\binits{I.~D.}}
(\byear{2022}).
\btitle{A machine learning method with filter-based feature selection for
  improved prediction of chronic kidney disease}.
\bjournal{Bioengineering}
\bvolume{9}
\bpages{350}.
\end{barticle}
\endbibitem

\bibitem{fan2008sure}
\begin{barticle}[author]
\bauthor{\bsnm{Fan},~\bfnm{Jianqing}\binits{J.}} \AND
  \bauthor{\bsnm{Lv},~\bfnm{Jinchi}\binits{J.}}
(\byear{2008}).
\btitle{Sure independence screening for ultrahigh dimensional feature space}.
\bjournal{Journal of the Royal Statistical Society: Series B (Statistical
  Methodology)}
\bvolume{70}
\bpages{849--911}.
\end{barticle}
\endbibitem

\bibitem{FanZhang1999}
\begin{barticle}[author]
\bauthor{\bsnm{Fan},~\bfnm{Jianqing}\binits{J.}} \AND
  \bauthor{\bsnm{Zhang},~\bfnm{Wenyang}\binits{W.}}
(\byear{1999}).
\btitle{{Statistical estimation in varying coefficient models}}.
\bjournal{The Annals of Statistics}
\bvolume{27}
\bpages{1491 -- 1518}.
\end{barticle}
\endbibitem

\bibitem{FitzmauriceJASA1995}
\begin{barticle}[author]
\bauthor{\bsnm{Fitzmaurice},~\bfnm{Garrett~M.}\binits{G.~M.}} \AND
  \bauthor{\bsnm{Laird},~\bfnm{Nan~M.}\binits{N.~M.}}
(\byear{1995}).
\btitle{Regression models for a bivariate discrete and continuous outcome with
  clustering}.
\bjournal{Journal of the American Statistical Association}
\bvolume{90}
\bpages{845--852}.
\end{barticle}
\endbibitem

\bibitem{mcmcseRpackage}
\begin{bmanual}[author]
\bauthor{\bsnm{Flegal},~\bfnm{James~M.}\binits{J.~M.}},
  \bauthor{\bsnm{Hughes},~\bfnm{John}\binits{J.}},
  \bauthor{\bsnm{Vats},~\bfnm{Dootika}\binits{D.}},
  \bauthor{\bsnm{Dai},~\bfnm{Ning}\binits{N.}},
  \bauthor{\bsnm{Gupta},~\bfnm{Kushagra}\binits{K.}} \AND
  \bauthor{\bsnm{Maji},~\bfnm{Uttiya}\binits{U.}}
(\byear{2021}).
\btitle{mcmcse: Monte Carlo Standard Errors for MCMC}
\bnote{R package version 1.5-0}.
\end{bmanual}
\endbibitem

\bibitem{GelmanHwangVehtari2014}
\begin{barticle}[author]
\bauthor{\bsnm{Gelman},~\bfnm{Andrew}\binits{A.}},
  \bauthor{\bsnm{Hwang},~\bfnm{Jessica}\binits{J.}} \AND
  \bauthor{\bsnm{Vehtari},~\bfnm{Aki}\binits{A.}}
(\byear{2014}).
\btitle{Understanding predictive information criteria for {B}ayesian models}.
\bjournal{Statistics and Computing}
\bvolume{24}
\bpages{997–1016}.
\end{barticle}
\endbibitem

\bibitem{GhoshKhareMichailidis2019}
\begin{barticle}[author]
\bauthor{\bsnm{Ghosh},~\bfnm{Satyajit}\binits{S.}},
  \bauthor{\bsnm{Khare},~\bfnm{Kshitij}\binits{K.}} \AND
  \bauthor{\bsnm{Michailidis},~\bfnm{George}\binits{G.}}
(\byear{2019}).
\btitle{High-dimensional posterior consistency in {B}ayesian vector
  autoregressive models}.
\bjournal{Journal of the American Statistical Association}
\bvolume{114}
\bpages{735-748}.
\end{barticle}
\endbibitem

\bibitem{GohDeyChen2017}
\begin{barticle}[author]
\bauthor{\bsnm{Goh},~\bfnm{Gyuhyeong}\binits{G.}},
  \bauthor{\bsnm{Dey},~\bfnm{Dipak~K.}\binits{D.~K.}} \AND
  \bauthor{\bsnm{Chen},~\bfnm{Kun}\binits{K.}}
(\byear{2017}).
\btitle{Bayesian sparse reduced rank multivariate regression}.
\bjournal{Journal of Multivariate Analysis}
\bvolume{157}
\bpages{14-28}.
\end{barticle}
\endbibitem

\bibitem{GriffinBrown2013}
\begin{barticle}[author]
\bauthor{\bsnm{Griffin},~\bfnm{Jim~E.}\binits{J.~E.}} \AND
  \bauthor{\bsnm{Brown},~\bfnm{Philip.~J.}\binits{P.~J.}}
(\byear{2013}).
\btitle{Some priors for sparse regression modelling}.
\bjournal{Bayesian Analysis}
\bvolume{8}
\bpages{691--702}.
\end{barticle}
\endbibitem

\bibitem{GriffinBiometrika2020}
\begin{barticle}[author]
\bauthor{\bsnm{Griffin},~\bfnm{J~E}\binits{J.~E.}},
  \bauthor{\bsnm{Łatuszyński},~\bfnm{K~G}\binits{K.~G.}} \AND
  \bauthor{\bsnm{Steel},~\bfnm{M~F~J}\binits{M.~F.~J.}}
(\byear{2020}).
\btitle{{In search of lost mixing time: adaptive Markov chain Monte Carlo
  schemes for Bayesian variable selection with very large p}}.
\bjournal{Biometrika}
\bvolume{108}
\bpages{53-69}.
\end{barticle}
\endbibitem

\bibitem{HeXuKang2019}
\begin{barticle}[author]
\bauthor{\bsnm{He},~\bfnm{Kevin}\binits{K.}},
  \bauthor{\bsnm{Xu},~\bfnm{Han}\binits{H.}} \AND
  \bauthor{\bsnm{Kang},~\bfnm{Jian}\binits{J.}}
(\byear{2019}).
\btitle{A selective overview of feature screening methods with applications to
  neuroimaging data}.
\bjournal{WIREs Computational Statistics}
\bvolume{11}
\bpages{e1454}.
\end{barticle}
\endbibitem

\bibitem{he2024framework}
\begin{barticle}[author]
\bauthor{\bsnm{He},~\bfnm{Qing}\binits{Q.}} \AND
  \bauthor{\bsnm{Huang},~\bfnm{Hsin-Hsiung}\binits{H.-H.}}
(\byear{2024}).
\btitle{A framework of zero-inflated bayesian negative binomial regression
  models for spatiotemporal data}.
\bjournal{Journal of Statistical Planning and Inference}
\bvolume{229}
\bpages{106098}.
\end{barticle}
\endbibitem

\bibitem{HwangPennell2014}
\begin{barticle}[author]
\bauthor{\bsnm{Hwang},~\bfnm{Beom~Seuk}\binits{B.~S.}} \AND
  \bauthor{\bsnm{Pennell},~\bfnm{Michael~L.}\binits{M.~L.}}
(\byear{2014}).
\btitle{Semiparametric {B}ayesian joint modeling of a binary and continuous
  outcome with applications in toxicological risk assessment}.
\bjournal{Statistics in Medicine}
\bvolume{33}
\bpages{1162-1175}.
\end{barticle}
\endbibitem

\bibitem{jackman2009bayesian}
\begin{bbook}[author]
\bauthor{\bsnm{Jackman},~\bfnm{Simon}\binits{S.}}
(\byear{2009}).
\btitle{Bayesian Analysis for the Social Sciences}.
\bpublisher{John Wiley \& Sons}.
\end{bbook}
\endbibitem

\bibitem{KangJQT2018}
\begin{barticle}[author]
\bauthor{\bsnm{Kang},~\bfnm{Lulu}\binits{L.}},
  \bauthor{\bsnm{Kang},~\bfnm{Xiaoning}\binits{X.}},
  \bauthor{\bsnm{Deng},~\bfnm{Xinwei}\binits{X.}} \AND
  \bauthor{\bsnm{Jin},~\bfnm{Ran}\binits{R.}}
(\byear{2018}).
\btitle{A {B}ayesian hierarchical model for quantitative and qualitative
  responses}.
\bjournal{Journal of Quality Technology}
\bvolume{50}
\bpages{290-308}.
\end{barticle}
\endbibitem

\bibitem{kang2020multivariate}
\begin{barticle}[author]
\bauthor{\bsnm{Kang},~\bfnm{Xiaoning}\binits{X.}},
  \bauthor{\bsnm{Chen},~\bfnm{Xiaoyu}\binits{X.}},
  \bauthor{\bsnm{Jin},~\bfnm{Ran}\binits{R.}},
  \bauthor{\bsnm{Wu},~\bfnm{Hao}\binits{H.}} \AND
  \bauthor{\bsnm{Deng},~\bfnm{Xinwei}\binits{X.}}
(\byear{2021}).
\btitle{Multivariate regression of mixed responses for evaluation of
  visualization designs}.
\bjournal{IISE Transactions}
\bvolume{53}
\bpages{313--325}.
\end{barticle}
\endbibitem

\bibitem{kharesu2022}
\begin{barticle}[author]
\bauthor{\bsnm{Khare},~\bfnm{Kshitij}\binits{K.}} \AND
  \bauthor{\bsnm{Su},~\bfnm{Zhihua}\binits{Z.}}
(\byear{2023}).
\btitle{Response variable selection in multivariate linear regression}.
\bjournal{Statistica Sinica (to appear)}.
\end{barticle}
\endbibitem

\bibitem{KonerWilliams2023}
\begin{barticle}[author]
\bauthor{\bsnm{Koner},~\bfnm{Salil}\binits{S.}} \AND
  \bauthor{\bsnm{Williams},~\bfnm{Jonathan~P.}\binits{J.~P.}}
(\byear{2023}).
\btitle{{The EAS approach to variable selection for multivariate response data
  in high-dimensional settings}}.
\bjournal{Electronic Journal of Statistics}
\bvolume{17}
\bpages{1947 -- 1995}.
\end{barticle}
\endbibitem

\bibitem{KuchibhotlaKolassaKuffner2022}
\begin{barticle}[author]
\bauthor{\bsnm{Kuchibhotla},~\bfnm{Arun~K.}\binits{A.~K.}},
  \bauthor{\bsnm{Kolassa},~\bfnm{John~E.}\binits{J.~E.}} \AND
  \bauthor{\bsnm{Kuffner},~\bfnm{Todd~A.}\binits{T.~A.}}
(\byear{2022}).
\btitle{Post-Selection Inference}.
\bjournal{Annual Review of Statistics and Its Application}
\bvolume{9}
\bpages{505-527}.
\end{barticle}
\endbibitem

\bibitem{KunduMitraGaskins2021}
\begin{barticle}[author]
\bauthor{\bsnm{Kundu},~\bfnm{Debamita}\binits{D.}},
  \bauthor{\bsnm{Mitra},~\bfnm{Riten}\binits{R.}} \AND
  \bauthor{\bsnm{Gaskins},~\bfnm{Jeremy~T.}\binits{J.~T.}}
(\byear{2021}).
\btitle{Bayesian variable selection for multioutcome models through shared
  shrinkage}.
\bjournal{Scandinavian Journal of Statistics}
\bvolume{48}
\bpages{295-320}.
\end{barticle}
\endbibitem

\bibitem{LahiriAOS2021}
\begin{barticle}[author]
\bauthor{\bsnm{Lahiri},~\bfnm{Soumendra~N.}\binits{S.~N.}}
(\byear{2021}).
\btitle{{Necessary and sufficient conditions for variable selection consistency
  of the LASSO in high dimensions}}.
\bjournal{The Annals of Statistics}
\bvolume{49}
\bpages{820 -- 844}.
\end{barticle}
\endbibitem

\bibitem{LiDuttaRoy2023}
\begin{barticle}[author]
\bauthor{\bsnm{Li},~\bfnm{Dongjin}\binits{D.}},
  \bauthor{\bsnm{Dutta},~\bfnm{Somak}\binits{S.}} \AND
  \bauthor{\bsnm{Roy},~\bfnm{Vivekananda}\binits{V.}}
(\byear{2023}).
\btitle{Model based screening embedded {B}ayesian variable selection for
  ultra-high dimensional settings}.
\bjournal{Journal of Computational and Graphical Statistics}
\bvolume{32}
\bpages{61-73}.
\end{barticle}
\endbibitem

\bibitem{LiDattaCraigBhadra2021}
\begin{barticle}[author]
\bauthor{\bsnm{Li},~\bfnm{Yunfan}\binits{Y.}},
  \bauthor{\bsnm{Datta},~\bfnm{Jyotishka}\binits{J.}},
  \bauthor{\bsnm{Craig},~\bfnm{Bruce~A.}\binits{B.~A.}} \AND
  \bauthor{\bsnm{Bhadra},~\bfnm{Anindya}\binits{A.}}
(\byear{2021}).
\btitle{Joint mean–covariance estimation via the horseshoe}.
\bjournal{Journal of Multivariate Analysis}
\bvolume{183}
\bpages{104716}.
\end{barticle}
\endbibitem

\bibitem{LiNanZhu2015}
\begin{barticle}[author]
\bauthor{\bsnm{Li},~\bfnm{Yanming}\binits{Y.}},
  \bauthor{\bsnm{Nan},~\bfnm{Bin}\binits{B.}} \AND
  \bauthor{\bsnm{Zhu},~\bfnm{Ji}\binits{J.}}
(\byear{2015}).
\btitle{Multivariate sparse group lasso for the multivariate multiple linear
  regression with an arbitrary group structure}.
\bjournal{Biometrics}
\bvolume{71}
\bpages{354-363}.
\end{barticle}
\endbibitem

\bibitem{LindermanNIPS2015}
\begin{binproceedings}[author]
\bauthor{\bsnm{Linderman},~\bfnm{Scott~W.}\binits{S.~W.}},
  \bauthor{\bsnm{Johnson},~\bfnm{Matthew~J.}\binits{M.~J.}} \AND
  \bauthor{\bsnm{Adams},~\bfnm{Ryan~P.}\binits{R.~P.}}
(\byear{2015}).
\btitle{Dependent multinomial models made easy: Stick-breaking with the
  {P}{\'o}lya-gamma augmentation}.
In \bbooktitle{Advances in Neural Information Processing Systems 28: Annual
  Conference on Neural Information Processing Systems 2015, December 7-12,
  2015, Montreal, Quebec, Canada}
(\beditor{\bfnm{Corinna}\binits{C.}~\bsnm{Cortes}},
  \beditor{\bfnm{Neil~D.}\binits{N.~D.}~\bsnm{Lawrence}},
  \beditor{\bfnm{Daniel~D.}\binits{D.~D.}~\bsnm{Lee}},
  \beditor{\bfnm{Masashi}\binits{M.}~\bsnm{Sugiyama}} \AND
  \beditor{\bfnm{Roman}\binits{R.}~\bsnm{Garnett}}, eds.)
\bpages{3456--3464}.
\end{binproceedings}
\endbibitem

\bibitem{LiquetMengersonPettittSutton2017}
\begin{barticle}[author]
\bauthor{\bsnm{Liquet},~\bfnm{B.}\binits{B.}},
  \bauthor{\bsnm{Mengersen},~\bfnm{K.}\binits{K.}},
  \bauthor{\bsnm{Pettitt},~\bfnm{A.~N.}\binits{A.~N.}} \AND
  \bauthor{\bsnm{Sutton},~\bfnm{M.}\binits{M.}}
(\byear{2017}).
\btitle{{B}ayesian variable selection regression of multivariate responses for
  group data}.
\bjournal{Bayesian Analysis}
\bvolume{12}
\bpages{1039--1067}.
\end{barticle}
\endbibitem

\bibitem{Matthews1975}
\begin{barticle}[author]
\bauthor{\bsnm{Matthews},~\bfnm{B.~W.}\binits{B.~W.}}
(\byear{1975}).
\btitle{Comparison of the predicted and observed secondary structure of T4
  phage lysozyme}.
\bjournal{Biochimica et Biophysica Acta (BBA) - Protein Structure}
\bvolume{405}
\bpages{442-451}.
\end{barticle}
\endbibitem

\bibitem{McCulloch2008}
\begin{barticle}[author]
\bauthor{\bsnm{McCulloch},~\bfnm{Charles}\binits{C.}}
(\byear{2008}).
\btitle{Joint modelling of mixed outcome types using latent variables}.
\bjournal{Statistical Methods in Medical Research}
\bvolume{17}
\bpages{53-73}.
\end{barticle}
\endbibitem

\bibitem{narisetty2014bayesian}
\begin{barticle}[author]
\bauthor{\bsnm{Narisetty},~\bfnm{Naveen~Naidu}\binits{N.~N.}} \AND
  \bauthor{\bsnm{He},~\bfnm{Xuming}\binits{X.}}
(\byear{2014}).
\btitle{Bayesian variable selection with shrinking and diffusing priors}.
\bjournal{The Annals of Statistics}
\bvolume{42}
\bpages{789--817}.
\end{barticle}
\endbibitem

\bibitem{neal2011mcmc}
\begin{barticle}[author]
\bauthor{\bsnm{Neal},~\bfnm{Radford~M}\binits{R.~M.}}
(\byear{2011}).
\btitle{{MCMC} using {H}amiltonian dynamics}.
\bjournal{Handbook of {M}arkov {C}hain {M}onte {C}arlo}
\bvolume{2}
\bpages{113-162}.
\end{barticle}
\endbibitem

\bibitem{ning2020bayesian}
\begin{barticle}[author]
\bauthor{\bsnm{Ning},~\bfnm{Bo}\binits{B.}},
  \bauthor{\bsnm{Jeong},~\bfnm{Seonghyun}\binits{S.}} \AND
  \bauthor{\bsnm{Ghosal},~\bfnm{Subhashis}\binits{S.}}
(\byear{2020}).
\btitle{Bayesian linear regression for multivariate responses under group
  sparsity}.
\bjournal{Bernoulli}
\bvolume{26}
\bpages{2353--2382}.
\end{barticle}
\endbibitem

\bibitem{PalKhare2014}
\begin{barticle}[author]
\bauthor{\bsnm{Pal},~\bfnm{Subhadip}\binits{S.}} \AND
  \bauthor{\bsnm{Khare},~\bfnm{Kshitij}\binits{K.}}
(\byear{2014}).
\btitle{{Geometric ergodicity for Bayesian shrinkage models}}.
\bjournal{Electronic Journal of Statistics}
\bvolume{8}
\bpages{604 -- 645}.
\end{barticle}
\endbibitem

\bibitem{PalKhareHobert2017}
\begin{barticle}[author]
\bauthor{\bsnm{Pal},~\bfnm{Subahdip}\binits{S.}},
  \bauthor{\bsnm{Khare},~\bfnm{Kshitij}\binits{K.}} \AND
  \bauthor{\bsnm{Hobert},~\bfnm{James~P.}\binits{J.~P.}}
(\byear{2017}).
\btitle{Trace class {M}arkov chains for {B}ayesian inference with generalized
  double Pareto shrinkage priors}.
\bjournal{Scandinavian Journal of Statistics}
\bvolume{44}
\bpages{307-323}.
\end{barticle}
\endbibitem

\bibitem{polson2013bayesian}
\begin{barticle}[author]
\bauthor{\bsnm{Polson},~\bfnm{Nicholas~G}\binits{N.~G.}},
  \bauthor{\bsnm{Scott},~\bfnm{James~G}\binits{J.~G.}} \AND
  \bauthor{\bsnm{Windle},~\bfnm{Jesse}\binits{J.}}
(\byear{2013}).
\btitle{Bayesian inference for logistic models using {P}{\'o}lya--Gamma latent
  variables}.
\bjournal{Journal of the American Statistical Association}
\bvolume{108}
\bpages{1339--1349}.
\end{barticle}
\endbibitem

\bibitem{ReganCatalano1999}
\begin{barticle}[author]
\bauthor{\bsnm{Regan},~\bfnm{Meredith~M.}\binits{M.~M.}} \AND
  \bauthor{\bsnm{Catalano},~\bfnm{Paul~J.}\binits{P.~J.}}
(\byear{1999}).
\btitle{Likelihood models for clustered binary and continuous outcomes:
  Application to developmental toxicology}.
\bjournal{Biometrics}
\bvolume{55}
\bpages{760-768}.
\end{barticle}
\endbibitem

\bibitem{RobertsRosenthal1998}
\begin{barticle}[author]
\bauthor{\bsnm{Roberts},~\bfnm{Gareth~O.}\binits{G.~O.}} \AND
  \bauthor{\bsnm{Rosenthal},~\bfnm{Jeffrey~S.}\binits{J.~S.}}
(\byear{1998}).
\btitle{Markov-chain Monte Carlo: Some practical implications of theoretical
  results}.
\bjournal{The Canadian Journal of Statistics}
\bvolume{26}
\bpages{5--20}.
\end{barticle}
\endbibitem

\bibitem{Rockova2018}
\begin{barticle}[author]
\bauthor{\bsnm{Ročkov{\'a}},~\bfnm{Veronika}\binits{V.}}
(\byear{2018}).
\btitle{{Bayesian estimation of sparse signals with a continuous spike-and-slab
  prior}}.
\bjournal{The Annals of Statistics}
\bvolume{46}
\bpages{401 -- 437}.
\end{barticle}
\endbibitem

\bibitem{RueINLA2009}
\begin{barticle}[author]
\bauthor{\bsnm{Rue},~\bfnm{Håvard}\binits{H.}},
  \bauthor{\bsnm{Martino},~\bfnm{Sara}\binits{S.}} \AND
  \bauthor{\bsnm{Chopin},~\bfnm{Nicolas}\binits{N.}}
(\byear{2009}).
\btitle{Approximate {B}ayesian inference for latent {G}aussian models by using
  integrated nested {L}aplace approximations}.
\bjournal{Journal of the Royal Statistical Society: Series B (Statistical
  Methodology)}
\bvolume{71}
\bpages{319-392}.
\end{barticle}
\endbibitem

\bibitem{salman2018impact}
\begin{barticle}[author]
\bauthor{\bsnm{Salman},~\bfnm{Ihsan}\binits{I.}},
  \bauthor{\bsnm{Ucan},~\bfnm{Osman~N}\binits{O.~N.}},
  \bauthor{\bsnm{Bayat},~\bfnm{Oguz}\binits{O.}} \AND
  \bauthor{\bsnm{Shaker},~\bfnm{Khalid}\binits{K.}}
(\byear{2018}).
\btitle{Impact of metaheuristic iteration on artificial neural network
  structure in medical data}.
\bjournal{Processes}
\bvolume{6}
\bpages{57}.
\end{barticle}
\endbibitem

\bibitem{Schwarz1978}
\begin{barticle}[author]
\bauthor{\bsnm{Schwarz},~\bfnm{Gideon}\binits{G.}}
(\byear{1978}).
\btitle{Estimating the dimension of a model}.
\bjournal{The Annals of Statistics}
\bvolume{6}
\bpages{461--464}.
\end{barticle}
\endbibitem

\bibitem{ShinBhattacharyaJohnson2018}
\begin{barticle}[author]
\bauthor{\bsnm{Shin},~\bfnm{Minsuk}\binits{M.}},
  \bauthor{\bsnm{Bhattacharya},~\bfnm{Anirban}\binits{A.}} \AND
  \bauthor{\bsnm{Johnson},~\bfnm{Valen~E.}\binits{V.~E.}}
(\byear{2018}).
\btitle{Scalable {B}ayesian variable selection using nonlocal prior densities
  in ultrahigh-dimensional settings}.
\bjournal{Statistica Sinica}
\bvolume{28}
\bpages{1053--1078}.
\end{barticle}
\endbibitem

\bibitem{song2023nearly}
\begin{barticle}[author]
\bauthor{\bsnm{Song},~\bfnm{Qifan}\binits{Q.}} \AND
  \bauthor{\bsnm{Liang},~\bfnm{Faming}\binits{F.}}
(\byear{2023}).
\btitle{Nearly optimal {B}ayesian shrinkage for high-dimensional regression}.
\bjournal{Science China Mathematics}
\bvolume{66}
\bpages{409--442}.
\end{barticle}
\endbibitem

\bibitem{SparksKhareGhosh}
\begin{barticle}[author]
\bauthor{\bsnm{Sparks},~\bfnm{Douglas~K.}\binits{D.~K.}},
  \bauthor{\bsnm{Khare},~\bfnm{Kshitij}\binits{K.}} \AND
  \bauthor{\bsnm{Ghosh},~\bfnm{Malay}\binits{M.}}
(\byear{2015}).
\btitle{{Necessary and sufficient conditions for high-dimensional posterior
  consistency under $g$-priors}}.
\bjournal{Bayesian Analysis}
\bvolume{10}
\bpages{627 -- 664}.
\end{barticle}
\endbibitem

\bibitem{Spiegelhalter2002}
\begin{barticle}[author]
\bauthor{\bsnm{Spiegelhalter},~\bfnm{David~J.}\binits{D.~J.}},
  \bauthor{\bsnm{Best},~\bfnm{Nicola~G.}\binits{N.~G.}},
  \bauthor{\bsnm{Carlin},~\bfnm{Bradley~P.}\binits{B.~P.}} \AND
  \bauthor{\bsnm{Van Der~Linde},~\bfnm{Angelika}\binits{A.}}
(\byear{2002}).
\btitle{Bayesian measures of model complexity and fit}.
\bjournal{Journal of the Royal Statistical Society: Series B (Statistical
  Methodology)}
\bvolume{64}
\bpages{583-639}.
\end{barticle}
\endbibitem

\bibitem{Stamey2013}
\begin{barticle}[author]
\bauthor{\bsnm{Stamey},~\bfnm{James~D.}\binits{J.~D.}},
  \bauthor{\bsnm{Natanegara},~\bfnm{Fanni}\binits{F.}} \AND
  \bauthor{\bsnm{Jr.},~\bfnm{John W.~Seaman}\binits{J.~W.~S.}}
(\byear{2013}).
\btitle{Bayesian sample size determination for a clinical trial with correlated
  continuous and binary outcomes}.
\bjournal{Journal of Biopharmaceutical Statistics}
\bvolume{23}
\bpages{790-803}.
\end{barticle}
\endbibitem

\bibitem{su2001molecular}
\begin{barticle}[author]
\bauthor{\bsnm{Su},~\bfnm{Andrew~I}\binits{A.~I.}},
  \bauthor{\bsnm{Welsh},~\bfnm{John~B}\binits{J.~B.}},
  \bauthor{\bsnm{Sapinoso},~\bfnm{Lisa~M}\binits{L.~M.}},
  \bauthor{\bsnm{Kern},~\bfnm{Suzanne~G}\binits{S.~G.}},
  \bauthor{\bsnm{Dimitrov},~\bfnm{Petre}\binits{P.}},
  \bauthor{\bsnm{Lapp},~\bfnm{Hilmar}\binits{H.}},
  \bauthor{\bsnm{Schultz},~\bfnm{Peter~G}\binits{P.~G.}},
  \bauthor{\bsnm{Powell},~\bfnm{Steven~M}\binits{S.~M.}},
  \bauthor{\bsnm{Moskaluk},~\bfnm{Christopher~A}\binits{C.~A.}},
  \bauthor{\bsnm{Frierson~Jr},~\bfnm{Henry~F}\binits{H.~F.}} \betal{et~al.}
(\byear{2001}).
\btitle{Molecular classification of human carcinomas by use of gene expression
  signatures}.
\bjournal{Cancer Research}
\bvolume{61}
\bpages{7388--7393}.
\end{barticle}
\endbibitem

\bibitem{TangBanerjeeMichailidis2011}
\begin{barticle}[author]
\bauthor{\bsnm{Tang},~\bfnm{Runlong}\binits{R.}},
  \bauthor{\bsnm{Banerjee},~\bfnm{Moulinath}\binits{M.}} \AND
  \bauthor{\bsnm{Michailidis},~\bfnm{George}\binits{G.}}
(\byear{2011}).
\btitle{{A two-stage hybrid procedure for estimating an inverse regression
  function}}.
\bjournal{The Annals of Statistics}
\bvolume{39}
\bpages{956 -- 989}.
\end{barticle}
\endbibitem

\bibitem{Tibshirani1996}
\begin{barticle}[author]
\bauthor{\bsnm{Tibshirani},~\bfnm{R.}\binits{R.}}
(\byear{1996}).
\btitle{Regression shrinkage and selection via the lasso}.
\bjournal{Journal of the Royal Statistical Society: Series B (Statistical
  Methodology)}
\bvolume{58}
\bpages{267--288}.
\end{barticle}
\endbibitem

\bibitem{vanderPasKleijnvanderVaart2014}
\begin{barticle}[author]
\bauthor{\bparticle{van~der} \bsnm{Pas},~\bfnm{S.~L.}\binits{S.~L.}},
  \bauthor{\bsnm{Kleijn},~\bfnm{B.~J.~K.}\binits{B.~J.~K.}} \AND
  \bauthor{\bparticle{van~der} \bsnm{Vaart},~\bfnm{A.~W.}\binits{A.~W.}}
(\byear{2014}).
\btitle{{The horseshoe estimator: Posterior concentration around nearly black
  vectors}}.
\bjournal{Electronic Journal of Statistics}
\bvolume{8}
\bpages{2585 -- 2618}.
\end{barticle}
\endbibitem

\bibitem{WagnerTuchler2010}
\begin{barticle}[author]
\bauthor{\bsnm{Wagner},~\bfnm{Helga}\binits{H.}} \AND
  \bauthor{\bsnm{T\"{u}chler},~\bfnm{Regina}\binits{R.}}
(\byear{2010}).
\btitle{Bayesian estimation of random effects models for multivariate responses
  of mixed data}.
\bjournal{Computational Statistics \& Data Analysis}
\bvolume{54}
\bpages{1206-1218}.
\end{barticle}
\endbibitem

\bibitem{MBSPcorrigendum}
\begin{barticle}[author]
\bauthor{\bsnm{Wang},~\bfnm{Shao-Hsuan}\binits{S.-H.}},
  \bauthor{\bsnm{Bai},~\bfnm{Ray}\binits{R.}} \AND
  \bauthor{\bsnm{Huang},~\bfnm{Hsin-Hsiung}\binits{H.-H.}}
(\byear{2022}).
\btitle{Corrigendum to ``{H}igh-dimensional multivariate posterior consistency
  under global-local shrinkage priors'' [{J}. {M}ultivariate {A}nal. 167 (2018)
  157-170]}.
\bjournal{Technical report}.
\end{barticle}
\endbibitem

\bibitem{WangRoy2018}
\begin{barticle}[author]
\bauthor{\bsnm{Wang},~\bfnm{Xin}\binits{X.}} \AND
  \bauthor{\bsnm{Roy},~\bfnm{Vivekananda}\binits{V.}}
(\byear{2018}).
\btitle{Geometric ergodicity of {P}ólya-{G}amma {G}ibbs sampler for {B}ayesian
  logistic regression with a flat prior}.
\bjournal{Electronic Journal of Statistics}
\bvolume{12}
\bpages{3295 -- 3311}.
\end{barticle}
\endbibitem

\bibitem{Watanabe2010}
\begin{barticle}[author]
\bauthor{\bsnm{Watanabe},~\bfnm{Sumio}\binits{S.}}
(\byear{2010}).
\btitle{Asymptotic equivalence of Bayes cross validation and widely applicable
  information criterion in singular learning theory}.
\bjournal{Journal of Machine Learning Research}
\bvolume{11}
\bpages{3571--3594}.
\end{barticle}
\endbibitem

\bibitem{YangWainwrightJordan2016}
\begin{barticle}[author]
\bauthor{\bsnm{Yang},~\bfnm{Yun}\binits{Y.}},
  \bauthor{\bsnm{Wainwright},~\bfnm{Martin~J.}\binits{M.~J.}} \AND
  \bauthor{\bsnm{Jordan},~\bfnm{Michael~I.}\binits{M.~I.}}
(\byear{2016}).
\btitle{On the computational complexity of high-dimensional {B}ayesian variable
  selection}.
\bjournal{The Annals of Statistics}
\bvolume{44}
\bpages{2497 -- 2532}.
\end{barticle}
\endbibitem

\bibitem{YooGhosal2019}
\begin{barticle}[author]
\bauthor{\bsnm{Yoo},~\bfnm{William~Weimin}\binits{W.~W.}} \AND
  \bauthor{\bsnm{Ghosal},~\bfnm{Subhashis}\binits{S.}}
(\byear{2019}).
\btitle{{Bayesian mode and maximum estimation and accelerated rates of
  contraction}}.
\bjournal{Bernoulli}
\bvolume{25}
\bpages{2330 -- 2358}.
\end{barticle}
\endbibitem

\bibitem{ZHANG2022104835}
\begin{barticle}[author]
\bauthor{\bsnm{Zhang},~\bfnm{Ruoyang}\binits{R.}} \AND
  \bauthor{\bsnm{Ghosh},~\bfnm{Malay}\binits{M.}}
(\byear{2022}).
\btitle{Ultra high-dimensional multivariate posterior contraction rate under
  shrinkage priors}.
\bjournal{Journal of Multivariate Analysis}
\bvolume{187}
\bpages{104835}.
\end{barticle}
\endbibitem

\bibitem{zhou2014regularized}
\begin{barticle}[author]
\bauthor{\bsnm{Zhou},~\bfnm{Hua}\binits{H.}} \AND
  \bauthor{\bsnm{Li},~\bfnm{Lexin}\binits{L.}}
(\byear{2014}).
\btitle{Regularized matrix regression}.
\bjournal{Journal of the Royal Statistical Society: Series B (Statistical
  Methodology)}
\bvolume{76}
\bpages{463--483}.
\end{barticle}
\endbibitem

\bibitem{zhou2013negative}
\begin{barticle}[author]
\bauthor{\bsnm{Zhou},~\bfnm{Mingyuan}\binits{M.}} \AND
  \bauthor{\bsnm{Carin},~\bfnm{Lawrence}\binits{L.}}
(\byear{2015}).
\btitle{Negative Binomial Process Count and Mixture Modeling}.
\bjournal{IEEE Transactions on Pattern Analysis and Machine Intelligence}
\bvolume{37}
\bpages{307-320}.
\bdoi{10.1109/TPAMI.2013.211}
\end{barticle}
\endbibitem

\bibitem{ZhouJRSSB2022}
\begin{barticle}[author]
\bauthor{\bsnm{Zhou},~\bfnm{Quan}\binits{Q.}},
  \bauthor{\bsnm{Yang},~\bfnm{Jun}\binits{J.}},
  \bauthor{\bsnm{Vats},~\bfnm{Dootika}\binits{D.}},
  \bauthor{\bsnm{Roberts},~\bfnm{Gareth~O.}\binits{G.~O.}} \AND
  \bauthor{\bsnm{Rosenthal},~\bfnm{Jeffrey~S.}\binits{J.~S.}}
(\byear{2022}).
\btitle{{Dimension-free mixing for high-dimensional Bayesian variable
  selection}}.
\bjournal{Journal of the Royal Statistical Society: Series B (Statistical
  Methodology)}
\bvolume{84}
\bpages{1751-1784}.
\end{barticle}
\endbibitem

\bibitem{ZhouPharmaceutical2006}
\begin{barticle}[author]
\bauthor{\bsnm{Zhou},~\bfnm{Yinghui}\binits{Y.}},
  \bauthor{\bsnm{Whitehead},~\bfnm{John}\binits{J.}},
  \bauthor{\bsnm{Bonvini},~\bfnm{Elisa}\binits{E.}} \AND
  \bauthor{\bsnm{Stevens},~\bfnm{John~W.}\binits{J.~W.}}
(\byear{2006}).
\btitle{Bayesian decision procedures for binary and continuous bivariate
  dose-escalation studies}.
\bjournal{Pharmaceutical Statistics}
\bvolume{5}
\bpages{125-133}.
\end{barticle}
\endbibitem

\bibitem{zhu2012}
\begin{barticle}[author]
\bauthor{\bsnm{Zhu},~\bfnm{Shenghuo}\binits{S.}}
(\byear{2012}).
\btitle{A short note on the tail bound of {W}ishart distribution}.
\bjournal{arXiv:1212.5860}.
\end{barticle}
\endbibitem

\bibitem{ZouJASA2006}
\begin{barticle}[author]
\bauthor{\bsnm{Zou},~\bfnm{Hui}\binits{H.}}
(\byear{2006}).
\btitle{The adaptive lasso and its oracle properties}.
\bjournal{Journal of the American Statistical Association}
\bvolume{101}
\bpages{1418--1429}.
\end{barticle}
\endbibitem

\end{thebibliography}

\end{document}